\documentclass{amsart} %%%% Revised version of December 30, 2013

\usepackage[centertags]{amsmath} % 
\usepackage{amsfonts,amssymb,amsthm,amscd,latexsym,euscript,bbm,stmaryrd,enumitem}
\usepackage[all]{xy}
\usepackage{hyperref} % Must be the last loaded package, since it redefines LaTeX commands

\newtheorem {theorem}{Theorem}[section]
\newtheorem {corollary}[theorem]{Corollary}
\newtheorem {proposition}[theorem]{Proposition}
\newtheorem {lemma}[theorem]{Lemma}

\theoremstyle{definition}

\newtheorem {definition}[theorem]{Definition}
\newtheorem {example}[theorem]{Example}
\newtheorem {remark}[theorem]{Remark}

\newtheorem {convention}[theorem]{Convention}
\newtheorem {notation}[theorem]{Notation}

\numberwithin{equation}{section}% 

\newcommand{\co}{\colon\thinspace}%
\newcommand{\mc}[1]{\ensuremath{\mathcal{#1}}}%
\newcommand{\toh}[1]{\ensuremath{\stackrel{#1}{\rightarrow}}}%
\newcommand{\colim}{\operatorname*{colim}}%
\newcommand{\holim}{\operatorname*{holim\,}}%
\newcommand{\hofib}{\operatorname*{hofib\,}}%
\newcommand{\hocolim}{\operatorname*{hocolim\,}}%
\newcommand{\free}{\,\_\!\_\,}%
\newcommand{\id}{\operatorname{id}}%
\newcommand{\Id}{\operatorname{Id}}%
\newcommand{\LKan}{\operatorname*{LKan}}%
\newcommand{\x}{\operatorname*{\times}}%                
\newcommand{\cd}{\ensuremath{\mathrm{cd}}}%
\newcommand{\ul}[1]{\underline{#1}}%
\newcommand{\wh}[1]{\widehat{#1}}%
\newcommand{\ev}{\ensuremath{\mathrm{ev}}}%
\newcommand{\Ev}{\ensuremath{\mathrm{Ev}}}%
\newcommand{\LEv}{\ensuremath{\mathcal{L}\mathrm{Ev}}}%
\newcommand{\cyl}{\ensuremath{\mathrm{Cyl}}}%
\newcommand{\cross}{\ensuremath{\mathrm{cr}}}%
\newcommand{\Lcross}{\ensuremath{\mathcal{L}}}%
\newcommand{\hocr}{\ensuremath{\mathrm{hocr}}}%
\newcommand{\ob}{\ensuremath{\mathrm{Ob}}}%
\newcommand{\tr}{\ensuremath{\mathrm{tr}}}%
\newcommand{\Sp}{\ensuremath{\mathrm{Sp}}}%
\newcommand{\Fun}{\ensuremath{\mathrm{Fun}}}%
\newcommand{\Sfin}{\ensuremath{\mc{S}^{\mathrm{fin}}}}%
\newcommand{\Ufin}{\ensuremath{\mc{U}^{\mathrm{fin}}}}%
\renewcommand{\hom}{\ensuremath{\mathrm{hom}}}%  % 
\newcommand{\dgrm}[1]{\ensuremath{\smash{\underset{\widetilde{\hphantom{#1}}}{#1}} \mathstrut}}%
\newcommand{\proj}{\ensuremath{\mathrm{proj}}}%
\newcommand{\fib}{\ensuremath{\mathrm{fib}}}%
\newcommand{\fibr}{\ensuremath{\mathrm{Fibr}}}%
\newcommand{\hf}{\ensuremath{\mathrm{hf}}}%
\newcommand{\diagram}[2]{ \begin{align} 
                          \xymatrix{#1} 
                          \label{#2} \end{align}}%    
\newcommand{\diagr}[1]{ \begin{equation*} \xymatrix{#1} \end{equation*}}%
\newcommand{\cotensor}[2]{{\bf hom}({#1},{#2})}%

\let\join=\star

\newcommand{\Vhom}{\mathcal{V}}%

\begin{document}

\title  [Calculus of functors and model categories II]{Calculus of functors and model categories II}
\author{Georg Biedermann}
\author{Oliver R\"ondigs}

\address{Institut f\"ur Mathematik, Universit\"at Osnabr\"uck, D-49069 Osnabr\"uck, Germany}
\email{gbiederm@uni-osnabrueck.de}

\address{Institut f\"ur Mathematik, Universit\"at Osnabr\"uck, D-49069 Osnabr\"uck, Germany}
\email{oroendig@uni-osnabrueck.de}

\begin{abstract}
This is a continuation, completion, and generalization of our previous joint work \cite{BCR:calc} with Boris Chorny.
We supply model structures and Quillen equivalences underlying Goodwillie's constructions on the homotopy level for functors between simplicial model categories satisfying mild hypotheses. 
\end{abstract}

\maketitle

\tableofcontents

\section{Introduction}%\label{sec:}

In a series of three papers \cite{Goo:calc1}, \cite{Goo:calc2}, and \cite{Goo:calc3}, Tom Goodwillie developed a method of
analyzing homotopy functors between various categories of
topological spaces and spectra, with concrete applications
towards Waldhausen's algebraic $K$-theory of spaces. 
The resulting theory has been extended to other situations
(such as categories of chain complexes and Weiss's orthogonal calculus), 
and successfully applied in
many different areas of algebraic and geometric topology, including 
chromatic stable homotopy theory, nonrealization results for 
unstable modules over the Steenrod algebra, and computations
of unstable homotopy groups of spheres. An overview addressing
several of these topics can be found in \cite{Kuhn:overview}.

The basic intuition behind Goodwillie's approach comes from 
classical calculus, which studies smooth functions via linear
or polynomial approximations. Suppose that $F$ is an endofunctor
on the category of pointed topological spaces preserving weak
homotopy equivalences, such as $\Omega \Sigma$. 
A priori, the homotopy groups of values of $F$ are not computable.
However, suppose that $F$ is {\em excisive\/} in the sense that
$F$ sends homotopy pushout squares (such as a Mayer-Vietoris square)
to homotopy pullback
squares (leading to long exact sequences of homotopy groups).
Then the homotopy groups of values of $F$ form a homology
theory. Homology theories are in principle computable for many spaces,
and can be studied via their representing spectra. Goodwillie
associates to a homotopy functor $F$ a tower of endofunctors
\begin{equation}\label{eq:goodwillie-tower} 
  F \to\dotsm\to P_{n+1}F\to P_{n}F\to\dotsm \to P_1F\to P_0F
\end{equation}
in which the functor $P_nF$ satisfies an $n$-th order excision 
and the map $F\to P_nF$ is initial up to weak equivalence among all 
maps from $F$ with $n$-excisive target. For instance, the functor $P_0F$ 
is weakly equivalent to the constant functor at $F(\ast)$, the functor 
$P_1F$ is ``the best'' excisive approximation to $F$, 
and the homotopy groups of the
homotopy fiber of $P_1F\to P_0F$
are a reduced homology theory, represented by a spectrum $\partial^1 F$.
More generally (and very surprisingly) Goodwillie determines the 
``homogeneous'' fibers of this tower in the sense that, for any $n\geq 1$, 
there is a spectrum $\partial^nF$ with $\Sigma_n$-action and a 
natural weak 
equivalence
\begin{equation}\label{intro:Dn}
  D_nF(X)=\hofib\bigl(P_nF(X)\to P_{n-1}F(X)\bigr)\simeq\Omega^{\infty}(\partial^n F\wedge X^{\wedge n})_{h\Sigma_n} 
\end{equation}
where $\Sigma_n$ acts by permutation on $X^{\wedge n}$. In
other words, the layers of the Goodwillie tower~(\ref{eq:goodwillie-tower}) 
are governed by spectra, whence the associated spectral sequence computes 
unstable homotopy groups from stable input. 
We cannot resist to point out the similarity with the formula for the $n$-th homogeneous term in Taylor series of a smooth function $f$ around zero:
\[ \frac{f^{(n)}(0)\, x^n}{n!} \]
Indeed, Goodwillie calls the $\Sigma_n$-spectrum $\partial^nF$ the $n$-th derivative of $F$ at $\ast$, and produces a corresponding derivative 
at any pointed topological space. 
The third and foundational paper \cite{Goo:calc3} contains this classification
of homogeneous functors, which is basically subsumed by a commutative diagram of equivalences of homotopy categories:

\begin{equation}\label{eq:diagram-of-equivalences-homotopy-cats}
  \xymatrix{ 
    \mathrm{Ho}
    \Bigl(\overset{{n\text{-homogeneous}}}{\scriptstyle{\text{spectrum-valued functors}}} \Bigr)
    \ar[r]^-{\hocr_n}     \ar[d]_{\Omega^{\infty}} & 
    \mathrm{Ho}
    \Bigl(\overset{{\text{symmetric multilinear}}}{\scriptstyle{\text{spectrum-valued functors}}} \Bigr)
    \ar[d]^{\Omega^{\infty}}   \\ 
    \mathrm{Ho}
    \Bigl(\overset{{n\text{-homogeneous}}}{\scriptstyle{\text{space-valued functors}}} \Bigr)
    \ar[r]^-{\hocr_n} & 
    \mathrm{Ho}
    \Bigl(\overset{{\text{symmetric multilinear}}}{\scriptstyle{\text{space-valued functors}}} \Bigr)
    }
\end{equation}

The right hand side of this diagram consists of symmetric functors in $n$ variables which are linear in each variable, on the left hand side are $n$-homogeneous functors. The horizontal functor is the {\it $n$-th homotopy cross effect}, which Goodwillie simply denotes $\cross_n$. The classification concludes via 
evaluation on the $n$-tuple $(S^0,\dotsc,S^0)$ of zero spheres,
which constitutes an equivalence from the homotopy category
of symmetric multilinear spectrum-valued homotopy functors
to the stable homotopy category of spectra with $\Sigma_n$-action.
Whereas Goodwillie's original proofs
of the basic theorems usually involve certain connectivity
assumptions, his more recent proofs in \cite{Goo:calc3} consist of
clever diagram manipulations which apply in great generality
(see also~\cite{Rezk:streamline}).
In fact, this generality is one reason ``for reworking this
whole theory in the context of closed model categories'' \cite[p.~655]{Goo:calc3}
which is the present goal. This goal has been
addressed already in \cite{Kuhn:overview}, \cite{BCR:calc},
\cite{Lurie:higher-algebra}, \cite{Barnes-Oman} (for orthogonal calculus),
and most recently in \cite{Pereira}. 

Our ``reworking'' obtained in this article provides not only
a lift of the above classification from the level of homotopy 
categories to the level of model categories, 
but also a generalization of Goodwillie's calculus to more general model
categories. In order to describe this more precisely, let $\mc{S}$ be the 
category of pointed simplicial sets. Given a pointed simplicial
model category $\mc{D}$, let $\Sp(\mc{D})$ denote the stable
model category of spectra in $\mc{D}$. Finally, let
$\mc{F}$ be the category of pointed simplicial
functors from $\mc{C}$ to $\mc{D}$. Here $\mc{C}$ is a
small full subcategory of a pointed simplicial model
category $\mc{B}$, and both $\mc{D}$ and $\mc{B}$ are required
to satisfy further, but not too restrictive, conditions, which
are given in Conventions~\ref{conv3},~\ref{conv:hf}, ~\ref{conv:n-exc}, and~\ref{conv:n-homog}. 
We explicitly describe a sequence 
\[ \mc{F}_{\mathrm{hf}} \rightarrow \dotsm \rightarrow
\mc{F}_{(n+1)\text{-}\mathrm{exc}} \rightarrow \mc{F}_{n\text{-}\mathrm{exc}} \rightarrow \dotsm \rightarrow
\mc{F}_{1\text{-}\mathrm{exc}} \rightarrow \mc{F}_{0\text{-}\mathrm{exc}}\simeq \ast \]
of left Bousfield localizations of a homotopy functor model
structure on $\mc{F}$, such that the respective fibrant
replacements are $n$-excisive approximations. This lifts Goodwillie's
tower~(\ref{eq:goodwillie-tower}). Furthermore,
right Bousfield localization supplies 
``fiber sequences''
\[ \mc{F}_{n\text{-}\mathrm{hom}} \rightarrow \mc{F}_{n\text{-}\mathrm{exc}} \rightarrow 
\mc{F}_{(n-1)\text{-}\mathrm{exc}} \]
of model structures for every $n$, such that cofibrant replacement
in $\mc{F}_{n\text{-}\mathrm{hom}}$ yields an $n$-reduced approximation. In order to 
complete the description of the layers of the tower, 
we supply model structures that promotes Goodwillie's 
diagram~(\ref{eq:diagram-of-equivalences-homotopy-cats}) 
to a commuting diagram of Quillen equivalences.
In the special case where $\mc{C}$ is the category of finite pointed 
simplicial sets, the model structure for symmetric multilinear 
functors is shown to be Quillen equivalent
to the stable model category of spectra in 
$\mc{D}$ with a $\Sigma_n$-action, thus also lifting
Goodwillie's derivative to the level of model categories.

This article applies to the target $\mc{D}=\mc{S}$, the category of all pointed simplicial sets, with the full subcategory of finite pointed simplicial sets $\mc{C}=\Sfin$ as source. Already in this case, our results extend \cite{BCR:calc}, 
which constructed homogeneous model structures only for spectrum-valued
functors. Moreover, we now cover
the important variation given by the relative setting, where one considers simplicial sets retractive over a fixed simplicial set $K$. Other possible
applications, to be investigated elsewhere, are equivariant homotopy 
theory, and homotopy theory of simplicial sheaves with respect
to suitable Grothendieck topologies. 

The category of finite pointed CW complexes does not satisfy the conditions we impose on the source category for the construction of the homotopy functor model structure. However, all finite CW complexes are both fibrant and cofibrant and, thus, all simplicial functors defined on finite CW complexes are automatically homotopy functors. In this case, the homotopy functor model structure is not needed and one can perform all the constructions in this article by using the projective model structure or the cross effect model structure directly. Shortly, finite pointed CW complexes as source category \mc{C} and all pointed topological spaces as target category \mc{D} qualify as further examples.

Jacob Lurie describes in \cite{Lurie:higher-algebra}, among many other things, the $n$-excisive part and derivatives in terms of $\infty$-categories for functors between certain $\infty$-categories. Our model structures provide an alternative approach, which should be useful for the practitioners of ``calculus'', since
it upgrades Goodwillie's results on spaces and their generalizations 
obtained in \cite{Pereira} from the
level of homotopy categories to the level of model categories.

As a guideline to the reader, we offer a 
short summary of the contents of this article, followed by
diagrams indicating the various model structures.
\begin{description}
\item[Section 2] supplies the necessary prerequisites on enriched
  category theory, in particular for functors of several variables,
  and discusses differences between the unpointed simplicial and
  the pointed simplicial case. Theorem~\ref{thm:proj-model-str.}
  provides projective model structures on functor categories as
  a starting ground.
\item[Section 3] introduces symmetric functors and the categorical
  cross effect. Theorem~\ref{thm:cross-model} supplies the cross
  effect model structure, a modification
  of the projective model structure with more cofibrations, such that
  the cross effect becomes a right Quillen functor, with Goodwillie's
  homotopy cross effect as its right derived functor 
  (Proposition~\ref{prop:cross-right-obj}).
\item[Section 4] develops homotopy functor model structures based 
  on mild hypotheses~\ref{conv:hf}, refining and extending 
  the results of \cite{DRO:enriched}. Both the projective and
  the cross effect model structure may serve as a basis, as
  explained in Theorem~\ref{thm:homotopy-model}. 
\item[Section 5] contains the description of
  the Goodwillie tower~(\ref{eq:goodwillie-tower}) on the level
  of model categories in the
  following sense: Theorem~\ref{thm:n-exc-model-str-exists} supplies, 
  for every natural number $n$, an
  $n$-excisive model structure as a suitable left Bousfield
  localization of the homotopy functor model structures
  given in Section 4. Furthermore, multilinear model structures
  are introduced in Theorem~\ref{thm:multilinear-model}, 
  in order to model Goodwillie's classification
  of homogeneous functors. Theorem~\ref{coefficient-spectra-equivalence}
  shows that, up to Quillen equivalence, 
  symmetric multilinear functors on finite pointed
  simplicial sets are just spectra with a symmetric group action.
\item[Section 6] introduces homogeneous model structures 
  in Theorem~\ref{thm:n-hom-model-exists} and completes
  the classification of homogeneous functors via
  diagram~(\ref{equivalences}). With Lemma~\ref{n-hom cof} 
  we supply a characterization of homogeneous cofibrations.
\end{description}

The plethora of model structures on $\mc{F} = \Fun(\mc{C},\mc{D})$ can be organized into the following schematic diagram of left Quillen identity functors, where an arrow pointing to the right represents left Bousfield localization (keep cofibrations, add weak equivalences), a downward arrow displays a special Quillen equivalence (add cofibrations, keep equivalences) and the single upward arrow is a right Bousfield localization (add equivalences, keep fibrations):
\diagr{ \mc{F}_{\rm proj} \ar[r]\ar[d]_{\sim}  & \mc{F}_{\rm hf} \ar[r]\ar[d]_{\sim}  & \mc{F}_{n\text{-exc}} \ar[r]\ar[d]_{\sim}   & \mc{F}_{(n-1)\text{-exc}} \ar[d]_{\sim}  \\
      \mc{F}_{\rm cr} \ar[r] & \mc{F}_{\text{hf-cr}} \ar[r] & \mc{F}_{n\text{-exc-cr}} \ar[r] & \mc{F}_{(n-1)\text{-exc-cr}} \\
                      & & \mc{F}_{n\text{-hom}} \ar[u] &   }
%In the following two diagrams, the acronyms in a line refer to model structures on the category displayed on the left. The horizontal arrows indicate left Bousfield localization and downward arrows represent left Quillen functors. 
%\diagr{ \Fun(\mc{C}_1\times\hdots\times\mc{C}_n,\mc{D}) \ar[d]_{p^*} & {\rm proj} \ar[r]\ar[d] & {\rm hf} \ar[r]\ar[d] & \text{mlt-exc} \ar[d] \\
%      \Fun(\mc{C}_1\wedge\hdots\wedge\mc{C}_n,\mc{D}) & {\rm proj} \ar[r] & \text{hf} \ar[r] & \text{ml} }

%\diagr{ \Fun(\mc{C}^{\wedge n},\mc{D}) \ar[d]_{\epsilon^*} & {\rm proj} \ar[r]\ar[d] & \text{hf} \ar[r]\ar[d] & \text{ml} \ar[d] \\
%      \Fun(\Sigma_n\wr\mc{C}^{\wedge},\mc{D}) & {\rm proj} \ar[r] & \text{hf} \ar[r] & \text{ml} }

\begin{notation}\label{spaces}
The category of simplicial sets (unpointed spaces) is denoted $\mc{U}$, and the category of pointed simplicial sets (spaces) is denoted $\mc{S}$. The corresponding full subcategories of finite (pointed) simplicial sets are denoted $\Ufin$ and $\Sfin$, respectively.
Left adjoints are always on top or to the left.
A terminal object in a category \mc{C} will be denoted $\ast_{\mc{C}}$ or simply $\ast$.
\end{notation}

{\bf Acknowledgements:} 
We would like to thank Bill Dwyer and Andr\'{e} Joyal for many encouraging and enlightning discussions. We thank Lu\'{i}s Alexandre Pereira and the referee for their detailed and helpful comments.

\section{Enriched functors}\label{sec:prel-enrich-funct}

This section recalls the necessary prerequisites on enriched
category theory, in particular for functors of several variables.
Special emphasis is given on 
the differences between the unpointed simplicial and
the pointed simplicial case in Section~\ref{sec:smash-prod-prod}.

\subsection{Preliminaries on enriched functors}
References for enriched category theory
are \cite{Bor:2}, \cite{Eilen-Kelly}, and \cite{Kelly}. 
This section mainly presents notation.

\begin{convention}\label{conv1}
Let $(\mc{V},\otimes,I)$ be a closed symmetric monoidal category.
\end{convention}

%All categories and functors are assumed to be enriched over $\mc{V}$. 
The two main examples of closed symmetric monoidal categories are mentioned
in Notation~\ref{spaces}: $(\mc{U},\times,\ast)$, the category of 
simplicial sets equipped with the categorical product, and $(\mc{S},
\wedge,S^0)$ the category of pointed simplicial sets equipped with the smash product.

\begin{notation}\label{not:functor-cat}
The $\mc{V}$-object of morphisms from $A$ to $B$ in any given
$\mc{V}$-category $\mc{C}$ is denoted $\Vhom_{\mc{C}}(A,B)$. 
The category 
of \mc{V}-functors from a small \mc{V}-category \mc{C}
to another \mc{V}-category \mc{D} is again a \mc{V}-category,
denoted by $\Fun_{\mc{V}}(\mc{C},\mc{D})$ or simply
$\Fun(\mc{C},\mc{D}) =\mc{D}^{\mc{C}}$ if no confusion may arise.
One example is the functor category $\mc{D}^G$, where $G$ is
a monoid, considered as a (discrete) 
\mc{V}-category with a single object.
\end{notation}

\begin{definition}\label{def:product-cat}
Given \mc{V}-categories $\mc{C}_i$ for $i=1,\dotsc,n$, the {\it monoidal product category} $\mc{C}_1\otimes\dotsm\otimes\mc{C}_n$ has as objects ordered $n$-tuples $(K_1,\dotsc,K_n)$ of objects $K_i$ in $\mc{C}_i$, and as $\mc{V}$-object of morphisms from $\ul{K}=(K_1,\dotsc,K_n)$ to $\ul{L}=(L_1,\dotsc,L_n)$ the
$n$-fold monoidal product
    \[ \Vhom_{\mc{C}_1\otimes\dotsm\otimes\mc{C}_n}(\ul{K},\ul{L})
        :=\bigotimes_{i=1}^n\Vhom_{\mc{C}}(K_i,L_{i}).\]
Composition and units are readily introduced, giving $\mc{C}_1\otimes\dotsm\otimes\mc{C}_n$ a $\mc{V}$-category structure. 
\end{definition}

Of course it suffices to give Definition~\ref{def:product-cat}
for two factors. The general case is presented in view of
discussing enriched functors in several variables. 

\begin{example}%\label{}
In the case where the closed symmetric monoidal base category
is $(\mc{U},\times,\ast)$, the
underlying category of a monoidal product category coincides
with the ordinary product category; an observation which could be
abbreviated as $\mc{C}\otimes_{\mc{U}}\mc{D} \cong \mc{C}\times \mc{D}$. 
This is different in the case of
$(\mc{S},\wedge,S^0)$. For example, any object of the form
$(K,\ast)$ or $(\ast,L)$ in $\mc{S}\wedge \mc{S}$ is a zero object.
\end{example}

\begin{definition}\label{some (co-)tensors}
Let \mc{C} and \mc{D} be \mc{V}-categories, with \mc{C} small. Recall that if $\mc{D}$ is tensored over $\mc{V}$, there is a \mc{V}-functor
\[ \Fun(\mc{C},\mc{V})\otimes\mc{D} \to \Fun(\mc{C},\mc{D})\]
sending $(D,\dgrm{X})$ to the \mc{V}-functor 
    \[\dgrm{X}\otimes D\co C\mapsto \dgrm{X}(C)\otimes D.\]
For fixed $D$ in \mc{D} the \mc{V}-functor $\dgrm{X}\mapsto \dgrm{X}\otimes D$ has a right adjoint $\dgrm{Y}\mapsto \dgrm{Y}^D$, where $\bigl(\dgrm{Y}^D\bigr)(C)=\Vhom_{\mc{D}}\bigl(D,\dgrm{Y}(C)\bigr)$.
For fixed $\dgrm{X}\in\Fun(\mc{C},\mc{V})$ the functor $D\mapsto \dgrm{X}\otimes D$ has a right adjoint 
\[ \dgrm{Y}\mapsto \cotensor{\dgrm{X}}{\dgrm{Y}} = 
        \int_{C}\bigl(\dgrm{Y}(C)\bigr)^{\dgrm{X}(C)} \in\mc{D}\]
if $\mc{D}$ is also cotensored over $\mc{V}$. Slightly adapting these 
definitions supplies \mc{V}-functors
\[ \Fun(\mc{C},\mc{D})\otimes\mc{V} \to \Fun(\mc{C},\mc{D})\]
and 
\[ \mc{V}^{\mathrm{op}}\otimes\Fun(\mc{C},\mc{D}) \to \Fun(\mc{C},\mc{D})\]
giving the functor category $\Fun(\mc{C},\mc{D})$ the usual structure of a tensored and cotensored \mc{V}-category. 
\end{definition}

\begin{notation}\label{not:representable}
The covariant \mc{V}-functor represented by the object
$C\in \mc{C}$ is denoted
    \[ R^C=R^C_{\mc{C}}=\mc{V}_{\mc{C}}(C,\free)\co\mc{C}\to\mc{V}.\]
\end{notation}
 
\begin{lemma}[Yoneda]\label{twisted Yoneda}
Let $C$ be an object in \mc{C} and \dgrm{Y} in $\Fun(\mc{C},\mc{D})$.
The $\mc{V}$-natural transformation
\[ \Bigl\lbrace\dgrm{Y}(C) \to \dgrm{Y}(D)^{R^C(D)}\Bigr\rbrace_{D\in \mc{C}} \] 
induces an isomorphism:
   \[ \dgrm{Y}(C) \cong\cotensor{R^C}{\dgrm{Y}}.\] 
\end{lemma}

\begin{definition}\label{objectwise tensor product}
  Let \mc{C} be a small \mc{V}-category.
  The {\it objectwise tensor product} of $\dgrm{X}$ and $\dgrm{Y}$ in 
  $\Fun(\mc{C},\mc{V})$, denoted by $\dgrm{X}\otimes\dgrm{Y}$, is 
  given by the equation
  \[ (\dgrm{X}\otimes\dgrm{Y})(C):=\dgrm{X}(C)\otimes\dgrm{Y}(C).\]
\end{definition}

\subsection{The projective model structure}%\label{}
For terminology concerning model categories, consider Hirschhorn \cite{Hir:loc} or Hovey \cite{Hov:model}. A (co)fibration which is also a weak equivalence will be called an {\em acyclic\/} (co)fibration. 

\begin{convention}\label{conv:monoidal-model-cat}
  Let $\mc{V}$ be a symmetric monoidal model category.
\end{convention}

Again the main examples are (pointed) simplicial sets,
with monomorphisms as cofibrations and maps inducing homotopy
equivalences after geometric realization as weak equivalences.
As explained in the references mentioned, a \mc{V}-model category 
is tensored and cotensored over \mc{V}, and the compatibility of the 
model structures with the enrichment is expressed using the following 
definition.

\begin{definition}\label{def:pushout product}
  Let $f\co A\to B$ and $g\co C\to D$ be two maps in \mc{V}. The map
  \[ f\,\square\, g\co (A\otimes D)\cup_{(A\otimes C)}B\otimes C\to B\otimes D \]
  induced by $f\otimes D$ and $B\otimes g$ is called 
  the {\it pushout product} of $f$ and $g$. The analogous construction 
  where $g$ is a map in a tensored \mc{V}-category \mc{D} with pushouts
  yields a map $f\,\square\, g$ is a map in \mc{D}. 
\end{definition}

Hirschhorn calls it the {\it pushout corner map} in \cite[9.3.5.(2)]{Hir:loc}.
In order to equip the category $\Fun(\mc{C},\mc{D})$ of \mc{V}-functors from a
small \mc{V}-category \mc{C} to a \mc{V}-model category \mc{D} with the projective 
model structure (whose weak equivalences and fibrations are defined objectwise),
certain assumptions on \mc{D} are necessary.

\begin{definition}\label{V-monoid-axiom}
Let $\mc{V}$ be a monoidal model category and let
\mc{D} be a $\mc{V}$-model category. The \mc{V}-model category \mc{D} {\it satisfies the $\mc{V}$-monoid axiom} if the following property holds: Let $\mathrm{acof}_\mc{D}$ be the class of acyclic cofibrations in  $\mc{D}$. Let $\mathcal{E}_{\mc{D}}$ be the class of relative cell complexes in \mc{D} generated by the class of morphisms
 \[ \{j\otimes A\,|\,j \in\mathrm{acof}_\mc{D},A\in\ob\mc{V}\}. \]
Then every morphism in $\mathcal{E}_{\mc{D}}$ is a weak equivalence.
\end{definition}

\begin{definition}\label{V-left-proper}
Let $\mc{V}$ be a monoidal model category, and let
\mc{D} be a $\mc{V}$-model category. Let $\mathcal{F}_\mc{D}$ be the class 
of relative cell complexes in \mc{D} generated by the class of morphisms
 \[ \{i\otimes A\,\vert \,i \in\mathrm{cof}_\mc{D},A\in\ob\mc{V}\}. \]
The \mc{V}-model category $\mc{D}$ is \mc{V}-{\em left proper\/}
if weak equivalences in $\mc{D}$ are closed under cobase change
along morphisms in $\mathcal{F}_\mc{D}$.
\end{definition}

\begin{remark}%\label{}
If all objects in \mc{V} are cofibrant, the \mc{V}-monoid axiom holds
automatically in any \mc{V}-model category \mc{D}. Furthermore, in that
case, \mc{V}-left properness is equivalent to left properness. This holds in particular for the cases $\mc{V}=\mc{S}$ or \mc{U}.
\end{remark}

\begin{theorem}\label{thm:proj-model-str.}
Let \mc{D} be a bicomplete \mc{V}-model category which is cofibrantly generated.
If the $\mc{V}$-monoid axiom holds in \mc{D}, the category 
\[ \Fun_\mc{V}(\mc{C},\mc{D}) \]
of $\mc{V}$-functors from a small \mc{V}-category \mc{C} to \mc{D} 
carries a cofibrantly generated model structure, where the weak
equivalences and fibrations are defined objectwise. If the model structure on \mc{D} is right proper, so is the projective model structure. If the model structure on \mc{D} is \mc{V}-left proper, the projective model structure is left proper.
\end{theorem}

\begin{proof} This follows by adapting the proof of \cite[Theorem 4.4]{DRO:enriched}. Left properness is shown as in the proof of \cite[Cor.~4.8]{DRO:enriched}. 
  Following standard terminology, this model structure will be referred to as the {\em projective\/} model structure. 
For future reference, generating sets for 
cofibrations and acyclic cofibrations in the
projective model structure are constructed as follows: 
Tensoring the functor $R^K$ with an object $E\in \ob(\mc{D})$
yields a \mc{V}-functor 
\[ R^K\otimes E \colon \mc{C}\to \mc{D},\quad 
  L\mapsto \mc{V}_\mc{C}(K,L)\otimes E.\]
The \mc{V}-Yoneda lemma~\ref{twisted Yoneda} implies that any \mc{V}-functor 
$\dgrm{X}\colon \mc{C}\to \mc{D}$ is naturally isomorphic to
the coend
\[ \int_{K\in \mc{C}} R^K\otimes \dgrm{X}(K).\]
Given generating sets $I_\mc{D}$ and
$J_\mc{D}$ for the model structure on $\mc{D}$, the sets 
\begin{align}\begin{split}
    I^{\mathrm{proj}}_{\Fun_\mc{V}(\mc{C},\mc{D})}&:=\{R^K\otimes i\,|\,K\in\ob(\mc{C}),i\in I_\mc{D}\}  \\
    J^{\mathrm{proj}}_{\Fun_\mc{V}(\mc{C},\mc{D})}&:=\{R^K\otimes j\,|\,K\in\ob(\mc{C}),j\in J_\mc{D}\} 
\end{split}
\end{align}
are generating (acyclic) cofibrations for the projective model
structure.
\end{proof}

\subsection{Enriched functors in several variables}

A \mc{V}-functor in several variables is simply a
\mc{V}-functor
\[ \mc{C}_1\otimes\dotsm\otimes\mc{C}_n \to \mc{D}\]
where \mc{D} and $\mc{C}_i$ for $i=1,\dotsc,n$ are 
\mc{V}-categories. In order to translate between
\mc{V}-functors in several variables and in a single
variable, let
$\mc{I}_{\mc{V}}$ denote
the \mc{V}-category given by the full subcategory of \mc{V} containing as its single object the unit $I$. In other words, it is a \mc{V}-category with one object, also denoted $I$, and endomorphism object $I$. 
For every \mc{V}-category \mc{A} there are canonical unit isomorphisms
  $$ \mc{I}_{\mc{V}}\otimes \mc{A} \stackrel{\cong}{\longleftarrow} \mc{A}\stackrel{\cong}{\longrightarrow}\mc{A}\otimes\mc{I}_{\mc{V}} $$
of \mc{V}-categories.

For any object $B$ in a \mc{V}-category \mc{B}, there is a \mc{V}-functor
  \[ i_B\co\mc{A}\otimes\mc{I}_{\mc{V}}\to\mc{A}\otimes\mc{B} \]
which is given on objects by $(A,I)\mapsto (A,B)$ and on morphisms by 
  \[\mc{V}_{\mc{A}}(A_1,A_2)\otimes I\to\mc{V}_{\mc{A}}(A_1,A_2)\otimes\mc{V}_{\mc{B}}(B,B),\]
where $I\to\mc{V}_{\mc{B}}(B,B)$ is the canonical unit map.
Of course, there is an analogous functor $i_A\co\mc{I}_{\mc{V}}\otimes \mc{B}\to\mc{A}\otimes\mc{B}$ for every $A$ in \mc{A}.
Given a \mc{V}-functor $G\co\mc{A}\otimes\mc{B}\to\mc{D}$, every object $B$ 
in \mc{B} defines the {\it partial functor} $G_B\co\mc{A}\to\mc{D}$ by composition:
\diagr{ \mc{A} \ar[r]^-{\gamma_{\mc{A}}}_-{\cong} & \mc{A}\otimes \mc{I}_{\mc{V}} \ar[r]^-{i_B} & \mc{A}\otimes\mc{B} \ar[r]^-{G} & \mc{D} }
The functor $G$ is uniquely determined by all its partial functors $G_A$ and $G_B$: 

\begin{proposition}[Prop. 4.2, \cite{Eilen-Kelly}] 
\label{EilKel} 
Suppose that, for all objects $A$ of $\mc{A}$ and $B$ of \mc{B}, there are
\mc{V}-functors $G_A\co\mc{B}\to\mc{D}$ and $G_B\co\mc{A}\to\mc{D}$ with the property $G_A(B)=G_B(A)=:G(A,B)$. Then there exists a unique \mc{V}-functor $G\co\mc{A}\otimes\mc{B}\to\mc{D}$ with $\{G_A\}$ and $\{G_B\}$ as partial functors if and only if the following diagram commutes:
\diagr{ \mc{V}_{\mc{A}}(A,A')\otimes\mc{V}_{\mc{B}}(B,B') \ar[rr]^-{G_{B'}\otimes G_A} \ar[dd]_-{\mathrm{switch}} && \mc{V}_{\mc{D}}\bigl(G(A,B'),G(A',B')\!\bigr)\otimes\mc{V}_{\mc{D}}\bigl(G(A,B),G(A,B')\!\bigr) \ar[d]^-{\mathrm{composition}} \\
    && \mc{V}_{\mc{C}}\bigl(G(A,B),G(A',B')\bigr) \\
        \mc{V}_{\mc{B}}(B,B')\otimes\mc{V}_{\mc{A}}(A,A')\ar[rr]^-{G_{A'}\otimes G_{B}} && \mc{V}_{\mc{D}}\bigl(G(A',B),G(A',B')\!\bigr)\otimes\mc{V}_{\mc{D}}\bigl(G(A,B),G(A',B)\!\bigr) \ar[u]_-{\mathrm{composition}}  }
\end{proposition}

In other words, a \mc{V}-functor from a monoidal product category is
essentially a functor in $n$ variables which is componentwise enriched over \mc{V}. The analogous result for \mc{V}-natural transformations
will be used as well.

\begin{proposition}[Prop. 4.12, \cite{Eilen-Kelly}]\label{EiKel2}
  Let \mc{V} be a symmetric monoidal category and $T,S\co\mc{A}\otimes\mc{B}\to\mc{D}$ be two \mc{V}-functors. For all objects $A$ in \mc{A} and $B$ in \mc{B} let
  \[ \alpha_{A,B}\co S(A,B)\to T(A,B) \]
  be a map in the underlying category of \mc{D}. The maps $\alpha_{A,B}$ are the components of a \mc{V}-natural transformation $\alpha\co S\to T$ if and only if, for each $A$, the map $\alpha_{A,B}$ is the $B$-component of a \mc{V}-natural transformation $\alpha_A\co S_A\to T_A$ and, for each $B$, the map $\alpha_{A,B}$ is the $A$-component of a \mc{V}-natural transformation $\alpha_B\co S_B\to T_B$.
\end{proposition}

Recall that a terminal object in a category \mc{C} is denoted $\ast_{\mc{C}}$ or simply $\ast$.

\begin{definition}\label{def:(multi-)red}
Suppose that the categories \mc{C} and \mc{D} admit a terminal object.
A functor $F\co\mc{C}\to\mc{D}$ is called {\it reduced} if $F(\ast_{\mc{C}})\cong\ast_{\mc{D}}$.
A functor $F$ to $\mc{D}$ in $n$ variables is called {\it multireduced} if
    $$ F(K_1,\dotsc,K_n)\cong \ast_{\mc{D}} $$
whenever $K_i$ is a terminal object for at least one $i\in \{1,\dotsc,n\}$.
\end{definition}

\begin{remark}\label{rem:multireduced}
Suppose that \mc{V} and the categories $\mc{C}_i$ are pointed categories.
Then every representable \mc{V}-functor $\mc{C}_1\otimes\dotsm\otimes \mc{C}_n\to \mc{V}$ is multireduced. Hence if \mc{D} is a cocomplete \mc{V}-category,
every object in $\Fun(\mc{C}_1\otimes\dotsm\otimes\mc{C}_n,\mc{D})$ is multireduced as a colimit of representable functors by the \mc{V}-Yoneda lemma~\ref{twisted Yoneda}.
\end{remark}

\subsection{Smash product and product categories}
\label{sec:smash-prod-prod}
This section discusses monoidal product categories in the special cases
of unpointed and pointed simplicial sets. Since the functor
$u\co\mc{S}\to\mc{U}$
forgetting the base point is lax symmetric monoidal, every
$\mc{S}$-category $\mc{C}$ has  an associated $\mc{U}$-category $u\mc{C}$
by simply forgetting base points in all morphism objects.

\begin{lemma}\label{lem:su}
  Let $\mc{C}$ and $\mc{D}$ be $\mc{S}$-categories. Then
  $u\mc{C}\times u\mc{D}$ is an $\mc{S}$-category in a natural way. 
\end{lemma}

\begin{proof}
  Given objects $\ul{K}=(K_1,K_2),\ul{L}=(L_1,L_2),\ul{M}=(M_1,M_2)\in 
  u\mc{C}\times u\mc{D}$, the 
  simplicial set 
  \[ \mc{U}_{u\mc{C}\times u\mc{D}}\bigl((K_1,K_2),(L_1,L_2)\bigr) =
  u\mc{S}_{\mc{C}}(K_1,L_1)\times u\mc{S}_{\mc{D}}(K_2,L_2) \]
  is naturally a pointed simplicial set.
  The \mc{S}-composition is induced by the \mc{U}-composition map
  \[ \mc{S}(\ul{L},\ul{M})\times \mc{S}(\ul{K},\ul{L}) \rightarrow 
  \mc{S}(\ul{K},\ul{L}) \]
  since if $\ul{f} =\ast$ or $\ul{g}=\ast$, then
  $f_i\circ g_i =\ast$ for all $1\leq i \leq n$.
  The unit 
  $$S^0\to u\mc{S}_{u\mc{C}\times u\mc{D}}(\ul{K},\ul{K})$$ 
  is induced by the diagonal.
  Associativity and unitality of the $\mc{U}$-composition
  imply associativity and unitality for the $\mc{S}$-composition.
\end{proof}

\begin{notation}\label{not:su}
  Let $\mc{C}$ and $\mc{D}$ be $\mc{S}$-categories. The 
  $\mc{S}$-category from Lemma~\ref{lem:su} is denoted
  $\mc{C}\times \mc{D}$. 
\end{notation}

\begin{lemma}\label{lem:bitensored}
  Let $\mc{C}$ and $\mc{D}$ be $\mc{S}$-model categories. Then 
  $\mc{C}\times \mc{D}$ is an $\mc{S}$-model category with
  the componentwise model structure.
\end{lemma}

\begin{proof}
  The category underlying the $\mc{S}$-category $\mc{C}\times \mc{D}$
  is simply the product category. Hence, the existence
  of the componentwise model structure follows from \cite[Ex. 1.1.6.]{Hov:model}.
  It remains to prove that $\mc{C}\times \mc{D}$ is tensored
  and cotensored over $\mc{S}$, and to verify the pushout product
  axiom. Tensor and cotensor are defined componentwise.
  It is straightforward to check that they constitute $\mc{S}$-functors
  which are part of $\mc{S}$-adjunctions. The pushout product
  axiom follows immediately.
\end{proof}

\begin{definition}\label{def:functor-p}
  Let $\mc{C}_1,\dotsc,\mc{C}_n$ be $\mc{S}$-categories. 
  The canonical functor
  \[ p\co\mc{C}_1\times\dotsm\times\mc{C}_n\to\mc{C}_1\wedge\dotsm\wedge\mc{C}_n\]
  being the identity on objects and the quotient map from the Cartesian product to the smash product on morphisms is an $\mc{S}$-functor. 
  If each $\mc{C}_i$ is small and $\mc{D}$ is another
  $\mc{S}$-category, $p$ induces an \mc{S}-adjoint pair
  \[ p_{*}\co\Fun(\mc{C}_1\times\dotsm\times\mc{C}_n,\mc{D})\,\rightleftarrows\,\Fun(\mc{C}_1\wedge\dotsm\wedge\mc{C}_n,\mc{D}):\!p^*. \]
\end{definition}

\begin{lemma}\label{lem:p-quillen}
  Let $\mc{D}$ be an $\mc{S}$-model category.
  The adjoint pair $(p_{*},p^*)$ is a Quillen pair of
  projective model structures. The functor $p^*$ preserves and 
  detects objectwise weak equivalences and objectwise fibrations.
\end{lemma}

\begin{proof}
  This is immediate, since $p^\ast$ is precomposition with
  a functor being the identity on objects.
\end{proof}

For the sake of brevity, notational differences
between an $\mc{S}$-category and its associated $\mc{U}$-category
may be ignored in the following discussion.
We denote by $\{\ast\}=\mc{I}_{\mc{U}}$ the unpointed simplicial category with one object and no non-identity morphisms and by $\{S^0\}=\mc{I}_{\mc{S}}$ the corresponding \mc{S}-category.
There is exactly one \mc{U}-functor 
\[N\co\{\ast\}\to \{S^0\},\]
and it is given on underlying simplicial sets by sending $\ast$ to the 
non-basepoint. Let $\mc{A}$ and $\mc{B}$ be  
$\mc{S}$-categories (not necessarily containing a zero object). 
There is a canonical isomorphism of unpointed simplicial categories
\[ \pi_{\mc{A}}\co\mc{A}\stackrel{\cong}{\longrightarrow}\mc{A}\times\{\ast\} \]
and an analogous one with entries switched. 
The functors $N$ and $\pi_{\mc{A}}$ are unpointed simplicial but not pointed simplicial.
In particular, the functor
\diagr{ J\co\mc{A}\times\{\ast\} \ar[r]^-{(\id,N)} & \mc{A}\times\{S^0\} \ar[r]^{p} & \mc{A}\wedge\{S^0\}}
is unpointed simplicial.
For any $B$ in \mc{B} we obtain an unpointed simplicial functor
   $$ i_B\co\mc{A}\times\{\ast\}\to\mc{A}\times\mc{B} $$
given on objects by $(A,\ast)\mapsto (A,B)$ and on morphisms by $(f,\ast)\to(f,\id_B)$. Again, there are also functors $i_A$.
We hope no confusion with the analogous definition in the pointed case will arise from the indiscriminate notation.

For every object $B$ in \mc{B} the following diagram commutes
\diagram{ \mc{A} \ar@{=}[d]\ar[r]^-{\pi_{\mc{A}}}_-{\cong} & \mc{A}\times\{\ast\} \ar[d]^{J} \ar[r]^-{i_B} & \mc{A}\times\mc{B} \ar[d]^p\ar[r]^-{F} & \mc{D}\\
        \mc{A} \ar[r]^-{\gamma_{\mc{A}}}_-{\cong} & \mc{A}\wedge \{S^0\} \ar[r]^-{i_B} & \mc{A}\wedge\mc{B} \ar[r]^-{G} & \mc{D} }{(un)pointed-partial-functors}
where the upper row consists of $\mc{U}$-functors and the lower row consists of $\mc{S}$-functors. Obviously, there is an analogous commutative diagram for every object $A$ in \mc{A}.
Given a $\mc{U}$-functor $F$ and an object $B$ as above we define the {\it partial functor} $F_B$ by composing the upper row. 
Given an $\mc{S}$-functor $G$ and $B$ as above we define the {\it partial functor} $G_B$ by composing the lower row. 
Similarly, we define partial functors $F_A$ and $G_A$ for every $A$ in \mc{A}. 
The functor $F$ in (\ref{(un)pointed-partial-functors}) is $\mc{S}$-enriched if and only if $F(\ast_{\mc{A}},\ast_{\mc{B}})\cong\ast_{\mc{D}}$. The latter does not imply that $F$ is multireduced.

\begin{lemma}%\label{}
Let \mc{A}, \mc{B} and \mc{D} be \mc{S}-categories.
For a $\mc{U}$-functor $F\co\mc{A}\times\mc{B}\to\mc{D}$,
the following are equivalent:
\begin{enumerate}
   \item
The functor $F$ is multireduced.
   \item
All partial functors of $F$ are reduced.
   \item
The functor $F$ is isomorphic to $p^*G$ for some \mc{S}-functor $G\co\mc{A}\wedge\mc{B}\to\mc{D}$.
\end{enumerate}
If these conditions hold, $F$ is in particular an \mc{S}-functor.
\end{lemma}

\begin{proof}
Since all \mc{S}-functors $G\co\mc{A}\wedge\mc{B}\to\mc{D}$ are multireduced and $p^*$ is the identity on objects, (3) implies (1).
Obviously, (1) implies (2). Now, assume (2). Then, for all objects $A$ in \mc{A} and $B$ in \mc{B} the partial functors $F_A\co\mc{B}\to\mc{D}$ and $F_B\co\mc{A}\to\mc{D}$ are \mc{S}-functors. The partial functors $G_A:=F_A$ and $G_B:=F_B$ assemble by Proposition~\ref{EilKel} to an \mc{S}-functor $G\co\mc{A}\wedge\mc{B}\to\mc{D}$. Because the diagram (\ref{(un)pointed-partial-functors}) commutes, there are canonical isomorphisms
   \[ (p^*G)_A\cong G_A=F_A\quad \mathrm{and} \quad (p^*G)_B\cong G_B=F_B\]
for all $A$ and $B$. This gives (3).
\end{proof}

\section{Symmetric functors}\label{sec:symm-funct-cross}

A major tool in Goodwillie's theory is the cross effect,
a functorial construction
which measures to which extent a given homotopy functor deviates from being
excisive.
The purpose of this section is to interpret Goodwillie's cross
effect construction as a Quillen functor between appropriate model
categories.

\subsection{Symmetric functors}\label{sec:symmetric-functors}
\begin{definition}\label{def:symmetric}
  Let \mc{C} and \mc{D} be $\mc{V}$-categories.
  A \mc{V}-functor $\dgrm{X}\co \mc{C}\otimes \dotsm\otimes \mc{C}\to \mc{D}$ 
  in $n$ variables with values in $\mc{D}$ is {\it symmetric\/} if 
  it is equipped with a \mc{V}-natural isomorphism
  \[ \sigma_{\dgrm{X}}\co \dgrm{X}(K_1,\dotsc,K_n)\cong \dgrm{X}(K_{\sigma(1)},\dotsc,K_{\sigma(n)}) \]
  for every $\sigma\in\Sigma_n$, such that the equalities
  $\id_{\dgrm{X}}=\id$ and 
  $(\tau\sigma)_{\dgrm{X}}=\tau_{\dgrm{X}}\sigma_{\dgrm{X}}$ hold. 
  A {\it symmetric \mc{V}-natural transformation} is a \mc{V}-natural transformation 
  between two symmetric functors that respects the symmetry in the obvious way.
  Let $\Fun_{\mathrm{sym}}(\mc{C}^{\otimes n},\mc{D})$ denote the corresponding \mc{V}-category.
\end{definition}

The standard example of a symmetric \mc{V}-functor is the
$n$-fold monoidal product
\[ \bigotimes_{i=1}^n \co \mc{V}\otimes \dotsm \otimes \mc{V}\to \mc{V}.\]
The main example of interest here are cross effect
functors, to be described in the next section.
In order to introduce model structures for symmetric functors,
it will be convenient to describe them as a genuine functor category
instead of just a proper subcategory of a functor category.

\begin{convention}\label{conv2}
  Suppose that the closed symmetric monoidal category $(\mc{V},\otimes,I)$ 
  has finite coproducts, denoted as $\vee$.
\end{convention}

\begin{definition}\label{def:wreath product category}
  Let \mc{C} be a \mc{V}-category. 
  The {\em wreath product category\/} $\Sigma_n\wr\mc{C}^{\otimes n}$ 
  has as its objects
  $n$-tuples $(K_1,\dotsc,K_n)$ of objects in $\mc{C}$. 
  The morphisms from $\ul{K}=(K_1,\dotsc,K_n)$ to $\ul{L}=(L_1,\dotsc,L_n)$ 
  are given by
  \[ \Vhom_{\Sigma_n\wr\mc{C}^{\otimes n}}(\ul{K},\ul{L}):=
  \bigvee_{\sigma\in\Sigma_n}\bigotimes_{i=1}^n\Vhom_{\mc{C}}(K_i,L_{\sigma^{-1}(i)}).\]
  Composition is defined as it is in the wreath product of groups or, more generally, in a semi-direct product by the following formula:
  \begin{equation}\label{wreath comp.}
    \bigl(\tau,(g_1\otimes \dotsm \otimes g_n)\bigr)\circ
    \bigl(\sigma,(f_1\otimes \dotsm \otimes f_n)\bigr) = 
    \bigl(\sigma\tau, (g_{\sigma^{-1}(1)}f_1\otimes \dotsm 
    \otimes g_{\sigma^{-1}(n)}f_n)\bigr) 
  \end{equation}
  More formally, composition is a map as follows:
  \[ \Vhom_{\Sigma_n\wr\mc{C}^{\otimes n}}\bigl(\ul{L},\ul{M}\bigr)
  \otimes\Vhom_{\Sigma_n\wr\mc{C}^{\otimes n}}\bigl(\ul{K},\ul{L}\bigr) 
  \to \Vhom_{\Sigma_n\wr\mc{C}^{\otimes n}}\bigl(\ul{K},\ul{M})\bigr).\]  
  Observe that the source of the composition map is canonically
  isomorphic to the term
  \[ \bigvee_{(\tau,\sigma)\in\Sigma_n\times\Sigma_n}\left[ \bigotimes_{j=1}^n\Vhom_{\mc{C}}(L_j,M_{\tau^{-1}(j)})\otimes
    \bigotimes_{i=1}^n\Vhom_{\mc{C}}(K_i,L_{\sigma^{-1}(i)}) \right]\]
  while the target is given by
  \[ \bigvee_{\omega\in\Sigma_n}\bigotimes_{k=1}^n\Vhom_{\mc{C}}(K_k,M_{\omega^{-1}(k)}).\]
  Given $\sigma$ and $\tau$, the corresponding summand in the first term 
  is mapped to the summand corresponding to $\omega=\sigma\tau$, with the 
  map being the $n$-fold monoidal product of the $\mc{V}$-composition
  \[ \Vhom_{\mc{C}}(L_{\sigma^{-1}(k)},M_{\tau^{-1}(\sigma^{-1}(k))}) \otimes 
  \Vhom_{\mc{C}}(K_k,L_{\sigma^{-1}(k)})  \to \Vhom_{\mc{C}}(K_k,M_{\omega^{-1}(k)}) \]
  up to a permutation of monoidal factors.
  This amounts to the formula~(\ref{wreath comp.}). 
  Associativity and identity conditions are checked readily, implying
  that $\Sigma_n\wr\mc{C}^{\otimes n}$ is indeed a \mc{V}-category.
\end{definition}

\begin{remark}\label{unordered}
    Here are two interpretations of this construction.
  \begin{enumerate}
  \item
    The \mc{V}-category $\Sigma_n\wr\mc{C}^{\otimes n}$ is the 
    \mc{V}-category of unordered $n$-tuples. More precisely, in 
    $\Sigma_n\wr\mc{C}^{\otimes n}$ there is for every $\sigma\in\Sigma_n$ and every $n$-tuple $\ul{K}$ a canonical map
    \diagr{ \ul{K}=(K_1,\dotsc,K_n)\ar[rr]^-{(\sigma^{-1},\id\otimes\dotsm\otimes \id)} && (K_{\sigma(1)},\dotsc,K_{\sigma(n)}) , }
    that is an isomorphism with inverse $(\sigma,\id\otimes\dotsm\otimes \id)$. Moreover, any map in $\Sigma_n\wr\mc{C}^{\otimes n}$ can be written as the composition of such an isomorphism with a map from $\mc{C}^{\otimes n}$:
    \begin{align*}
      (\sigma,f_1\otimes\dotsm\otimes f_n) &=(\sigma,\id\otimes\dotsm\otimes \id)\circ(\id,f_1\otimes\dotsm\otimes  f_n) \\
      &=(\id,f_{\sigma^{-1}(1)}\otimes\dotsm\otimes  f_{\sigma^{-1}(n)})\circ(\sigma,\id\otimes\dotsm\otimes \id)
    \end{align*}
  \item
    The \mc{V}-category $\Sigma_n\wr\mc{C}^{\otimes n}$ is obtained as the 
    \mc{V}-Grothendieck construction or \mc{V}-category of elements of the 
    functor $\Sigma_n\to\mc{V}\mathrm{-Cat}$ sending the unique object to 
    $\mc{C}^{\otimes n}$ with the permutation action. 
  \end{enumerate}
\end{remark}

\begin{definition}\label{varepsilon}
  For every \mc{V}-category \mc{C} there is a functor
  \[ \varepsilon\co\mc{C}^{\otimes n}\to\Sigma_n\wr\mc{C}^{\otimes n} \]
  which is the identity on objects and the inclusion of
  the summand indexed by the identity in $\Sigma_n$ on morphisms: 
  $(f_1\otimes\dotsm\otimes f_n)\mapsto(\id,f_1\otimes\dotsm\otimes f_n)$.
\end{definition}

\begin{lemma}\label{isom}
Let \mc{C} and \mc{D} be \mc{V}-categories. Suppose also that \mc{C} is small.
Precomposition with $\varepsilon$ induces an equivalence
\[ \varepsilon^*\co\Fun(\Sigma_n\wr\mc{C}^{\otimes n},\mc{D})\to\Fun_{\mathrm{sym}}(\mc{C}^{\otimes n},\mc{D})\]
of $\mc{V}$-categories.
\end{lemma}

\begin{proof}
Precomposition with $\varepsilon$ defines a \mc{V}-functor
\[ \varepsilon^*\co\Fun(\Sigma_n\wr\mc{C}^{\otimes n},\mc{D})\to\Fun({\mc{C}^{\otimes n}},{\mc{D}}) \]
By construction, every $\mc{V}$-functor (\mc{V}-natural transformation)
in the image of
$\varepsilon^\ast$ is symmetric. The $\mc{V}$-functor
with target restricted to the category of symmetric functors
will also be denoted $\varepsilon^\ast$. Unravelling the definitions
shows that a $\mc{V}$-functor $\mc{C}^{\otimes n}\to \mc{D}$ is
symmetric precisely if its domain extends (via the extra data)
to the wreath product category $\Sigma_n\wr\mc{C}^{\otimes n}$,
which essentially completes the proof.
\end{proof}

\begin{definition}\label{symmetric (co-)tensors}
Let \mc{D} be a tensored and cotensored \mc{V}-category.
For an object $L$ in $\mc{V}^{\Sigma_n}$ and a functor \dgrm{X} in $\Fun(\Sigma_n\wr\mc{C}^{\otimes n},\mc{D})$ set
\[ (\dgrm{X}\otimes_{\Sigma_n} L)(\ul{K}):=\dgrm{X}(\ul{K})\otimes L\]
using the tensor $\mc{D}\otimes\mc{V}\to\mc{D}$. This is a functor in $\ul{K}$. 
The symmetry automorphisms
\[ \sigma_{\dgrm{X}\otimes L}\co\dgrm{X}(K_1,\dotsc,K_n)\otimes L\to\dgrm{X}(K_{\sigma^{-1}(1)},\dotsc,K_{\sigma^{-1}(n)})\otimes L\]
defined by $\sigma_{\dgrm{X}\otimes L}:=\sigma_{\dgrm{X}}\otimes\sigma_{L}$ 
for every permutation $\sigma\in\Sigma_n$ turn $\dgrm{X}\otimes_{\Sigma_n} L$ 
into a symmetric functor.
For fixed $L$ in $\mc{V}^{\Sigma_n}$, the functor 
$\dgrm{X}\mapsto\dgrm{X}\otimes_{\Sigma_n} L$ has as a \mc{V}-right adjoint 
$\dgrm{Y}\mapsto {\bf hom}_{\Sigma_n}(L,\dgrm{Y})$, where 
\[ {\bf hom}_{\Sigma_n}(L,\dgrm{Y})(\ul{K}):=\hom_{\mc{D}}\bigl(L,\dgrm{Y}(\ul{K})\bigr).\]
with symmetric structure obtained by the conjugation action.
\end{definition}

\subsection{The cross effect}\label{sec:cr, basics}

\begin{convention}\label{conv3}
  Suppose that $\mc{V}=\mc{S}$, so that Convention~\ref{conv2}
  is satisfied. Suppose further that 
  \mc{D} is a bicomplete \mc{S}-category, and that \mc{C} 
  is a small \mc{S}-category with finite coproducts
  and terminal object $\ast$. 
\end{convention}

\begin{notation}%\label{}
  The \mc{S}-functor $\tr\co\mc{D}\to\mc{D}^{\Sigma_n}$, sending an object 
  to itself equipped with the trivial $\Sigma_n$-action, has a 
  \mc{S}-left adjoint given by the orbit functor $(\free)_{\Sigma_n}$. 
\end{notation}

\begin{definition}\label{P_0(ul{n})}
  Let $\ul{n}=\{1,\dotsc,n\}$, with associated power set $P(\ul{n})$, and let 
  \[ P_0(\ul{n}):=P(\ul{n})-\{\emptyset\} \]
  be the partially ordered set of non-empty subsets of $\ul{n}$.
  For every $n$-tuple $\ul{K}=\{K_1,\dotsc,K_n\}$ of objects in $\mc{C}$
  and every $S\in P_0(\ul{n})$ there is a map 
  \[ \bigvee_{i=1}^nK_i \to \bigvee_{i\in \ul{n}-S}K_i \]
  induced by the canonical inclusion $K_i \to \bigvee_{i\in \ul{n}-S}K_i$ if
  $i\notin S$ and the trivial map $K_i \to \ast$ if $i\in S$. 
\end{definition}

\begin{definition}\label{def:(ho)cross_n}
  The {\it $n$-th cross effect} 
  $\cross_n\co\mc{C}^{\wedge n}\to\mc{D}$ of a functor 
  $\dgrm{X}\co\mc{C}\to\mc{D}$  is given by the formula
  \begin{align*}
    \cross_n\dgrm{X}(K_1,\dotsc,K_n):&=\fib\left[ \dgrm{X}(\bigvee_{i=1}^nK_i)\to\lim_{S\in P_0(\ul{n})}\dgrm{X}(\bigvee_{i\in\ul{n}-S}K_i) \right] 
  \end{align*}
  where the map is induced by the maps described in
  Definition~\ref{P_0(ul{n})}. Here `$\fib$' refers to the
  strict fiber, the preimage of the basepoint. 
\end{definition}

\begin{remark}\label{rem:cross-effect}
  The $n$-th cross effect, as defined above, does not coincide with
  the construction (denoted by the same symbol) $\cross_n$ introduced
  in \cite{Goo:calc3}. Goodwillie's $\cross_n$ refers to the functor
  \[ \hocr_n\dgrm{X}(K_1,\dotsc,K_n):=\hofib\left[ \dgrm{X}(\bigvee_{i=1}^nK_i)\to\holim_{S\in P_0(\ul{n})}\dgrm{X}(\bigvee_{i\in\ul{n}-S}K_i) \right] \]
  that we call the {\it $n$-th homotopy cross effect}. 
  Section \ref{section:cr model str} supplies a model structure on 
  $\Fun(\mc{C},\mc{D})$ such that the homotopy cross effect becomes 
  the right derived functor of the strict cross effect. 
\end{remark}

  A permutation $\sigma\in \Sigma_n$ defines an automorphism of an 
  $n$-fold coproduct, whence $\cross_n$ produces symmetric functors: 
  \[ \cross_n\co\Fun(\mc{C},\mc{D})\to\Fun(\Sigma_n\wr\mc{C}^{\wedge n},\mc{D}).\]
  The main goal of this section is to construct a left adjoint
  to this functor.

\begin{definition}\label{phi}
  For every object $\ul{K}=\{K_1,\dotsc,K_n\}$ in 
  $\Sigma_n\wr\mc{C}^{\wedge n}$, the maps given in Definition~\ref{P_0(ul{n})} 
  induce a map  
  \[ \phi_{\ul{K}}\co\colim_{S\in P_0(\ul{n})}R^{\bigvee_{i\in\ul{n}-S}K_i}\to R^{\bigvee_{i=1}^nK_i} \]
  in $\Fun(\mc{C},\mc{S})$, functorial in $\ul{K}$.
\end{definition}

\begin{lemma}\label{cr-formula}
  For $\ul{K}$ in $\Sigma_n\wr\mc{C}^{\wedge n}$ and \dgrm{X} in $\Fun(\mc{C},\mc{D})$ there is a canonical isomorphism:
  \[ \cotensor{\bigwedge_{i=1}^nR^{K_i}}{\dgrm{X}}\cong\cross_n\dgrm{X}(K_1,\dotsc,K_n) \]
\end{lemma}

\begin{proof}
%%% ARC 2.1:
  In the case $n=2$, the map $\phi_{\ul{K}}$ corresponds to the
  canonical map 
  \[R^{K_1}\vee R^{K_2}\rightarrow R^{K_1\vee K_2} \cong R^{K_1}\times R^{K_2}\]
  whose objectwise 
  cofiber is the objectwise smash product $R^{K_1}\wedge R^{K_2}$ described in 
  Definition~\ref{objectwise tensor product}.
  For arbitrary $n$, %%% ARC 2.1
  the map $\phi_{\ul{K}}$ given in Definition~\ref{phi} induces an
  objectwise cofiber sequence
  \[ \colim_{S\in P_0(\ul{n})}R^{\bigvee_{i\in\ul{n}-S}K_i}\stackrel{\phi_{\ul{K}}}{\longrightarrow} R^{\bigvee_{i=1}^nK_i} \to\bigwedge_{i=1}^nR^{K_i} \]
  in $\Fun(\mc{C},\mc{S})$, 
  where $\bigwedge_{i=1}^nR^{K_i}\co\mc{C}\to\mc{S}$ 
  is the $n$-fold objectwise smash product.
  The result then follows from the enriched Yoneda lemma \ref{twisted Yoneda}.
\end{proof}

\begin{lemma}\label{Lcross formula}
  The functor 
  \[ \cross_n\co\Fun(\mc{C},\mc{D})\to
  \Fun(\Sigma_n\wr\mc{C}^{\wedge n},\mc{D})\] has a left \mc{S}-adjoint
  $\Lcross_n$,
  sending the symmetric functor $\dgrm{X}$ to the functor
  \[ K \mapsto \bigl(\dgrm{X}(K,\dotsc,K)\bigr)_{\Sigma_n}\]
  where $\Sigma_n$ operates by permuting the entries
  in $(K,\dotsc,K)$.
\end{lemma}

  Lemma~\ref{cr-formula} implies for the case  $\mc{D}=\mc{S}$ that a left \mc{S}-adjoint
  of $\cross_n$ (if it exists) sends the functor represented by 
  $\ul{K}\in \Sigma_n\wr\mc{C}^{\wedge n}$ to the objectwise smash product
  $\bigwedge_{i=1}^nR^{K_i}$. This supplies a candidate for the definition of the
  left adjoint by the enriched Yoneda lemma~\ref{twisted Yoneda}.

\begin{proof}
  Let \dgrm{X} in $\Fun(\Sigma_n\wr\mc{C}^{\wedge n},\mc{D})$ be written as a 
  colimit of representable functors:
  \[ \dgrm{X}\cong\int^{\ul{K}}\dgrm{X}(\ul{K})\wedge R^{\ul{K}} \]
  Then the functor $\Lcross_n\co\Fun(\Sigma_n\wr\mc{C}^{\wedge n},\mc{D})\to
  \Fun(\mc{C},\mc{D})$ is defined by the following coend:
  \[
    \Lcross_n(\dgrm{X})\cong\int^{\ul{K}}\!\dgrm{X}(\ul{K})\wedge\Lcross_n \bigl(R^{\ul{K}}\bigr)\cong\int^{\ul{K}}\!\bigl(\dgrm{X}(\ul{K})\wedge\bigwedge_{i=1}^n R^{K_i}\bigr).\]
  For every \mc{S}-functor $\dgrm{Y}\co \mc{C}\to \mc{D}$,
  one obtains natural isomorphisms:
  \begin{align*}
    \mc{S}_{\Fun(\mc{C},\mc{D})}\bigl(\Lcross_n(\dgrm{X}),\dgrm{Y}\bigr)
    &\cong\mc{S}_{\Fun(\mc{C},\mc{D})}\bigl(\int^{\ul{K}}\!\bigl(\dgrm{X}(\ul{K})\wedge\bigwedge_{i=1}^n R^{K_i}\bigr),\dgrm{Y}\bigr) \\
    &\cong\int_{\ul{K}}\mc{S}_{\mc{D}}\bigl(\dgrm{X}(\ul{K}),\cotensor{\bigwedge_{i=1}^nR^{K_i}}{\dgrm{Y}}\bigr) \\
    &\cong\int_{\ul{K}}\mc{S}_{\mc{D}}\bigl(\dgrm{X}(\ul{K}),(\cross_n\dgrm{Y})(\ul{K})\bigr) \\
    &\cong\mc{S}_{\Fun(\Sigma_n\wr\mc{C}^{\wedge n},\mc{D})}\bigl(\dgrm{X},\cross_n\dgrm{Y}\bigr) 
  \end{align*}
  Lemma~\ref{cr-formula} is used for the third isomorphism.
  Hence, the functor 
  $\Lcross_n$ is \mc{S}-left adjoint to $\cross_n$.
  One identifies the functor $\Lcross_n$ explicitly by the following formula:
  \begin{equation}\label{formula for Lcross}
    \Lcross_n\cong\bigl(\Delta_n^*(\free)\bigr)_{\Sigma_n}=\LKan_{\mathrm{pr}_{\mc{C}}}\circ\, \Delta_n^* 
  \end{equation}
  Here $ \mathrm{pr}_{\mc{C}}\co\mc{C}\times\Sigma_n\to\mc{C} $ is the 
  projection onto the first factor, and
  $\Delta_n^\ast$ is precomposition
  with the symmetric diagonal 
  $\Delta_n\co\Sigma_n\times\mc{C}\to\Sigma_n\wr\mc{C}^{\wedge n} $
  sending an object $K$ to the $n$-tuple $\Delta_n(K)=(K,\dotsc,K)$ and a 
  morphism $(\sigma,f)$ to $\bigl(\sigma,(f,\dotsc,f)\bigr)$.
  A straightforward computation shows that the right hand side
  of~(\ref{formula for Lcross}) sends 
  the functor represented by 
  $\ul{K}\in \Sigma_n\wr\mc{C}^{\wedge n}$ to the objectwise smash product
  $\bigwedge_{i=1}^nR^{K_i}$. Since it also commutes with colimits, 
  the isomorphism~(\ref{formula for Lcross})
  holds by the universal property of the left Kan extension. This supplies the formula stated in the Lemma.
\end{proof}

\subsection{The cross effect model structure}\label{section:cr model str}

Suppose that \mc{D} 
is a cofibrantly generated \mc{S}-model category, so
that Theorem~\ref{thm:proj-model-str.} is applicable. 
In all interesting cases, the cross effect 
\[ \cross_n\colon \Fun(\mc{C},\mc{D})_\proj\to \Fun(\Sigma_n\wr(\mc{C})^{\wedge n},\mc{D})_\proj\]
is not a right Quillen functor of projective model structures.
%%% ARC 2.2
One can deduce this from the behaviour of the second homotopy cross effect,
which measures the failure of linearity for reduced homotopy functors.
For example, the \mc{S}-functor 
$\dgrm{Q}:=\mathrm{Sing} \circ \lvert - \rvert \co \mc{S}\to \mc{S}$
sending a pointed simplicial set to the singular complex of its
geometric realization is
an objectwise fibrant replacement of the identity functor $\Id_{\mc{S}}$.
As the functor $\Id_{\mc{S}}$ is not linear, its homotopy cross effect 
$\hocr_2(\Id_{\mc{S}})\simeq \hocr_2(\dgrm{Q})$ is not contractible.
However, since the canonical map 
\[ \dgrm{Q}(K\vee L) \to \dgrm{Q}(K\times L)\cong 
\dgrm{Q}(K)\times \dgrm{Q}(L)\]
is injective for all $K,L\in \mc{S}$, one has $\cross_2\dgrm{Q} = \dgrm{\ast}$.
%%% ARC 2.2
The purpose of this section is to supply a model structure on 
$\Fun(\mc{C},\mc{D})$ with objectwise weak equivalences,
such that the $n$-th cross effect is a right Quillen functor. 
The task will be accomplished by introducing more cofibrations.
The right derived functor of the cross effect turns out to be Goodwillie's 
homotopy cross effect, as promised in Remark~\ref{rem:cross-effect}.

\begin{definition}%\label{Phi}
  Let $n\ge 1$. If $\ul{K}=\{K_1,\dotsc,K_n\}$ is an $n$-tuple of objects in 
  $\mc{C}$, one says $|\ul{K}|=n$. Definition~\ref{phi} supplies a map 
  \[ \phi_{\ul{K}}\co\colim_{S\in P_0(\ul{n})}R^{\bigvee_{i\in\ul{n}-S}K_i}\to 
  R^{\bigvee_{i=1}^nK_i} \]
  in $\Fun(\mc{C},\mc{S})$ for every $\ul{K}$ with $\lvert \ul{K}\rvert =n$. 
  Let $\Phi_n:=\{\phi_{\ul{K}}\,|\,|\ul{K}|\le n\}$,
  and $\Phi:=\Phi_{\infty}:=\bigcup_{n\ge 1}\Phi_n$. 
\end{definition}

The argument with $\dgrm{Q}$ given above indicates that not
all of the maps in $\Phi_2$ are projective cofibrations.

\begin{definition}\label{gen-sets-for-cr}
  The tensor $\mc{S}\times\mc{D}\to\mc{D}$ induces a tensor $\Fun(\mc{C},\mc{S})\times\mc{D}\to\Fun(\mc{C},\mc{D})$. Hence, the pushout product $\square$ of Definition~\ref{def:pushout product} of a map in $\Fun(\mc{C},\mc{S})$ and a map in \mc{D} is defined. For $1\le n\le\infty$, the pushout product defines two sets of maps in $\Fun(\mc{C},\mc{D})$:
  \[ I^{\mathrm{cr}}_n =\left(\Phi_n\,\square\, I_{\mc{D}}\right) \quad \quad
  J^{\mathrm{cr}}_n =\left(\Phi_n\,\square\, J_{\mc{D}}\right) \]
  Set $I^{\mathrm{cr}}:=I^{\mathrm{cr}}_\infty$ and $J^{\mathrm{cr}}:=J^{\mathrm{cr}}_\infty$. 
  A map in $\Fun(\mc{C},\mc{D})$ is 
  \begin{enumerate}
  \item
    a {\it cross effect fibration} or {\it cr fibration} if it belongs to the class $J^{\mathrm{cr}}$-inj, i.e. it has the right lifting property with respect to all maps in $J^{\mathrm{cr}}$. 
  \item
    a {\it cross effect cofibration} or {\it cr cofibration} if it belongs to the class $I^{\mathrm{cr}}$-cof, i.e. it has the left lifting property with respect to all cr fibrations. 
  \end{enumerate}
  Weak equivalences are still given by objectwise weak equivalences.
\end{definition}

\begin{remark}\label{rem:proj-cross}
  Note that $I^{\mathrm{cr}}_1=I^{\proj}_{\Fun(\mc{C},\mc{D})}$ and 
  $J^{\mathrm{cr}}_1=J^{\proj}_{\Fun(\mc{C},\mc{D})}$. In particular, 
  every projective cofibration is a cr cofibration and 
  every acyclic projective cofibration is an acyclic cr cofibration.
\end{remark}

\begin{theorem}\label{thm:cross-model}
  The classes of objectwise weak equivalences, cross effect fibrations and 
  cross effect 
  cofibrations form a cofibrantly generated 
  $\mc{S}$-model 
  structure on the category $\Fun(\mc{C},\mc{D})$,
  which is as proper as $\mc{D}$. 
\end{theorem}

The idea for the proof of the theorem can be found in \cite{Isa:flasque}. 
It is an application of the recognition principle for 
cofibrantly generated model structures \cite[11.3.1]{Hir:loc}.
Replacing $\Phi$ with $\Phi_n$ for some $1\le n\le\infty$ leads to a model structure where only the $k$-th cross effects for $1\le k\le n$ become right Quillen functors. The model structure supplied by Theorem~\ref{thm:cross-model}
is called the {\it cross effect model structure} or simply 
{\it cr model structure} and denoted by $\Fun(\mc{C},\mc{D})_{\cross}$.

\begin{proof}
  Remark~\ref{rem:proj-cross} implies that every
  projective cofibration is a cr cofibration.
  Sources and targets of maps in $I^{\mathrm{cr}}$ and $J^{\mathrm{cr}}$ are as 
  small as the sources and targets of the maps in $I_\mc{D}$ and $I_\mc{D}$,
  thus allowing the small object argument. 
  Lemmata~\ref{isaksen1},~\ref{isaksen2} and~\ref{J-cof are we} below 
  conclude the proof of the existence of the cofibrantly generated 
  model structure. The model structure is an $\mc{S}$-model structure, 
  as one checks on the generators $I^{\mathrm{cr}}$ and $J^{\mathrm{cr}}$.
  The statement regarding properness is proved in \ref{lem:cr-proper}.
\end{proof}

\begin{lemma}\label{isaksen1}
  A map in $\Fun(\mc{C},\mc{D})$ is in \emph{$I^{\mathrm{cr}}$-inj} if and only if it is an objectwise weak equivalence and a cross effect fibration. 
\end{lemma}

\begin{proof}
  Let $f\co\dgrm{X}\to\dgrm{Y}$ be in $I^{\mathrm{cr}}$-inj. 
  Then $f$ is in $I^{\proj}_{\Fun(\mc{C},\mc{D})}$-inj, and so 
  an acyclic projective fibration. In particular, 
  $f$ is an objectwise weak equivalence.
  The inclusion $I_{\mc{D}}\mathrm{-inj}\subset J_{\mc{D}}\mathrm{-inj}$ implies
  \[ (\Phi\,\square\, I_{\mc{D}})\mathrm{-inj}\subset (\Phi\,\square\, J_{\mc{D}})\mathrm{-inj},\] 
  whence $f$ is a cr fibration as well.

  Conversely, let $f$ be an objectwise equivalence and a cr fibration. 
  Then it is in particular an acyclic projective fibration, hence in 
  $I^{\proj}_{\Fun(\mc{C},\mc{D})}$-inj. By assumption, it is also in 
  $(\Phi\,\square\, J_{\mc{D}})$-inj, which means exactly that 
  the map 
  \begin{equation}\label{sp condition}
    \dgrm{X}(\bigvee_{i=1}^nK_i)\to
    \lim_{S\in P_0(\ul{n})}\dgrm{X}(\bigvee_{i\in\ul{n}-S}K_i)\x_{\lim\limits_{S\in P_0(\ul{n})}
      \dgrm{Y}(\bigvee\limits_{i\in\ul{n}-S}K_i)}\dgrm{Y}(\bigvee_{i=1}^nK_i)
  \end{equation}
  is a fibration for every possible choice of $n$ and $\ul{K}$. 
  It remains to show that $f$ is in $(\Phi\,\square\, I_{\mc{D}})$-inj. This is equivalent to the 
  map in~(\ref{sp condition}) being an acyclic fibration. By 
  assumption, the map
  \[ \dgrm{X}(\bigvee_{i=1}^nK_i) \to \dgrm{Y}(\bigvee_{i=1}^nK_i)\] 
  is a weak equivalence. 
  Thus it suffices to show that the map
  \[ \lim_{S\in P_0(\ul{n})}\dgrm{X}(\bigvee_{i\in\ul{n}-S}K_i)\to
  \lim_{S\in P_0(\ul{n})}\dgrm{Y}(\bigvee_{i\in\ul{n}-S}K_i)\]
  is an acyclic fibration. To conclude this, recall that 
  $P_0(\ul{n})$ is an inverse category by 
  the functor $\mathrm{deg}\co P_0(\ul{n})^{\mathrm{op}} \to \mathbb{N}$ 
  which sends $S$ to the number of elements in $\ul{n} \smallsetminus S$. By 
  \cite[Theorem 5.1.3]{Hov:model}, there is a model structure on the category of
  functors from $P_0(\ul{n})$ to any model category. It has objectwise weak 
  equivalences
  and cofibrations, and the fibrations are characterized by an appropriate
  matching space condition. The limit is thus a right Quillen 
  functor on this functor category, and
  in particular preserves acyclic fibrations. 
  Since $f\co \dgrm{X} \to \dgrm{Y}$
  is a cr fibration, the induced natural transformation $f^\prime$ of 
  functors on $P_0(\ul{n})$
  is a fibration. As $f$ is an objectwise equivalence, $f^\prime$ is
  an objectwise weak equivalence. The result follows.
\end{proof}

\begin{lemma}\label{isaksen2}
A map in $J^{\mathrm{cr}}$-cof is a cr cofibration.
\end{lemma}

\begin{proof}
  If a map is in $J^{\mathrm{cr}}$-cof, it has the left lifting property with respect to all cr fibrations. 
  In particular, it has the left lifting property with respect to all cr fibrations that are also 
  objectwise weak equivalences. So, by Lemma~\ref{isaksen1} it is an cr cofibration.
\end{proof}

\begin{lemma}\label{J-cof are we}
  A map in $J^{\mathrm{cr}}$-cof is an objectwise weak equivalence.
\end{lemma}

\begin{proof}
  By the small object argument, every map in $J^{\mathrm{cr}}$-cof is a 
  retract of a map in $J^{\mathrm{cr}}$-cell. Since $\mc{V}=\mc{S}$,
  every map in $\Phi$ is an objectwise cofibration
  in $\Fun(\mc{C},\mc{S})$. This implies that
  every map in $J^{\mathrm{cr}}$-cell 
  is an objectwise weak equivalence. 
\end{proof}

\begin{lemma}\label{lem:cr-proper}
  If $\mc{D}$ is right or left proper, then the cross effect model structure is 
  right or left proper, respectively. 
\end{lemma}

\begin{proof}
  Any cr fibration is an objectwise fibration and any cr cofibration is an 
  objectwise cofibration, again using $\mc{V}=\mc{S}$ for the latter
  statement. Since pullbacks and pushouts are formed objectwise, 
  the statement follows. 
\end{proof}

\begin{lemma}\label{cr-fibrant}
  If the functor $\dgrm{X}$ is cross effect fibrant, the canonical map
  \[ \cross_n\dgrm{X}\to\hocr_n\dgrm{X} \]
  is an objectwise weak equivalence.
\end{lemma}

\begin{proof}
  If \dgrm{X} is cr fibrant, then -- as in the proof of Lemma \ref{isaksen1} -- 
  the $P_0(\ul{n})$-diagram 
  \[  S\mapsto \dgrm{X}\left( \bigvee_{i\in\ul{n}-S}K_i\right)\]
  is injectively fibrant. 
  This follows from the right lifting property of the map $\dgrm{X}\to\ast$ 
  with respect to $J_{m}^{\mathrm{cr}}$ for $1\le m<n$. 
  Thus, the map
  \[ \lim_{S\in P_0(\ul{n})}\dgrm{X}(\bigvee_{i\in\ul{n}-S}K_i)\to
  \holim_{S\in P_0(\ul{n})}\dgrm{X}(\bigvee_{i\in\ul{n}-S}K_i)\]
  is a weak equivalence. The right lifting property with respect to 
  $J_{n}^{\mathrm{cr}}$ implies that the map
  \[ \dgrm{X}(\bigvee_{i=1}^nK_i)\to\lim_{S\in P_0(\ul{n})}\dgrm{X}(\bigvee_{i\in\ul{n}-S}K_i)\]
is a fibration. The claim follows.
\end{proof}

\begin{proposition}\label{prop:cross-right-obj}
  The functor 
  \[\cross_n\colon \Fun(\mc{C},\mc{D})_{\cross}\to 
     \Fun(\Sigma_n\wr (\mc{C})^{\wedge n},\mc{D})_\proj\]
  is a right Quillen functor and $\hocr_n$ is its right derived functor.
\end{proposition}

\begin{proof}
  Let $\ul{K}=(K_1,\dotsc,K_n)$ be an $n$-tuple of objects in \mc{C}. 
  The cofiber sequence
  \diagr{ \colim_{S\in P_0(\ul{n})}R^{\bigvee_{i\in\ul{n}-S}K_i}\ar[r]^-{\phi_{\ul{K}}} & R^{\bigvee_{i=1}^nK_i} \ar[r] & \bigwedge_{i=1}^nR^{K_i} }
  in $\Fun(\mc{C},\mc{S})$ implies that
  the functor $\bigwedge_{i=1}^nR^{K_i}$ is cr cofibrant, because
  the map $\phi_{\ul{K}}$ is a cr cofibration. 
  By the formula \ref{cr-formula}
  \[ \cotensor{\bigwedge_{i=1}^nR^{K_i}}{\dgrm{X}}\cong\cross_n\dgrm{X}(K_1,\dotsc,K_n),\] 
  the strict cross effect is a right Quillen functor.
  Its right derived functor $\hocr_n$ is identified by Lemma~\ref{cr-fibrant}.
\end{proof}

\section{Homotopy functors}\label{sec:homotopy-functors}

\begin{definition}\label{def-homotopy-functor}
  Suppose \mc{B} and \mc{D} are model categories and \mc{C} is a small 
  full subcategory of \mc{B}.
  A functor in $\Fun(\mc{C},\mc{D})$ is called a {\it homotopy functor\/} 
  if for every weak equivalence $A\to B$ in \mc{C} the image $F(A)\to F(B)$ 
  is a weak equivalence in \mc{D}.
\end{definition}

Homotopy functors are the main object of study in 
Goodwillie's calculus of functors. From the point of view of
model categories, the full subcategory of homotopy functors is
usually inadequate.
The aim of this section is to construct a model structure in
which every functor is a homotopy functor, up to weak equivalence.

\subsection{Homotopy functors and simplicial functors}\label{sec:homot-funct-simpl}

A preliminary goal is to show that every homotopy functor of
reasonable categories is objectwise weakly equivalent to
a simplicial functor. The following statement is a slight generalization of 
a lemma by Waldhausen~\cite[Lemma 3.1.2, 3.1.3]{Waldhausen:alg-k-theory-spaces}
to certain $\mc{U}$-model categories. 

\begin{lemma}\label{lem:waldhausen-general}
  Let \mc{C} be a small subcategory of a simplicial model category, 
  closed under cotensoring with finite simplicial sets, and let \mc{D} 
  be a \mc{U}-model category. Suppose that $X\colon \mc{C}\to \mc{D}$ is a 
  homotopy functor. If the simplicial object $n \mapsto X(A^{\Delta^n})$ 
  is Reedy cofibrant for every object $A\in \mc{C}$, then there exists a 
  \mc{U}-functor $\dgrm{X}\colon \mc{C}\to \mc{D}$ and a 
  natural objectwise weak equivalence $f\co X\to \dgrm{X}$.
\end{lemma}

\begin{proof}
  The value of the functor $\dgrm{X}$ at an object $A$ of \mc{C} is defined as the coend
  \[ \dgrm{X}(A):=\int^{n} X(A^{\Delta^n})\times \Delta^n, \]
  which is in fact the standard realization of a simplicial object in $\mc{D}$.
  Expressing $X(A)$ as the standard realization of a constant simplicial 
  object, one obtains a natural transformation 
  $f\colon X \to \dgrm{X}$ via $\Delta^n \to \Delta^0$: 
  \[ f(A)\co X(A)\cong\int^{n} X(A^{\Delta^0})\times \Delta^n\to\int^{n} X(A^{\Delta^n})\times \Delta^n=\dgrm{X}(A)  \]
  The map $A=A^{\Delta^0} \to A^{\Delta^n}$ is a simplicial homotopy equivalence, 
  since $\Delta^n \to \Delta^0$ is one. It follows that it is a weak 
  equivalence in \mc{C}. Hence, so is its image under $X$ by assumption. 
  The constant simplicial object $X(A)$ is cofibrant in the Reedy model 
  structure. Since by assumption the target of $f(A)$ is Reedy cofibrant 
  as well, $f$ is a natural weak equivalence. The reason is that realization 
  is a left Quillen functor on the Reedy model structure.

  It remains to prove that $\dgrm{X}$ is a $\mc{U}$-functor. 
  A map of simplicial sets
  \[ \mc{U}_{\mc{C}}(A,B)\to \mc{U}_{\mc{D}}\bigl(\dgrm{X}(A),\dgrm{X}(B)\bigr)\]
  will be given in simplicial degree $m$ as follows.
  An $m$-simplex $A\times \Delta^m \to B$ can equivalently be described 
  as a map $\alpha\colon A\to B^{\Delta^m}$. 
  Consider the simplicial objects 
  $[n]\mapsto F_n=X(A^{\Delta^n}), [n]\mapsto G_n=X(B^{\Delta^n})$. 
  An $m$-simplex $F_\bullet \to G_\bullet$ is the same as a natural transformation
  \[ \bigl(\gamma\colon [n]\to [m]\bigr) \mapsto (t_\gamma \colon F_n\to G_n).\]
  Set $t_\gamma$ to be the composition
  \[ X\bigl(A^{\Delta^n}\bigr) \xrightarrow{X\bigl(\alpha^{\Delta^n}\bigr)}    X\bigl(B^{\Delta^m\times \Delta^n}\bigr) \xrightarrow{X\bigl(B^{(\gamma,\id)}\bigr)}
    X\bigl(B^{\Delta^n}\bigr) \]
    which induces the desired map
    \[ \dgrm{X}(A)\times \Delta^m \to \dgrm{X}(B).\]
    The verification of the relevant axioms this map has to fulfill is 
    left to the reader.
\end{proof}

Note that the condition on Reedy cofibrancy is fulfilled automatically
in many cases, for example in the category of simplicial presheaves
with the injective model structure.

\begin{remark}\label{rem:waldhausen-pointed}
  Recall from Definition \ref{def:(multi-)red} that a functor $X$ between 
  pointed categories is {\it reduced} if $X(\ast)\cong\ast$.  
  A \mc{U}-functor is an $\mc{S}$-functor if and only it is reduced. 
  Hence, the analog of Lemma~\ref{lem:waldhausen-general} for 
  \mc{S}-model categories 
  and reduced homotopy functors holds as well. 
  In fact, one can replace a homotopy functor with $X(\ast)\simeq\ast$ by a weakly equivalent \mc{S}-functor. 
\end{remark}

\subsection{A model structure for simplicial homotopy functors}
\label{sec:simp-homot-funct}

The purpose of this section is to construct a model structure on a 
category of enriched functors in which every enriched functor is weakly 
equivalent to an enriched homotopy functor. 
For specific categories of enriched functors this has been obtained 
in \cite{Lydakis}, \cite{DRO:enriched} and \cite{BCR:calc}. Although 
more general results are possible, a restriction to simplicial functors
seems adequate, as Lemma~\ref{lem:waldhausen-general} suggests.
None of this is necessary if all functors are already homotopy functors, as it is the case if the source category is the category of finite CW complexes,
or more generally consists of bifibrant objects only.
This section is written in the pointed setting. 
All statements in this section and their proofs have unpointed variants,
whose formulation is left to the reader.

\begin{definition}\label{def:simp-fib-repl}
  An $\mc{S}$-model category $\mc{B}$ has a {\em decent fibrant replacement functor\/} if there exists a $\mc{S}$-natural transformation
  \[ \phi_\fibr\colon \Id_{\mc{B}} \to \fibr \] 
  of $\mc{S}$-functors satisfying the following conditions:
  \begin{enumerate}
  \item For every object $A\in \mc{B}$ the object $\fibr(A)$ is
    fibrant and $\phi_\fibr(A)$ is an acyclic cofibration.
  \item The functor $\fibr$ sends weak equivalences
    of cofibrant objects to simplicial homotopy equivalences.
  \item The functor $\fibr$ commutes with filtered colimits.
  \end{enumerate}
\end{definition}

\begin{example}
  In the case $\mc{B} =\mc{S}$ or \mc{U} one can use Kan's 
  $\mathrm{Ex}^\infty$, as well as the composition of the geometric 
  realization and the singular complex, as a decent fibrant replacement functor.
\end{example}

\begin{lemma}\label{lem:simp-fib-repl}
  Let $\mc{B}$ be an $\mc{S}$-model category. Suppose
  there exists a set \[ \{j\colon sj\to tj\}_{j\in J}\]
  of acyclic cofibrations
  in $\mc{B}$ with the following properties:
  \begin{enumerate}
    \item An object $A\in \mc{B}$ is fibrant if 
      $A\to \ast$ has the right lifting property with
      respect to $J$.
    \item The functor $\mc{S}_{\mc{B}}(sj,-)=R^{sj}\colon \mc{B}\to \mc{S}$
      commutes with filtered colimits for every $j\in J$.
  \end{enumerate}
  Then $\mc{B}$ has a decent fibrant replacement functor. 
\end{lemma}

\begin{proof}
  This is an enriched version of Quillen's small object argument, as 
  constructed in \cite{DRO:enriched}. Let $\fibr_1$
  be the $\mc{S}$-functor defined as the pushout of
  \[ \xymatrix{ \bigvee_{j\in J} R^{sj}\wedge tj &
           \bigvee_{j\in J} R^{sj}\wedge sj  \ar[l] 
            \ar[r] & \Id_{\mc{B}} }. \]
  It comes together with an $\mc{S}$-natural transformation
  $\phi_1 \colon \Id_{\mc{B}} \to \fibr_1$. For $n\geq 1 $ set $\fibr_{n+1} = \fibr_1 \circ \fibr_{n}$
  and let 
  \[ \fibr = \colim \bigl( \Id_{\mc{B}} \xrightarrow{\phi_1} \fibr_1 \xrightarrow{\phi_1\circ \fibr_1}
    \fibr_2 \to \dotsm\bigr). \]
  The natural transformation $A\to \fibr(A)$ is then an acyclic cofibration.
  Since every $sj$ is in particular finitely presentable, 
  a morphism $\alpha\colon sj \to \fibr(A)$ factors over $\fibr_n(A)$. The composition
  \[ tj \cong S^0\wedge tj \to \bigvee_{j \in J} \mc{S}_{\mc{B}}\bigl(sj,\fibr_n(A)\bigr)\wedge tj
      \to \fibr_{n+1}(A)\to \fibr(A) \]
  solves the lifting problem given by $\alpha$. Thus $\fibr(A)$ is fibrant.
  Weak equivalences of bifibrant objects in an \mc{S}-model category
  are simplicial homotopy equivalences by \cite[Section 9.5]{Hir:loc}. 
  Thus it remains to prove
  the third condition. It follows because $\fibr_1$ is a colimit
  of functors preserving filtered colimits by definition. 
\end{proof}

\begin{convention}\label{conv:hf}
  In addition to Convention~\ref{conv3},
  the following statements are assumed to be true:
  \begin{enumerate}
  \item\label{item:finpres} The category \mc{C} is a small full sub-\mc{S}-category of an \mc{S}-model category $\mc{B}$, containing only \mc{S}-finitely presentable cofibrant objects.
  \item\label{item:decfib} There exists a decent fibrant replacement functor $\Id_{\mc{B}} \to \fibr$ such that, for every $A\in \mc{C}$, the object $\fibr(A)\in\mc{B}$ is a filtered colimit of objects in $\mc{C}$.
  \item\label{item:smodel} The category \mc{D} is a right proper cofibrantly generated \mc{S}-model category.
  \item\label{item:fingen} In \mc{D}, weak equivalences, fibrations with fibrant codomain, and pullbacks are preserved under filtered colimits.
  \end{enumerate}
\end{convention}

A sufficient condition on the model category \mc{D} to 
satisfy \ref{conv:hf}(\ref{item:fingen}) is essentially
due to Voevodsky.

\begin{definition}\cite[3.4]{DRO:enriched} \label{weakly finitely generated}
  A cofibrantly generated model category \mc{D} is called {\it weakly finitely generated} if we can choose a set of generating cofibrations $I$ and of generating acyclic cofibrations $J$ such that the following conditions hold:
  \begin{enumerate}
  \item
    The domains and codomains of the maps in $I$ are finitely presentable.
  \item
    The domains of the maps in $J$ are small.
  \item
    There exists a subset $J'$ of $J$ of maps with finitely presentable domains and codomains such that a map in \mc{D} with fibrant codomain is a fibration if and only if it is in $J'$-inj.
  \end{enumerate}
\end{definition}

If \mc{D} is a locally finitely presentable category, then pullbacks are 
preserved under filtered colimits in \mc{D}. The 
remaining requirements listed in \ref{conv:hf}(\ref{item:fingen}) are met, 
as proved in~\cite[Lemma 3.5]{DRO:enriched}, if the model structure 
on \mc{D} is weakly finitely generated. 
The class of weakly finitely generated model structures is closed under 
left Bousfield localization with respect to a set of morphisms with finitely 
presentable cofibrant (co)domains.

Suppose from now on that Convention~\ref{conv:hf} holds. 
%%% ARC 2.3 and 2.4
The first
condition of Convention~\ref{conv:hf} has the following consequence.

\begin{lemma}\label{lem:filt-colim}
  Let $i\colon \mc{C}\to \mc{B}$ be the inclusion
  functor. Every $\mc{S}$-functor $\dgrm{X}\colon \mc{C}
  \to \mc{D}$ admits a left $\mc{S}$-Kan extension
  $i_{*}(\dgrm{X})\colon \mc{B} \to \mc{D}$, and
  the latter
  preserves filtered colimits.
\end{lemma}

\begin{proof}
  Condition~\ref{conv:hf}(\ref{item:finpres}) implies
  that every \mc{S}-functor $\dgrm{Y}\colon \mc{B}\to \mc{D}$
  which is represented by an object of $\mc{C}$  
  preserves filtered colimits. Suppose $\dgrm{X}\colon \mc{C}\to \mc{D}$
  is an $\mc{S}$-functor. By construction, the
  left \mc{S}-Kan extension $i_\ast(\dgrm{X})$ 
  is a colimit of \mc{S}-functors represented by
  objects of $\mc{C}$, which gives the result.
\end{proof}

\begin{definition}
  Let $\dgrm{X}\colon \mc{C}\to \mc{D}$
  be an $\mc{S}$-functor, and let
  \begin{equation}\label{eq:hf}
    \phi_{\dgrm{X}}\colon \dgrm{X} \to \dgrm{X}^\hf:= i_{*}(\dgrm{X})\circ \fibr \circ i
  \end{equation}
  denote the canonical map to the composition. The
  composition $\dgrm{X}^\hf\co\mc{C}\to \mc{D}$ is again an $\mc{S}$-functor.
\end{definition}
%%% ARC 2.3 and 2.4
\begin{definition}\label{def:hf}
  A map $f\colon \dgrm{X}\to \dgrm{Y}$ is 
  \begin{enumerate}
  \item
    an {\em hf equivalence\/} if the map $f^\hf  \colon \dgrm{X}^\hf \to \dgrm{Y}^\hf$ is an objectwise weak equivalence.
  \item  
    an {\em hf fibration\/} if it is an objectwise fibration $\dgrm{X}\to\dgrm{Y}$ such that the square
    \diagr{ \dgrm{X} \ar[r]\ar[d] & \dgrm{X}^\hf \ar[d] \\ 
      \dgrm{Y} \ar[r] & \dgrm{Y}^\hf  }
    
    is an objectwise homotopy pullback square.
  \end{enumerate}
  The hf cofibrations are the projective cofibrations. 
  Theorem~\ref{thm:homotopy-model} states that these classes form a
  model structure on $\Fun(\mc{C},\mc{D})$. It is called the 
  {\it homotopy functor model structure\/} or {\it hf model structure}
  for short, and denoted $\Fun(\mc{C},\mc{D})_{\mathrm{hf}}$. 
  Analogous definitions can be given starting from the cross effect model 
  structure instead of the projective model structure. The resulting
  model category is denoted $\Fun(\mc{C},\mc{D})_{\mathrm{hf}-\cross}$.
\end{definition}

\begin{remark}\label{rem:simp-hty-eq-weak-eq}
  In an \mc{S}-model category,
  any simplicial homotopy equivalence is in particular
  a weak equivalence. Simplicial homotopy equivalences -- 
  unlike weak equivalences -- are preserved by any \mc{S}-functor.
\end{remark}

\begin{remark}\label{remarks about hf}
  A few words about hf fibrant \mc{S}-functors:
\begin{enumerate} 
   \item
     A (cr) fibrant functor $\dgrm{X}$ is (cr) hf fibrant if and only if the map (\ref{eq:hf}) is an objectwise weak equivalence. 
   \item
     The functor $\dgrm{X}^\hf$ preserves weak equivalences. Thus a (cr) hf fibrant functor preserves weak equivalences. 
   \item
     A (cr) fibrant functor preserving weak equivalences is (cr) hf fibrant.
   \item
     The hf fibrant functors are exactly the objectwise fibrant homotopy functors.
     The hf cr fibrant functors are exactly the cr fibrant homotopy functors.
\end{enumerate}
\end{remark}

\begin{theorem}\label{thm:homotopy-model}
  Assume Convention~\emph{\ref{conv:hf}}. The classes of 
  maps given in Definition~\emph{\ref{def:hf}}, starting
  from the projective or the cross effect model 
  structure, constitute
  a right proper cofibrantly generated \mc{S}-model structure. 
  It is left proper if \mc{D} is left proper.
\end{theorem}

\begin{proof}
  As in our previous article \cite{BCR:calc},
  it suffices to check that the natural 
  transformation $\phi_{\dgrm{X}}\colon \dgrm{X}\to \dgrm{X}^\hf$
  satisfies the axioms (A1), (A2), and (A3) given 
  by Bousfield in \cite[9.2]{Bou:telescopic}. Cofibrant generation 
  is delegated to Lemma \ref{lem:gen-cof-hf}.
  
  Axiom (A1): Let $f\colon \dgrm{X}\to \dgrm{Y}$ be
  an objectwise weak equivalence. To prove that 
  $f^\hf$ is an objectwise weak equivalence, let $A\in \mc{C}$
  and express $\fibr(A)$ as a filtered colimit of objects $B_i$
  in \mc{C} by condition~\ref{conv:hf}(\ref{item:decfib}).
  Lemma~\ref{lem:filt-colim} implies that
  $f^\hf(A)$ is the morphism induced on filtered colimits
  by the morphisms $f(B_i)$, which are weak equivalences.
  Condition~\ref{conv:hf}(\ref{item:fingen}) implies that
  $f^\hf(A)$ is itself a weak equivalence.

  Axiom (A2): The task is to identify the two
  natural transformations
  \[ \phi_{\dgrm{X}^\hf},\phi_{\dgrm{X}}^\hf \colon \dgrm{X}^\hf \to 
             \bigl(\dgrm{X}^\hf \bigr)^\hf\]
  as weak equivalences. The triangular identities and the
  natural isomorphism 
  \[\dgrm{X}\xrightarrow{\cong} i_{*}(\dgrm{X})\circ i \]
  reduce the problem to the value of the two natural maps
  \[ \fibr(A) \to \fibr(\fibr(A)) \]
  under the \mc{S}-functor $\dgrm{X}$. Both maps are
  simplicial homotopy equivalences, since $\fibr$ is a decent
  fibrant replacement functor and all objects in \mc{C} are
  cofibrant by Condition~\ref{conv:hf}(\ref{item:finpres}).
  Remark~\ref{rem:simp-hty-eq-weak-eq} then implies that
  the maps in question are objectwise weak equivalences.

  Axiom (A3): Let $f\colon \dgrm{X}\to \dgrm{Y}$ be
  an hf weak equivalence and let 
  $p\colon \dgrm{Z} \to \dgrm{Y}$ be an objectwise 
  fibration with $\dgrm{Y}$
  objectwise fibrant; cr fibrations are not necessary.
  Consider the pullback diagram
  \diagram{\dgrm{X}\times_{\dgrm{Y}}\dgrm{Z} \ar[r]^-g\ar[d] & \dgrm{Z} \ar[d]^-p \\
    \dgrm{X}\ar[r]^f & \dgrm{Y}}{eq:a6}
  The goal is to prove that $g^\hf$ is an objectwise weak equivalence.
  Pullbacks are computed objectwise. 
  Lemma~\ref{lem:filt-colim} and Condition~\ref{conv:hf}(\ref{item:fingen})
  imply that the diagram 
  \diagram{\bigl(\dgrm{X}\times_{\dgrm{Y}}\dgrm{Z}\bigr)^\hf \ar[r]^-{g^\hf}\ar[d] & 
    \dgrm{Z}^\hf \ar[d]^-{p^\hf} \\
    \dgrm{X}^\hf\ar[r]^{f^\hf} & \dgrm{Y}^\hf}{eq:a6-2}
  is a pullback diagram. Moreover, since fibrations in $\mc{D}$ 
  with fibrant target are closed
  under filtered colimits, $p^\hf$ is still an objectwise fibration.
  Now $f^\hf$ is an objectwise weak equivalence by assumption and 
  $\mc{D}$ is right
  proper. Thus, $g^\hf$ is an objectwise weak equivalence, which finishes the
  proof. 
\end{proof}

\begin{lemma}\label{hf fibrations}
  Let $p\co\dgrm{X}\to\dgrm{Y}$ be an objectwise fibration. 
  Then the following statements are equivalent:
  \begin{enumerate}
  \item[\emph{(i)}]
    The map $p$ is an hf fibration.
  \item[\emph{(ii)}]
    The induced square
    \diagr{ \dgrm{X} \ar[r]\ar[d] & \dgrm{X}^{\hf} \ar[d] \\ \dgrm{Y} \ar[r] & \dgrm{Y}^{\hf}}
    is an objectwise homotopy pullback.
  \item[\emph{(iii)}]
    For each weak equivalence $A\to B$ in \mc{C} the induced square
    \diagr{ \dgrm{X}(A) \ar[r]\ar[d] & \dgrm{X}(B) \ar[d] \\ \dgrm{Y}(A) \ar[r] & \dgrm{Y}(B)}
    is a homotopy pullback square.
  \end{enumerate}
  The corresponding statement for cr fibrations holds as well.
\end{lemma}

\begin{proof}
  The equivalence of (i) and (ii) follows from the 
  characterization \cite[Theorem 9.3]{Bou:telescopic} of fibrations in the 
  localized model structure. The equivalence of (ii) and (iii) will be 
  shown now.
  Let $\dgrm{X}\to\dgrm{Y}$ satisfy (ii). 
  Since the functors $\dgrm{X}^{\hf}$ and $\dgrm{Y}^{\hf}$ are homotopy functors, the induced diagram
  \diagr{ \dgrm{X}^{\hf}(A) \ar[r]\ar[d] & \dgrm{X}^{\hf}(B) \ar[d] \\ \dgrm{Y}^{\hf}(A) \ar[r] & \dgrm{Y}^{\hf}(B) }
  is a homotopy pullback diagram for any weak equivalence $A\to B$ in \mc{C} for
  trivial reasons. This means that the 
  composed outer square and the right hand square in the following 
  diagram are homotopy pullbacks:
  \diagr{ \dgrm{X}(A) \ar[r]\ar[d] & \dgrm{X}(B) \ar[r]\ar[d] & \dgrm{X}^{\hf}(B) \ar[d] \\
    \dgrm{Y}(A) \ar[r] & \dgrm{Y}(B) \ar[r] & \dgrm{Y}^{\hf}(B) }
  It follows that the left hand square is a homotopy pullback. 

  Now let $\dgrm{X}\to\dgrm{Y}$ satisfy property (iii), and let $A\to B$ be a weak equivalence in \mc{C}. Consider a decent fibrant replacement $\phi_A\co A\to\fibr (A)$. By Convention~\ref{conv:hf}(\ref{item:decfib}), this map factors through a colimit 
  \[A\to\dotsm\to B_i\to B_{i+1}\to\dotsm\to\fibr (A) \]
  where all objects $B_i$ are in \mc{C}. By (iii) there are homotopy pullback diagrams:
  \diagr{ \dgrm{X}(A)\ar[r]\ar[d] & \dgrm{X}(B_i) \ar[d] \\ \dgrm{Y}(A) \ar[r] & \dgrm{Y}(B_i) }
  Their colimit yields the desired homotopy pullback in (ii) since homotopy pullbacks commute with filtered colimits in \mc{D} by Convention~\ref{conv:hf}(\ref{item:fingen}).
\end{proof}

\begin{lemma}\label{lem:gen-cof-hf}
  Assume Convention~\emph{\ref{conv:hf}}. Then the homotopy functor model structure is cofibrantly generated.
\end{lemma}

\begin{proof}
  In order to enlarge the set of generating acyclic cofibrations of the cr or
  projective model structure, respectively, 
  take an arbitrary weak equivalence $w\co A\stackrel{\simeq}{\to} B$ 
  in \mc{C}. It induces the map 
  \[ w^*\co R^B\to R^A \]
  which, using the simplicial mapping cylinder construction, 
  can be factored as a projective cofibration $w\prime$, followed by a 
  simplicial homotopy equivalence. The
  additional set of generating acyclic cofibrations is
  \[ \{ w^\prime \square\, i\} \] 
  where $i$ runs through a set $I_{\mc{D}}$ of generating cofibrations 
  of \mc{D} and $w$ runs through the set of weak equivalences
  in $\mc{C}$. The fact that hf fibrations are exactly those 
  objectwise fibrations with the right lifting property with respect to this 
  set follows from Lemma~\ref{hf fibrations}.
\end{proof}

\begin{remark}
It is now clear that the (cr) hf model structure on $\Fun(\mc{C},\mc{D})$ can also be viewed as the left Bousfield localization of the (cr) projective model structure with respect to the set 
  \[\{R^B\to R^A\,|\, A\to B \text{ is a weak equivalence in } \mc{C}\}.\]
\end{remark}

\subsection{Homotopy functors in several variables}\label{sec:homot-funct-sever}

Recall that in the product of $\mc{S}$-model categories
$\mc{B}_1\times\dotsm\times\mc{B}_n$ a morphism 
$f\colon \underline{K}\to \underline{L}$ is a {\em weak equivalence\/} 
(or {\em fibration\/}, or {\em cofibration\/}) if 
every component $K_i\to L_i$ is 
so in $\mc{C}_i$. This defines an \mc{S}-model structure.
In particular, 
Definition \ref{def-homotopy-functor} 
is applicable in 
$\Fun(\mc{C}_1\times\dotsm\times\mc{C}_n,\mc{D})$, where
$\mc{C}_k\subset \mc{B}_k$ is a full subcategory for every $k\in \ul{n}$. 
Note that a functor in the product category is a homotopy functor if
and only if all its partial functors are homotopy functors.
The results of Section~\ref{sec:simp-homot-funct} apply.
%However, only left Bousfield localizations of the projective model structure
%will occur for functors of several variables.

\begin{corollary}\label{cor:homot-funct-sever}
  Assume that for all $ k\in \ul{n}$ the categories 
  $\mc{C}_i\subset\mc{B}_i$ and $\mc{D}$ satisfy 
  Convention~\emph{\ref{conv:hf}}. Then the homotopy functor model structure 
  obtained from the projective model structure on the category 
  $\Fun(\mc{C}_1\times\dotsm\times\mc{C}_n,\mc{D})$ exists.
  Furthermore, it is a right proper cofibrantly generated \mc{S}-model 
  structure, and it is left proper if \mc{D} is left proper.
\end{corollary}

We want to obtain an hf model structure on the category $\Fun(\mc{C}_1\wedge\dotsm\wedge\mc{C}_n,\mc{D})$. Unlike the corresponding cartesian product category, the underlying category of $\mc{B}_1\wedge\dotsm\wedge\mc{B}_n$ is usually neither cocomplete, nor complete (Definition~\ref{def:functor-p} addresses their relation). In particular, we cannot directly apply the results from section \ref{sec:simp-homot-funct}. We first define homotopy functors.

\begin{definition}
A functor \dgrm{X} in $\Fun(\mc{C}_1\wedge\dotsm\wedge\mc{C}_n,\mc{D})$ is called a {\it homotopy functor} if for any object $(K_1,\dotsc,\wh{K}_i,\dotsc,K_n)$ in $\mc{C}_1\wedge\dotsm\wedge\wh{\mc{C}}_i\wedge\dotsm\wedge\mc{C}_n$ the associated partial functor
   $$ \dgrm{X}_{(K_1,\dotsc,\wh{K}_i,\dotsc,K_n)}\co\mc{C}_i\to\mc{D}$$
is a homotopy functor. The hat indicates that the corresponding entry is left out. 
\end{definition}

A coaugmented \mc{S}-functor $\fibr$ from $\mc{B}_1\wedge\dotsm\wedge\mc{B}_n$ to itself is defined by
   \[ \fibr(\ul{K}):=\bigl(\fibr_{\mc{B}_1}(K_1),\dotsc,\fibr_{\mc{B}_n}(K_n)\bigr), \]
using the decent fibrant replacement in each category $\mc{B}_i$.
Analogous to (\ref{eq:hf}), a coaugmented \mc{S}-functor 
  \[ (\free)^\hf\co\Fun(\mc{C}_1\wedge\dotsm\wedge\mc{C}_n,\mc{D})\to\Fun(\mc{C}_1\wedge\dotsm\wedge\mc{C}_n,\mc{D})\] 
is defined by
\begin{align*}
   \phi_{\dgrm{X}}\co\dgrm{X}\to(\dgrm{X})^\hf(\ul{K}):=\bigl((i^{\wedge n})_{*}\dgrm{X}\bigr)\circ\fibr\circ i^{\wedge n}
\end{align*}
where $i\co\mc{C}_1\wedge\dotsm\wedge\mc{C}_n\to\mc{B}_1\wedge\dotsm\wedge\mc{B}_n$ is the inclusion. These enriched functors are well defined by proposition \ref{EiKel2}.

\begin{definition}
A functor \dgrm{X} in $\Fun(\Sigma_n\wr\mc{C}^{\wedge n},\mc{D})$ is called a {\it homotopy functor} if $\varepsilon^*\dgrm{X}$ is a homotopy functor, where $\varepsilon\co\mc{C}^{\wedge n}\to\Sigma_n\wr\mc{C}^{\wedge n}$ was defined in \ref{varepsilon}.
\end{definition}

We observe that the previous constructions extend to the wreath product category. The inclusion $\mc{C}\to\mc{B}$ induces a symmetric inclusion $\Sigma_n\wr\mc{C}^{\wedge n}\to\Sigma_n\wr\mc{B}^{\wedge n}$ and a decent fibrant replacement functor of \mc{B} extends to a symmetric functor
    $$\fibr=(\fibr_{\mc{B}},\dotsc,\fibr_{\mc{B}})\co\Sigma_n\wr\mc{B}^{\wedge n}\to\Sigma_n\wr\mc{B}^{\wedge n},$$
which is a decent fibrant replacement functor for $\Sigma_n\wr\mc{B}^{\wedge n}$.
There is then a coaugmented \mc{S}-functor
  $$(\free)^\hf\co\Fun(\Sigma_n\wr\mc{C}^{\wedge n},\mc{D})\to\Fun(\Sigma_n\wr\mc{C}^{\wedge n},\mc{D})$$ 
as above and an associated \mc{S}-natural transformation $\phi_{\dgrm{X}}$.
In order to treat functors out of the smash product category and the wreath product category simultaneously,
the construction $(-)^{\hf}$ has to be interpreted appropriately
in the following statements.

\begin{definition}\label{def:multi-variable-hf}
  A map $f\co \dgrm{X}\to \dgrm{Y}$ in $\Fun(\mc{C}_1\wedge\dotsm\wedge\mc{C}_n,\mc{D})$ or $\Fun(\Sigma_n\wr\mc{C}^{\wedge n},\mc{D})$ is called
  \begin{enumerate}
  \item 
    an {\em hf equivalence\/} if the map $f^\hf\co \dgrm{X}^\hf \to \dgrm{Y}^\hf$ is an objectwise weak equivalence, and
  \item
    an {\em hf fibration\/} if it is an objectwise fibration such that the square
    \diagr{ \dgrm{X} \ar[r]\ar[d]_-{f} & \dgrm{X}^\hf \ar[d]^-{f^\hf} \\ 
      \dgrm{Y} \ar[r] & \dgrm{Y}^\hf  }
    is a homotopy pullback square in the objectwise model structure.
  \end{enumerate}
  The hf cofibrations are the projective cofibrations. 
\end{definition}

\begin{remark}
For any \dgrm{X} in $\Fun(\mc{C}_1\wedge\dotsm\wedge\mc{C}_n,\mc{D})$ or $\Fun(\Sigma_n\wr\mc{C}^{\wedge n},\mc{D})$ the functor $\dgrm{X}^\hf$ is a homotopy functor and analogues of Remark \ref{remarks about hf} hold.
\end{remark}

\begin{theorem}\label{hf-model-str-on-several-var}
  Assume Convention~\emph{\ref{conv:hf}}. The classes given in
  Definition~\emph{\ref{def:multi-variable-hf}} constitute 
  a right proper cofibrantly generated \mc{S}-model structure
  on the categories $\Fun(\mc{C}_1\wedge\dotsm\wedge\mc{C}_n,\mc{D})$
  and $\Fun(\Sigma_n\wr\mc{C}^{\wedge n},\mc{D})$, respectively. 
  It is left proper if \mc{D} is left proper.
\end{theorem}

\begin{proof}
  The arguments of the proof of Theorem~\ref{thm:homotopy-model} 
  showing the existence of the hf model structure 
  on $\Fun(\mc{C},\mc{D})$ apply
  componentwise.
\end{proof}

The model structure from Theorem~\ref{hf-model-str-on-several-var}
is called the {\it homotopy functor model structure} and is denoted 
by $\Fun(\mc{C}_1\wedge\dotsm\wedge\mc{C}_n,\mc{D})_{\hf}$ and
$\Fun(\Sigma_n\wr\mc{C}^{\wedge n},\mc{D})_{\hf}$, respectively.

\begin{proposition}%\label{}
  Assume Convention~\emph{\ref{conv:hf}}. 
  The adjoint pair
  \[ p_{*}\co\Fun(\mc{C}_1\times\dotsm\times\mc{C}_n,\mc{D})_{\hf}\,\rightleftarrows\,\Fun(\mc{C}_1\wedge\dotsm\wedge\mc{C}_n,\mc{D})_{\hf}:\!p^* \]
  is a Quillen pair of homotopy functor model structures. 
  The functor $p^*$ preserves and detects weak equivalences and fibrations. 
  Also, the adjoint pair
  \[ \varepsilon_{*}\co\Fun(\mc{C}^{\wedge n},\mc{D})_{\hf}\,\rightleftarrows\,\Fun(\Sigma_n\wr\mc{C}^{\wedge n},\mc{D})_{\hf}:\!\varepsilon^* \]
  is a Quillen pair of homotopy functor model structures. 
  The functor $\varepsilon^*$ preserves and detects weak equivalences and 
  fibrations.
\end{proposition}

\begin{proof}
  The pair $(p_*,p^*)$ is a Quillen pair for the respective projective 
  model structures by Lemma~\ref{lem:p-quillen}. 
  The same holds for the pair 
  $(\varepsilon_*,\varepsilon^*)$. The claims above then
  follow from the canonical natural isomorphisms 
  $p^*\circ(-)^\hf\cong(-)^\hf\circ p^*$ and 
  $\varepsilon^*\circ(-)^\hf\cong(-)^\hf\circ \varepsilon^*$. 
\end{proof}

\begin{lemma}\label{lem:cross-right-hf}
  Consider the homotopy functor 
  model structure on $\Fun(\mc{C},\mc{D})$ obtained from the cross
  effect  model structure. 
  Then the $n$-th cross effect 
  \[\cross_n\colon \Fun(\mc{C},\mc{D})_{\mathrm{hf}\text{-}\mathrm{cr}}\to 
  \Fun(\Sigma_n\wr\mc{C}^{\wedge n},\mc{D})_{\mathrm{hf}}\]
  is a right Quillen functor.
\end{lemma}

\begin{proof}
%%% ARC 2.6
  Proposition~\ref{prop:cross-right-obj} implies that the left adjoint
  $\Lcross_n$ of 
  $\cross_n$ preserves cofibrations. In order to prove that
  $\Lcross_n$ also preserves acyclic cofibrations in the homotopy
  functor model structure, it suffices
  to prove that 
  $\cross_n$ maps fibrant objects to fibrant objects in the
  respective homotopy functor model structures. To do so,
  let $\dgrm{X}\co \mc{C}\to \mc{D}$ be a cross effect
  fibrant homotopy functor.
  Then $\cross_n(\dgrm{X})$ is objectwise weakly equivalent
  to $\hocr_n(\dgrm{X})$ by Lemma~\ref{cr-fibrant}, 
  and the latter is a homotopy functor
  by construction. %%% ARC 2.6
\end{proof}

\section{Excisive functors}\label{sec:excisive-functors}

The goal of this section is to localize the
homotopy functor model structures
on the various functor categories further, such
that every functor is weakly equivalent to an $n$-excisive functor.
Fibrant replacement in such a model structure then serves as
$n$-excisive approximation.
Recall that a homotopy functor is {\it $n$-excisive} if it maps 
strongly homotopy cocartesian $(n+1)$-cubes to homotopy Cartesian ones.

\subsection{The excisive model structures}\label{sec:model-categ-vers}

Set $\ul{n}:=\{1,\hdots,n\}$, let $\mc{P}(\ul{n})$ be its power set, and 
let $\mc{P}_0(\ul{n}):=\mc{P}(\ul{n})-\{\emptyset\}$. 
\begin{definition}
  For an object $A$ in \mc{C}, let $CA$ be the simplicial cone over $A$. 
  This is the reduced or unreduced cone depending on whether 
  \mc{C} is a \mc{U}- or \mc{S}-category.
  For a finite set $U$, the {\it join} $A\join U$ is
  defined as
  \[ CA\sqcup_A\dotsm\sqcup_A CA, \]
  gluing $|U|$ many copies of $CA$ along their base $A$. 
\end{definition}

\begin{convention}\label{conv:n-exc}
  In addition to Convention~\ref{conv:hf}, suppose 
  that for any object $A$ in $\mc{C}$ and any finite set $U$ 
  the object $A\join U$ is also in \mc{C}.
\end{convention}

\begin{remark}
  There are other models for $A\join U$.
  In an ambient model category \mc{B}, the join $A\join U$
  is weakly equivalent to a homotopy colimit of the 
  asterisk-shaped 
  diagram given by $|U|$ copies of the map $A\to\ast$ out of a single 
  copy of $A$. 
  For instance, $A\join U$ is weakly equivalent to $|U|-1$  
  wedge summands of $\Sigma A$. Hence, the assumption that \mc{C} is closed 
  under suspensions and finite coproducts is an equally good convention. 
  Because \mc{C} is a full subcategory of \mc{B}, a reasonable sufficient 
  condition is to assume that \mc{C} is closed under finite pushouts 
  along cofibrations.
\end{remark}

The join is an enriched bifunctor and comes with a 
natural map $ A \to A\join U$ induced by the inclusion $\emptyset\subset U$. 
The $\mc{P}_0(\ul{n+1})$-diagram $U\mapsto R^{U\star A}$ of representable functors
yields the functor $\hocolim_{U\in \mc{P}_0(\ul{n+1})} R^{U\star A}$. Using 
repeated factorization by suitable simplicial mapping cylinder constructions 
supplies a cr cofibrant model, denoted as
\[ \xymatrix{\dgrm{A}_n \ar[r]^-{\simeq} & \hocolim\limits_{U\in \mc{P}_0(\ul{n+1})} R^{U\star A}}.\]
The induced natural transformation $\dgrm{A}_n\to R^A$ is factored 
via a simplicial mapping cylinder as a cr cofibration $\xi_{A,n}$, 
followed by a simplicial homotopy equivalence:
\[ \dgrm{A}_n \xrightarrow{\xi_{A,n}} 
\cyl(\xi_{A,n}) \xrightarrow{\simeq} R^A. \]

\begin{definition}\label{def:Pn}
  Goodwillie's construction $T_n$ \cite[p. 657]{Goo:calc3} on the category of 
  objectwise fibrant homotopy functors may be rewritten as
  \[ T_n\dgrm{X}(A) := {\bf hom}(\dgrm{A}_n,\dgrm{X}). \]
  Let $P_n$ be 
  the colimit of the following sequence:
  \[ \Id \to (-)^\hf\to T_n(-)^\hf \to T_n^2(-)^\hf\to\dotsm\to\colim_kT_n^k(-)^\hf\]
\end{definition}

Convention~\ref{conv:hf} implies that filtered colimits and 
filtered homotopy colimits are weakly equivalent in $\mc{D}$. Thus, 
Goodwillie's $n$-excisive approximation is 
weakly equivalent to $P_n$ as defined above.
The canonical inclusion $\ul{n}\to\ul{n+1}$  
induces a map $\mc{P}_0(\underline{n})\hookrightarrow \mc{P}_0(\underline{n+1})$ of posets. This induces natural transformations $T_n\to T_{n-1}$ 
and $q_n\co P_n\to P_{n-1}$ which commute with the coaugmentations 
from the identity. The map $q_n$ is constructed for categories
of spaces or spectra in \cite[p. 664]{Goo:calc3} 
and generalizes to this setup. 

\begin{lemma}\label{Tn and Pn commute with ho(co)lims}
The functor $T_n$ commutes up to natural weak equivalence with all homotopy limits. The functor $P_n$ commutes up to natural weak equivalence with finite homotopy limits. Both functors commute up to natural weak equivalence with filtered homotopy colimits.
\end{lemma}

\begin{proof}
  Since $\dgrm{A}_n$ is cr cofibrant, $T_n$ is a right Quillen functor
  and commutes with all homotopy limits. %%% ARC 2.9
  Part (4) of %%% ARC 2.9
  Convention~\ref{conv:hf} ensures that $T_n$ commutes
  with filtered colimits. It then follows from its definition
  that $P_n$ commutes with filtered colimits. Again Convention~\ref{conv:hf}
  implies that $P_n$ commutes at least with 
  finite homotopy limits.
\end{proof}

\begin{lemma}\label{lem:Pn-excisive}
  Let $\dgrm{X}$ be an objectwise fibrant homotopy functor. 
  \begin{enumerate}
  \item[\rm (1)]
    The functor $P_n\dgrm{X}$ is $n$-excisive. 
  \item[\rm (2)]
    The map $\dgrm{X}\to P_n\dgrm{X}$ is initial among maps in the objectwise homotopy category out of \dgrm{X} into an $n$-excisive functor. 
  \item[\rm (3)]
    Both maps $P_n(p_n\dgrm{X})$ and $p_n(P_n\dgrm{X})$ are objectwise weak equivalences and homotopic to each other.
  \end{enumerate}
\end{lemma}

\begin{proof}
  Part (1) follows by adapting Goodwillie's opaque \cite[pp. 662]{Goo:calc3} or Rezk's slightly less opaque proof \cite{Rezk:streamline}. Part (2) and (3) are as in \cite[pp. 661]{Goo:calc3}.
\end{proof}

\begin{definition}\label{def:n-exc-model}
  A map $\dgrm{X}\to\dgrm{Y}$ in $\Fun(\mc{C},\mc{D})$ is called 
   \begin{enumerate}
     \item
       an {\em $n$-excisive equivalence\/} if the induced map 
       $P_{n}\dgrm{X}\to P_{n}\dgrm{Y}$ is an objectwise weak equivalence.
     \item
       an {\em $n$-excisive fibration\/} if it is an hf fibration and the 
       following diagram
       \diagr{ \dgrm{X} \ar[r]\ar[d]_-f & P_{n}\dgrm{X} \ar[d]^-{P_n(f)} \\
	 \dgrm{Y} \ar[r] & P_{n}\dgrm{Y} }
       is a homotopy pullback square in the homotopy functor model structure.
   \end{enumerate}
   The $n$-excisive cofibrations are projective cofibrations. 
   Analogous definitions can be given starting from the cr model structure. 
\end{definition}

In the case of $\mc{D}=\mc{S}$ and $\mc{C}=\Sfin$, 
the following theorem was already obtained in~\cite{BCR:calc}.

\begin{theorem}\label{thm:n-exc-model-str-exists}
  Assume Convention~\emph{\ref{conv:n-exc}}. The classes described in
  Definition~\emph{\ref{def:n-exc-model}}
  form a right proper cofibrantly generated \mc{S}-model structure on 
  $\Fun(\mc{C},\mc{D})$, which is left proper if \mc{D} is left proper.
\end{theorem}

\begin{proof}
  It suffices to show that the coaugmented functor $P_n$ satisfies the 
  axioms (A1), (A2), and (A3) given in~\cite[9.2]{Bou:telescopic}. 
  The functors $T_n$ and $(-)^\hf$ preserve objectwise weak equivalences
  by construction and the proof of Theorem~\ref{thm:homotopy-model}, 
  respectively. 
  Filtered colimits in \mc{D} preserve weak equivalences by 
  part~(4) of Convention~\ref{conv:hf}, whence $P_n$ also preserves objectwise weak equivalences. This implies (A1).  
  Property (A2) is verified in Lemma~\ref{lem:Pn-excisive}(3). 
  Property (A3) follows directly from 
  Lemma~\ref{Tn and Pn commute with ho(co)lims}. 
  The remaining task is to add further generating acyclic cofibrations for 
  the $n$-excisive model structure.
  With $I_{\mc{D}}$ being a set of generating cofibrations of $\mc{D}$, 
  let 
  \[ J_n:= \bigl\{ \xi_{A,n}\, \square \, i\}_{A\in \mc{C},i\in I_{\mc{D}}}. \]
  An objectwise fibration 
  $\dgrm{X}\to \dgrm{Y}$ has the right lifting property with respect to the set $J_n$ if and only if the morphism of pointed simplicial sets
  \begin{equation}\label{eq:1} 
    \mc{S}_{\mc{D}}(\cyl(\xi_{A,n}) ,\dgrm{X})
    \to \mc{S}_{\mc{D}}(\dgrm{A}_n,\dgrm{X})\times_{\mc{S}_{\mc{D}}(\dgrm{A}_n,\dgrm{Y})} \mc{S}_{\mc{D}}(\cyl(\xi_{A,n}) ,\dgrm{Y}) 
\end{equation}
  has the right lifting property with respect to $I_{\mc{D}}$. Since the map in~(\ref{eq:1}) 
  is a fibration in $\mc{D}$ anyway, $\dgrm{X}\to \dgrm{Y}$ 
  has the right lifting property with respect to $J_n$ if and only if~(\ref{eq:1}) is
  a weak equivalence. 
  The simplicial homotopy equivalence
  \[ \cyl(\xi_{A,n}) \xrightarrow{\simeq} R^A\]
  induces a simplicial homotopy equivalence on $\mc{D}$-mapping objects. 
  As $\mc{S}_{\mc{D}}(-,\dgrm{Z})$ transforms homotopy colimits to homotopy limits for $\dgrm{Z}$ objectwise
  fibrant, and $\dgrm{X}\to \dgrm{Y}$ is an objectwise fibration, the map $\dgrm{X}\to \dgrm{Y}$
  has the right lifting property with respect to $J_n$ if and only if the square
  \[\xymatrix{
    \dgrm{X}(A) \ar[r] \ar[d] & \holim \dgrm{X}(A\join U) \ar[d] \\
    \dgrm{Y}(A) \ar[r]  & \holim \dgrm{Y}(A\join U) }\]
  is a homotopy pullback square. Definition~\ref{def:Pn} of $T_n$ implies
  that there is a natural zig-zag of weak equivalences
  connecting this commutative diagram to the commutative diagram
  \diagr{ \dgrm{X}(A) \ar[r]\ar[d]_{f(A)} & T_{n}\dgrm{X}(A) \ar[d]^{T_nf(A)} \\
    \dgrm{Y}(A) \ar[r] & T_n\dgrm{Y}(A) }
  which is a homotopy pullback square by Lemma~\ref{lem:n-exc-fib}.
  Thus adding $J_n$ to a 
  suitable set of generating acyclic cofibrations for the hf 
  model structure yields a set of generating acyclic cofibrations 
  for $n$-excisive fibrations; analogously for their cr versions.
\end{proof}

The model structures provided by Theorem~\ref{thm:n-exc-model-str-exists}
are the {\it $n$-excisive model structures}. They are denoted 
$\Fun(\mc{C},\mc{D})_{n\text{-}\mathrm{exc}}$ and $\Fun(\mc{C},\mc{D})_{n\text{-}\mathrm{exc}\text{-}\mathrm{cr}}$,
respectively. 
The following
statement is analogous to  Lemma~\ref{hf fibrations}.

\begin{lemma}\label{lem:n-exc-fib}
  A map $f\co\dgrm{X}\to\dgrm{Y}$ is an $n$-excisive fibration if and only if it is an hf fibration and the following diagram
  \diagr{ \dgrm{X} \ar[r]\ar[d]_f & T_{n}\dgrm{X} \ar[d]^{T_n(f)} \\
    \dgrm{Y} \ar[r] & T_n\dgrm{Y} }
  is a homotopy pullback square in the hf model structure.
\end{lemma}

\begin{proof}
  This is straightforward using the fact that in \mc{D} filtered
  colimits preserve homotopy pullbacks.
%  The proof is similar to that of Lemma~\ref{hf fibrations}.
\end{proof}

\begin{remark}%\label{}
  The proof of Theorem \ref{thm:n-exc-model-str-exists} shows that the $n$-excisive (cr) model structure is a left Bousfield localization of the hf (cr) model structure with respect to the set $\{\xi_{A,n}\,|\, A\in\mc{C}\}$ or equivalently the set 
  \[ \left\{ \hocolim\limits_{U\in \mc{P}_0(\ul{n+1})} R^{U\star A}\to R^A\,|\, A\in\mc{C}\right\}.\]
\end{remark}

\subsection{Excisive functors in several variables}

\begin{definition}\label{def:multi-excisive}
  A functor $\dgrm{X}\co\mc{C}_1\times\dotsm\times\mc{C}_n\to\mc{D}$ is 
  {\it multi-excisive} if for every object $(K_1,\dotsc,\wh{K}_i,\dotsc,K_n)$ 
  in $\mc{C}_1\times\dotsm\times\wh{\mc{C}}_i\times\dotsm\times\mc{C}_n$ 
  the associated partial functor
  \[ \dgrm{X}_{(K_1,\dotsc,\wh{K}_i,\dotsc,K_n)}\co\mc{C}_i\to\mc{D} \]
  is excisive. 
  A functor $\dgrm{X}$ is 
  {\it $(d_1,\dotsc,d_n)$-excisive} if, for every $1\leq i\leq n$,
  the associated partial functor 
  $\dgrm{X}_{(K_1,\dotsc,\wh{K}_i,\dotsc,K_n)}$ is $d_i$-excisive.
\end{definition}

The augmented functor $P_1$ from Definition \ref{def:Pn} can
be applied componentwise 
to functors in $\Fun(\mc{C}_1\times\dotsm\times\mc{C}_n,\mc{D})$.
More precisely, for each object \ul{K} in $\mc{C}_1\times\dotsm\times\mc{C}_n$ and each $1\le i< j\le n$, the diagram
\diagr{ \dgrm{X}(K_1,\dotsc,K_n)\ar[r]\ar[d] & (P_1^{\mc{C}_i}\dgrm{X}_{(K_1,\dotsc,\wh{K}_i,\dotsc,K_n)})(K_i) \ar[d] \\
        (P_1^{\mc{C}_j}\dgrm{X}_{(K_1,\dotsc,\wh{K}_j,\dotsc,K_n)})(K_j) \ar[r] & (P_1^{\mc{C}_j}P_1^{\mc{C}_i}\dgrm{X}_{(K_1,\dotsc,\wh{K}_i,\dotsc,\wh{K}_j,\dotsc,K_n)})(K_i,K_j) 
      }
in \mc{D} is natural in \ul{K} and \dgrm{X}. 
By definition, $P_1^{\mc{C}_i}$ and $P_1^{\mc{C}_j}$ commute with each other 
up to canonical isomorphism, and thus the order $i<j$ is purely cosmetic. 
These diagrams assemble into a \mc{U}-natural transformation
\begin{equation}\label{p_1...1,1}
     p_{1,\dotsc,1}(\dgrm{X})\co\dgrm{X}\to P_{1,\dotsc,1}\dgrm{X}:=P_1^{\mc{C}_1}\dotsm P_1^{\mc{C}_n}\dgrm{X} 
\end{equation}
of functors from $\mc{C}_1\times\dotsm\times\mc{C}_n$ by Propositions \ref{EilKel} and \ref{EiKel2}. If the functor \dgrm{X} is multireduced, then this natural transformation is in fact enriched over \mc{S}.

A functor \dgrm{X} is multi-excisive if and only if the natural transformation $p_{1,\dotsc,1}(\dgrm{X})$ is an objectwise weak equivalence.

\begin{definition}\label{def:excis-funct-sever}
  A map $f\co \dgrm{X}\to \dgrm{Y}$ in $\Fun(\mc{C}_1\times\dotsm\times\mc{C}_n,\mc{D})$ is
  \begin{enumerate}
  \item 
    a {\em multi-excisive equivalence\/} if the map $P_{1,\dotsc,1}(f)$ is an hf equivalence.
  \item
    a {\em multi-excisive fibration\/} if it is an {\em hf fibration\/} such that the square
    \diagr{ \dgrm{X} \ar[r]\ar[d]_-f & P_{1,\dotsc,1}\dgrm{X} \ar[d]^-{P_{1,\dotsc,1}(f)} \\ 
      \dgrm{Y} \ar[r] & P_{1,\dotsc,1}\dgrm{Y}  }
    is a homotopy pullback square in the hf model structure.
  \end{enumerate}
  The multi-excisive cofibrations are the projective cofibrations.
\end{definition}

\begin{theorem}\label{thm:multi-excisive-model-str-on-several-var}
The classes given in Definition~\ref{def:excis-funct-sever} are a right proper cofibrantly generated \mc{S}-model structure on $\Fun(\mc{C}_1\times\dotsm\times\mc{C}_n,\mc{D})$, which is left proper if \mc{D} is left proper.
\end{theorem}

\begin{proof}
  The fact that each functor $P_1^{\mc{C}_i}$ above satisfies properties 
  (A1), (A2) and (A3) by the proof of Theorem~\ref{thm:n-exc-model-str-exists} 
  implies that the composite functor $P_{1,\dotsc,1}$ satisfies them. 
  Generating acyclic cofibrations are constructed as in the proof
  of Theorem~\ref{thm:n-exc-model-str-exists}.
\end{proof}

The corresponding model structure on $\Fun(\mc{C}_1\times\dotsm\times\mc{C}_n,\mc{D})$ is called the {\it multi-excisive model structure} and is denoted
$\Fun(\mc{C}_1\times\dotsm\times\mc{C}_n,\mc{D})_{\mathrm{mlt}\text{-}\mathrm{exc}}$.

\subsection{Multilinear and symmetric multilinear functors}\label{sec:mult-symm-mult}

The natural transformation $p_{1,\dotsc,1}$ from (\ref{p_1...1,1}) 
for functors $\dgrm{X}\co\mc{C}_1\times\dotsm\times\mc{C}_n\to\mc{D}$
will be constructed differently
for functors in $\Fun(\mc{C}_1\wedge\dotsm\wedge\mc{C}_n,\mc{D})$ and $\Sigma_n\wr\mc{C}^{\wedge n}$.
Let $S^n$ denote the $n$-fold smash product of $S^1=\Delta^1/\partial \Delta^1$.
Adapting Convention~\ref{conv:n-exc} slightly, if necessary,
each of the categories $\mc{C}_k$ is closed under tensoring with $S^1$
as a sub-$\mc{S}$-category of the ambient model category $\mc{B}_k$.
In particular, every object $\ul{K}$ has
a functorial suspension 
$\ul{K}\wedge S^1=(K_1\wedge S^1,\dotsc,K_n\wedge S^1)$, allowing the
following observation.

\begin{definition}\label{T_n and P_n}
  For every object $\ul{K}$ in $\mc{C}_1\wedge\dotsm\wedge\mc{C}_n$ there is a natural map:
  \begin{align*}  
    \mc{S}_{\mc{C}_1\wedge\dotsm\wedge\mc{C}_n}(\ul{K}\wedge S^1,\ul{K}\wedge S^1)&
    = \bigwedge_{i=1}^n\mc{S}_{\mc{C}_i}(K_i\wedge S^1,K_{i}\wedge S^1)\bigr) \\
    & \cong\bigwedge_{i=1}^n\mc{S}\bigl(S^1,\mc{S}_{\mc{C}_i}(K_i,K_{i}\wedge S^1)\bigr) \\
    &\toh{f}\mc{S}\bigl(S^n,\bigwedge_{i=1}^n\mc{S}_{\mc{C}_i}(K_i,K_{i}\wedge S^1)\bigr) \\
    &= \mc{S}\bigl(S^n,\mc{S}_{\mc{C}_1\wedge\dotsm\wedge\mc{C}_n}(\ul{K},\ul{K}\wedge S^1)\bigr),
  \end{align*}
  The isomorphism uses that the ambient categories $\mc{B}_i$ are 
  tensored over \mc{S}. The map labelled $f$ is an $n$-fold version
  of the canonical map
  \[ \mc{S}(S^1,L)\wedge \mc{S}(S^1,M) \to \mc{S}(S^1\wedge S^1,L\wedge M). \]
  If \dgrm{X} is a functor in $\Fun(\mc{C}_1\wedge\dotsm\wedge\mc{C}_n,\mc{D})$, the composition above induces a map
  \[ S^n \to \mc{S}_{\mc{C}_1\wedge\dotsm\wedge\mc{C}_n}(\ul{K},\ul{K}\wedge S^1) \to 
  \mc{S}_{\mc{D}}\bigl(\dgrm{X}(\ul{K}),\dgrm{X}(\ul{K}\wedge S^1)\bigr). \]
  that has an adjoint map 
  \[ \dgrm{X}(\ul{K})\to \bigl(\dgrm{X}(\ul{K}\wedge S^1)\bigr)^{S^n} \]
  which is natural in $\ul{K}$. The result is a
  composite natural transformation
  \begin{equation}\label{eq:T_n}
    t_{\dgrm{X}}\colon \dgrm{X} \to T_{(1,\dotsc,1)}(\dgrm{X}) := 
    \Omega^{n}\bigl(\dgrm{X}(\free \wedge S^1)\bigr)
  \end{equation}
  of $\mc{S}$-functors. 
  Let $p_{(1,\dotsc,1)}(\dgrm{X})\colon \dgrm{X} \to P_{(1,\dotsc,1)}(\dgrm{X})$ denote the canonical map to the colimit
  \begin{equation}\label{eq:P_n}   
    \xymatrix{\dgrm{X}\to \dgrm{X}^\hf \ar[r]^-{t_{\dgrm{X}^\hf}} & T_{(1,\dotsc,1)}(\dgrm{X}^\hf) \ar[rr]^-{t_{T_{(1,\dotsc,1)}(\dgrm{X}^\hf)}} &&\dotsm\to\colim\limits_nT^n_{(1,\dotsc,1)}(\dgrm{X}^\hf).}
  \end{equation}
  One can check that $p^*\bigl(p_{(1,\dotsc,1)}(\dgrm{X})\bigr)$ is canonically weakly equivalent to the natural transformation obtained in (\ref{p_1...1,1}). Here, $p$ is the functor from Definition~\ref{def:functor-p}.
\end{definition}

\begin{remark}\label{rem:permutation-action}
  In order to repeat this for symmetric functors, observe
  that the $n$-fold smash product $S^n$ of the unit circle
  carries a natural $\Sigma_n$-action given by permutation.
\end{remark}

\begin{definition}\label{symmetric T_n and P_n}
  For every object $\ul{K}$ in $\Sigma_n\wr\mc{C}^{\wedge n}$ 
  there is a natural map
  \begin{align*}  
    \mc{S}_{\Sigma_n\wr\mc{C}^{\wedge n}}(\ul{K}\wedge S^1,\ul{K}\wedge S^1)&\cong\bigvee_{\sigma}\bigwedge_{i=1}^n\mc{S}\bigl(S^1,\mc{S}_{\mc{C}}(K_i,K_{\sigma^{-1}(i)}\wedge S^1)\bigr) \\
    &\toh{f}\bigvee_{\sigma}\mc{S}\bigl(S^n,\bigwedge_{i=1}^n\mc{S}_{\mc{C}}(K_i,K_{\sigma^{-1}(i)}\wedge S^1)\bigr) \\
    &\toh{g} \mc{S}\bigl(S^n,\mc{S}_{\Sigma_n\wr\mc{C}^{\wedge n}}(\ul{K},\ul{K}\wedge S^1)\bigr),
  \end{align*}
  where the first two maps are as in Definition~\ref{T_n and P_n},
  and the map $g$ is the natural map relating the compositions
  of coproduct and the functor $\mc{S}(S^n,-)$. 
  The induced map
  \[ S^n \to \mc{S}_{\Sigma_n\wr\mc{C}^{\wedge n}}(\ul{K},\ul{K}\wedge S^1) \to 
  \mc{S}_{\mc{D}}\bigl(\dgrm{X}(\ul{K}),\dgrm{X}(\ul{K}\wedge S^1)\bigr) \]
  for any functor \dgrm{X} in $\Fun(\Sigma_n\wr\mc{C}^{\wedge n},\mc{D})$ 
  has as adjoint the map 
  \[ \dgrm{X}(\ul{K})\to \bigl(\dgrm{X}(\ul{K}\wedge S^1)\bigr)^{S^n}, \]
  where the target has the correct $\Sigma_n$-action from
  Definition~\ref{symmetric (co-)tensors} and 
  Remark~\ref{rem:permutation-action}. 
  The resulting map of symmetric $\mc{S}$-functors
  is denoted
  \begin{equation}\label{eq:T_n-sym}
    t_{\dgrm{X}}\colon \dgrm{X} \to T_{(1,\dotsc,1)}(\dgrm{X}) = 
    \bigl(\dgrm{X}(\free \wedge S^1)\bigr)^{S^n}
  \end{equation}
  Let $p_{\dgrm{X}}\colon \dgrm{X} \to P_{(1,\dotsc,1)}(\dgrm{X})$ denote the canonical map to the colimit of the sequence
  \begin{equation}\label{eq:P_n-sym}   
    \xymatrix{\dgrm{X}\to\dgrm{X}^\hf \ar[r]^-{t_{\dgrm{X}^\hf}} &
      T_{(1,\dotsc,1)}(\dgrm{X}^\hf) \ar[rr]^-{t_{T_{(1,\dotsc,1)}(\dgrm{X}^\hf)}} &&\dotsm\to\colim\limits_nT^n(\dgrm{X}^\hf)}=: P_{(1,\dotsc,1)}\dgrm{X} 
  \end{equation}
\end{definition}

\begin{definition}\label{def:mult-symm}
  A functor $\dgrm{X}\co\mc{C}_1\wedge\dotsm\wedge\mc{C}_n\to\mc{D}$ 
  is called {\it multilinear} if it is multi-excisive in the sense 
  of Definition~\ref{def:multi-excisive}, that is,
  if all partial functors are excisive.
  A functor $\dgrm{X}\co\Sigma_n\wr\mc{C}^{\wedge n}\to\mc{D}$ is called 
  {\it multilinear} if $\varepsilon^*\dgrm{X}\co\mc{C}^{\wedge n}\to\mc{D}$ 
  is multilinear.
\end{definition}

In order to treat the categories $\Fun(\mc{C}_1\wedge\dotsm\wedge\mc{C}_n,\mc{D})$ and $\Fun(\Sigma_n\wr\mc{C}^{\wedge n},\mc{D})$ simultaneously, 
the notation in what follows has to be 
interpreted appropriately. 

\begin{remark}
  Let \dgrm{X} be a functor in $\Fun(\mc{C}_1\wedge\dotsm\wedge\mc{C}_n,\mc{D})$ or $\Fun(\Sigma_n\wr\mc{C}^{\wedge n},\mc{D})$. In particular,
  \dgrm{X} is multireduced. Hence, being multi-excisive
  is equivalent to being multilinear. In any case,
  the functor $P_{(1,\dotsc,1)}\dgrm{X}$ is componentwise excisive by 
  construction and the map $\dgrm{X}\to P_{(1,\dotsc,1)}\dgrm{X}$ is 
  initial among maps to multilinear functors in the objectwise 
  homotopy category.
\end{remark}

\begin{definition}\label{def:multilinear-model}
  A map $f\co \dgrm{X}\to \dgrm{Y}$ in $\Fun(\mc{C}_1\wedge\dotsm\wedge\mc{C}_n,\mc{D})$ or $\Fun(\Sigma_n\wr\mc{C}^{\wedge n},\mc{D})$ is
  \begin{enumerate}
  \item 
    a {\em multilinear equivalence\/} if the map $P_{1,\dotsc,1}(f)$ 
    is an hf equivalence, and 
  \item
    a {\em multilinear fibration\/} if it is an hf fibration such that 
    the square
    \diagr{ \dgrm{X} \ar[r]\ar[d]_-f & P_{1,\dotsc,1}\dgrm{X} \ar[d]^-{P_{1,\dotsc,1}(f)} \\ 
      \dgrm{Y} \ar[r] & P_{1,\dotsc,1}\dgrm{Y}  }
    is a homotopy pullback square in the hf model structure.
  \end{enumerate}
  The multilinear cofibrations are the projective ones.
\end{definition}

\begin{theorem}\label{thm:multilinear-model}
  Assume Convention~\emph{\ref{conv:hf}}. The classes described
  in Definition~\emph{\ref{def:multilinear-model}} constitute
  right proper cofibrantly generated \mc{S}-model structures
  on $\Fun(\Sigma_n\wr\mc{C}^{\wedge n},\mc{D})$ and 
  $\Fun(\mc{C}_1\wedge\dotsm\wedge\mc{C}_n,\mc{D})$, respectively. 
  They 
  are left proper if \mc{D} is left proper. The adjoint pairs
  \[ p_{*}\co\Fun(\mc{C}_1\times\dotsm\times\mc{C}_n,\mc{D})_{\mathrm{mlt}\text{-}\mathrm{exc}}\,\rightleftarrows\,\Fun(\mc{C}_1\wedge\dotsm\wedge\mc{C}_n,\mc{D})_{\mathrm{ml}}:\!p^* \]
  and
  \[ \varepsilon_{*}\co\Fun(\mc{C}^{\wedge n},\mc{D})_{\mathrm{mlt}\text{-}\mathrm{exc}}\,\rightleftarrows\,\Fun(\Sigma_n\wr\mc{C}^{\wedge n},\mc{D})_{\mathrm{ml}}:\!\varepsilon^* \]
  are Quillen pairs.
  The functors $p^*$ and $\varepsilon^*$ preserve and detect multilinear 
  equivalences and multilinear fibrations.
\end{theorem}

\begin{proof}
  The assertions about $p^*$ and $\varepsilon^*$ are straightforward 
  observations. Further, it is tedious but true that $\varepsilon^*$ 
  maps the symmetric version (\ref{eq:P_n-sym}) of the natural 
  transformation $p_{(1,\dotsc,1)}$ to the non-symmetric
  version (\ref{eq:P_n}), and 
  that $p^*$ maps the \mc{S}-enriched version (\ref{eq:P_n}) to the 
  unpointed version (\ref{p_1...1,1}). 
  The proof then proceeds as the proof of 
  Theorem \ref{thm:n-exc-model-str-exists}.
  In order to describe additional generating acyclic cofibrations,
  let $\ul{K}$ be an object in $\mc{C}$. By Lemma~\ref{lem:T_111-square} 
  below, the map $f$ is a multilinear fibration if and only if the diagram
  \diagram{\dgrm{X}(\ul{K}) \ar[r] \ar[d]& T_{(1,\dotsc,1)}(\dgrm{X})(\ul{K})\ar[d] \\
    \dgrm{Y}(\ul{K}) \ar[r]       & T_{(1,\dotsc,1)}(\dgrm{Y})(\ul{K})}{eq:square}
  is a homotopy pullback diagram in $\mc{D}$ for all $\ul{K}$. 
  The horizontal maps in this square are from~(\ref{eq:T_n}).
  Evaluated at $\ul{K}$, the upper horizontal
  map may be written as follows:
  \[ \cotensor{R^{\ul{K}}}{\dgrm{X}}\cong \dgrm{X}(\ul{K}) \to \bigl(\dgrm{X}(\ul{K}\wedge S^1)\bigr)^{S^n}\cong \cotensor{R^{\ul{K}\wedge S^1}\wedge S^{n}}{\dgrm{X}}\]
  By the $\mc{S}$-Yoneda lemma~\ref{twisted Yoneda}, 
  it is thus induced by a map
  \[ \tau_{\ul{K}}\co R^{\ul{K}\wedge S^1}\wedge S^{n} \to R^{\ul{K}}\]
  of simplicial functors. Hence, the square above
  is isomorphic to the following square:
  \diagr{\cotensor{R^{\ul{K}}}{\dgrm{X}} \ar[r] \ar[d] &
    \cotensor{\bigl(R^{\ul{K}\wedge S^1}\wedge S^{n}\bigr)}{ \dgrm{X}}\ar[d] \\
    \cotensor{R^{\ul{K}}}{\dgrm{Y}} \ar[r] &
    \cotensor{\bigl(R^{\ul{K}\wedge S^1}\wedge S^{n}\bigr)}{ \dgrm{Y}.}}
  Since $\tau_{\ul{K}}$ is a map between 
  projectively cofibrant objects, the simplicial mapping cylinder 
  yields a factorization
  as 
  \[ R^{\ul{K}\wedge S^1}\wedge S^{n}\xrightarrow{j(\ul{K})}
    \cyl(\tau_{\ul{K}}) \xrightarrow{q(\ul{K})}  R^{\ul{K}}\]
  where $j(\ul{K})$ is a projective cofibration, $q(\ul{K})$ 
  is a simplicial homotopy equivalence, and all objects in this 
  factorization are finitely presentable and projectively cofibrant. 
  The square above factors accordingly as follows:
  \diagr{\cotensor{R^{\ul{K}}}{\dgrm{X}} \ar[r] \ar[d] & 
    \cotensor{\cyl(\tau_{\ul{K}})}{\dgrm{X}} \ar[r] \ar[d] &
    \cotensor{\bigl(R^{\ul{K}\wedge S^1}\wedge S^{n}\bigr)}{ \dgrm{X}}\ar[d] \\
    \cotensor{R^{\ul{K}}}{\dgrm{Y}} \ar[r] & 
    \cotensor{\cyl(\tau_{\ul{K}})}{\dgrm{Y}} \ar[r] &
    \cotensor{\bigl(R^{\ul{K}\wedge S^1}\wedge S^{n}\bigr)}{ \dgrm{Y}}}
  Since the map $q(\ul{K})$ is a simplicial homotopy equivalence, 
  then so are the horizontal maps on the left hand square. 
  Since the map $j(\ul{K})$ is a projective cofibration and $f$ 
  is at least an objectwise fibration, the map
  \begin{equation}\label{eq:map} 
    \cotensor{\cyl(\tau_{\ul{K}})}{\dgrm{X}} \to 
    \cotensor{\cyl(\tau_{\ul{K}})}{\dgrm{Y}} 
    \x\limits_{\cotensor{(R^{\ul{K}\wedge S^1}\wedge S^{n})}{ \dgrm{Y}}}
    \cotensor{(R^{\ul{K}\wedge S^1}\wedge S^{n})}{ \dgrm{X}} 
  \end{equation}
  is a fibration in $\mc{D}$. Hence, this fibration is acyclic
  if and only if the square above is
  a homotopy pullback square.
  Because $\mc{D}$ is cofibrantly generated,
  the map in question is an acyclic fibration if and only if it
  has the right lifting property with respect to $I_\mc{D}$. By
  adjointness, the map~(\ref{eq:map}) is thus a weak equivalence
  if and only if $f$ has the right lifting property with respect to the
  set of maps $\{i\,\square\, j(\ul{K})\}_{i\in i_\mc{D}}$.
\end{proof}

The model structures provided by Theorem~\ref{thm:multilinear-model}
are referred to as the {\it  multilinear model structures}, and  
denoted $\Fun(\mc{C}_1\wedge\dotsm\wedge\mc{C}_n,\mc{D})_{\mathrm{ml}}$ 
and $\Fun(\Sigma_n\wr\mc{C}^{\wedge n},\mc{D})_{\mathrm{ml}}$, respectively.
The proof of Theorem~\ref{thm:multilinear-model} shows that they can be seen as left Bousfield localizations.

\begin{lemma}\label{lem:T_111-square}
%%%%%%%%%%%%%%%%%%%%%%%%%%%%%%%%%%%%%%%%%%%%%%%%%%%%%%%%%%%%%%%%%%
  A map $f\co\dgrm{X}\to\dgrm{Y}$ is a multilinear fibration if and only if it is an hf fibration and the following diagram
  \diagr{ \dgrm{X} \ar[r]\ar[d]_f & T_{(1,\dotsc,1)}\dgrm{X} \ar[d]^{T_{(1,\dotsc,1)}(f)} \\
    \dgrm{Y} \ar[r] & T_{(1,\dotsc,1)}\dgrm{Y} }
  is an objectwise homotopy pullback square.
\end{lemma}

\begin{proof}
  This is straightforward using the fact that in \mc{D} filtered
  colimits preserve homotopy pullbacks.
\end{proof}

\subsection{Coefficient spectra}\label{sec:coefficient-spectra}

The aim of this section is to connect the multilinear
model category of symmetric functors with the model 
category of spectra with an action of a symmetric group.
 
\begin{definition}
  Let $\Sp(\mc{D})$ denote the category of {\it Bousfield-Friedlander spectra} in the $\mc{S}$-model category $\mc{D}$ as defined by Schwede \cite{Schwede:cotangent}.
  An object in $\Sp(\mc{D})$ is a sequence $(E_0,E_1,\dotsc)$ of objects in $\mc{D}$, together with structure maps $\sigma_n^E\colon \Sigma E_n\to E_{n+1}$. A morphism of such objects is a sequence of morphisms in $\mc{D}$ commuting strictly with the structure maps. 
  A map $f\co E\to F$ in $\Sp(\mc{D})$ is called
  \begin{enumerate}
  \item 
    a {\em stable equivalence\/} of spectra in $\mc{D}$ if $QE\to QF$ is a levelwise equivalence where $Q$ is a certain model for ``$\Omega$-spectrum'' given in \cite[p. 90]{Schwede:cotangent}.
  \item
    a {\em projective cofibration} if the map $E_0\to F_0$ and for all $n\ge 1$ the maps
    $$ E_n\vee_{E_{n-1}}F_{n-1}\to F_n $$
    are cofibrations in \mc{D}.
  \end{enumerate} 
\end{definition}

\begin{theorem}[Schwede]\label{thm:schwede-stable}
  Suppose that $\mc{D}$ satisfies parts~\emph{(3)} and~\emph{(4)} of 
  Convention~\emph{\ref{conv:hf}}. Then 
  there is an $\mc{S}$-model structure on $\Sp(\mc{D})$ with stable equivalences
  as weak equivalences and projective cofibrations as cofibrations.
  It satisfies itself Convention~\emph{\ref{conv:hf}}.
  The $\mc{S}$-functor $\Ev_0\co\Sp(\mc{D})\to \mc{D},\, E\mapsto E_0$ 
  is a right Quillen functor.
\end{theorem}

%%% ARC 2.11
The proof can be found in \cite{Schwede:cotangent}, with the
modification that the assumption on properness appearing in
\cite[Prop.~2.1.5]{Schwede:cotangent} may be relaxed
to right properness by \cite{Bou:telescopic}. %%% ARC 2.11
The following ingredients of the proof are relevant later: If $\{i\}_{i\in I_{\mc{D}}}$ is a set of generating cofibrations for $\mc{D}$, then $\{\mathrm{Fr}_k(i)\}_{k\in \mathbbm{N}, i\in I_{\mc{D}}}$ is a set of generating cofibrations for $\Sp(\mc{D})$. Here the functor $\mathrm{Fr}_k\co\mc{D}\to\Sp(\mc{D})$ is left adjoint to evaluating at the $k$-th level, explicitly
   $$ \bigl(\mathrm{Fr}_k(D)\bigr)_{\ell}=\left\{\begin{array}{cl} 
                                     \ast,&\text{ for }0\le\ell<k \\
                                     \Sigma^{\ell-k}D,&\text{ for }\ell\ge k
                                           \end{array}\right.$$
The functor $\Ev_0$ commutes with all limits and colimits. 
It is worth mentioning that $\Fun\bigl(\mc{C},\Sp(\mc{D})\bigr)$
is canonically isomorphic to $\Sp\bigl(\Fun(\mc{C},\mc{D})\bigr)$.
%As it is a right Quillen functor, it commutes with homotopy limits. 

\begin{theorem}\label{thm:symm-ev-eq}
  Composing with $\mathrm{Fr}_0$ and $\Ev_0$ induces a Quillen equivalence: 
  \[ F\co\Fun(\Sigma_n\wr\mc{C}^n,\mc{D})_{\mathrm{ml}}\rightleftarrows\Fun(\Sigma_n\wr\mc{C}^n,\Sp(\mc{D}))_{\mathrm{ml}}\!:G \]
\end{theorem}

\begin{proof}
  Since the functor $\Ev_0\co\Sp(\mc{D})\to\mc{D}$ preserves objectwise fibrations and objectwise acyclic fibrations, the same is true for $G$. Hence, $G$ is a right Quillen functor for the projective model structures. 
  The functor $\Ev_0$ commutes with all limits, colimits, and is a right Quillen functor. 
  Hence $\Ev_0\colon \Sp(\mc{D}) \to \mc{D}$ commutes up to natural weak equivalence with $\dgrm{X}\to\dgrm{X}^\hf$ and $T_{(1,\dotsc,1)}$. In particular, the induced functor $G$ is a right Quillen functor on the homotopy functor and the multilinear model structures.
  A right Quillen functor is a Quillen equivalence
  if and only if its total right derived functor is an equivalence.
  The proof of \cite[Prop.~3.7]{Goo:calc3}, which states
  that $\Ev_0$ induces an equivalence on the
  (naive) homotopy categories of multilinear functors,
  extends to the setup here, which
  concludes the proof.
\end{proof}

\begin{corollary}\label{ml-stable}
  The multilinear model structure on symmetric functors is stable. 
\end{corollary}

A suitable evaluation functor connects
symmetric functors directly with spectra having
a symmetric group action. In order to describe it, recall
from Notation~\ref{not:functor-cat} that
the category $\Sp(\mc{D})^{\Sigma_n}$ is the category of 
functors $\Sigma_n\to \Sp(\mc{D})$, where $\Sigma_n$ is viewed as a 
category with one object. In other words, an object in $\Sp(\mc{D})^{\Sigma_n}$ 
is a spectrum with a right $\Sigma_n$-action. Since the stable model 
structure on $\Sp(\mc{D})$ is cofibrantly generated, the category 
$\Sp(\mc{D})^{\Sigma_n}$ carries a cofibrantly generated model structure 
with fibrations and weak equivalences defined on underlying spectra. 
This is sometimes called the model structure for ``naive $\Sigma_n$-spectra''.

\begin{definition}\label{def:eval-spectra}
  Precomposition with the symmetric diagonal \mc{S}-functor 
  \[\Delta_n\co\Sigma_n\times\mc{C}\to\Sigma_n\wr\mc{C}^{\wedge n}\]
  as introduced in the proof of Lemma~\ref{Lcross formula}, defines a functor 
  \[ \Delta_n^*\co\Fun(\Sigma_n\wr\mc{C}^{\wedge n},\Sp(\mc{D}))\to\Fun(\Sigma_n\times\mc{C},\Sp(\mc{D}))\cong\Fun(\mc{C},\Sp(\mc{D})^{\Sigma_n}). \]
  Evaluating at an object $C$ in \mc{C} induces the functor
  \[ \ev_C\co\Fun(\mc{C},\Sp(\mc{D})^{\Sigma_n})\to\Sp(\mc{D})^{\Sigma_n}. \]
  The composition is denoted
  \[ \Ev_C=\ev_C\circ\Delta_n^*\co\Fun(\Sigma_n\wr\mc{C}^{\wedge n},\Sp(\mc{D}))\to\Sp(\mc{D})^{\Sigma_n}. \]
  As a composition of two right adjoint functors, the functor $\Ev_C$ has 
  a left \mc{S}-adjoint denoted by 
  \[ \LEv_{C}\colon \Sp(\mc{D})^{\Sigma_n} \to 
  \Fun\bigl(\Sigma_n\wr(\mc{C})^{\wedge n},\Sp(\mc{D})\bigr). \]
\end{definition}

\begin{theorem}\label{coefficient-spectra-equivalence}
  Suppose that \mc{C} is the category $\Sfin$ of finite pointed simplicial 
  sets.
  The functor
  \[  \LEv_{S^0}\colon \Sp(\mc{D})^{\Sigma_n} \to 
           \Fun\bigl(\Sigma_n\wr(\Sfin)^{\wedge n},\Sp(\mc{D})\bigr)_{\mathrm{ml}} \]
  is a left Quillen equivalence.
\end{theorem}

\begin{proof}
  Choosing  $S^0\in \mc{C}=\Sfin$ yields the functor
  \[ \Ev_{\ul{S^0}}\colon \Fun\bigl(\Sigma_n\wr(\Sfin)^{\wedge n},\Sp(\mc{D})\bigr)\to\Sp(\mc{D})^{\Sigma_n}. \]
  Explicitly, it is given by
  $ \dgrm{X}\mapsto\dgrm{X}(S^0,\dotsc,S^0)$
  with $\Sigma_n$-action induced by permuting 
  the $n$-tuple $(S^0,\dotsc,S^0)$. Its left adjoint $\LEv_{\ul{S^0}}$
  sends a $\Sigma_n$-spectrum $E$ to the symmetric functor
  \[ \xymatrix{\ul{K}=(K_1,\dotsc,K_n) \ar@{|->}[r] &\LEv_{\ul{S^0}}(\ul{K})= E\wedge K_1 \wedge \dotsm \wedge K_n}\]
  having the following effect on morphism spaces:
  \[ \xymatrix{
    \bigvee_{\sigma\in \Sigma_n} \bigwedge_{i=1}^n\mc{S}(K_i,L_{\sigma^{-1}(i)}) \ar[r] &
    \mc{S}(E\wedge K_1 \wedge \dotsm \wedge K_n,E\wedge L_1 \wedge \dotsm \wedge L_n)\\
    (\sigma,\ul{f}= (f_1,\dotsc,f_n)) \ar@{|->}[r] & \sigma_E \wedge \bigl(\sigma_\ast
    \circ (f_1\wedge \dotsm \wedge f_n)\bigr)} \]
  Here $\sigma_\ast$ denotes the permutation
  $ L_{\sigma^{-1}(1)}\wedge \dotsm \wedge L_{\sigma^{-1}(n)} \to
  L_1 \wedge \dotsm \wedge L_n $
  induced by $\sigma$. 

  The unit $E\to \Ev_{S^0}(\LEv_{S^0}(E))$ is the canonical isomorphism
  identifying $E$ with $E\wedge S^0\wedge \dotsm \wedge S^0$.
  The counit $\LEv_{S^0}(\Ev_{S^0}(\dgrm{X}))\to \dgrm{X}$ is the
  natural transformation
  \[ \dgrm{X}(S^0,\dotsc,S^0)\wedge K_1\wedge \dotsm \wedge K_n \to 
  \dgrm{X}(K_1, \dotsc, K_n) \] 
  which is a special case
  of the assembly map
  \[ \dgrm{X}(L_1,\dotsc,L_n)\wedge K_1\wedge \dotsm \wedge K_n \to 
  \dgrm{X}\bigl((K_1\wedge L_1), \dotsc ,(K_n\wedge L_n)\bigr). \] 
  The latter is adjoint to the natural map
  \[ \xymatrix{ K_1\wedge \dotsm \wedge K_n \ar[d] \\ \{\id\}\times  
    \mc{S}(L_1,K_1\wedge L_1)\wedge \dotsm \wedge \mc{S}(L_n,K_n\wedge L_n) \ar[d] \\
       {\Sigma_n\wr(\Sfin)^{\wedge n}}\bigl((L_1, \dotsc, L_n),(K_1\wedge L_1,\dotsc,K_n\wedge L_n)\bigr) \ar[d]\\
       {\Sp}\bigl(\dgrm{X}(\ul{L}),\dgrm{X}(K_1\wedge L_1,\dotsc,K_n\wedge L_n)\bigr).} \]
  Since $\Delta^\ast$ and $\ev_{S^0}$ are right Quillen functors for 
  projective model structures on functor categories, so is
  their composition. Hence, $\LEv_{S^0}$ is a left Quillen functor
  to the projective model structure, and to the multilinear
  model structure as well.

To show that the derived unit $E \to \Ev_{S^0}\bigl(P_{(1,\dotsc,1)}(\LEv_{S^0}(E))\bigr)$ is a weak equivalence for $E$ cofibrant, recall that the unit is an isomorphism. Further, the functor $\LEv_{S^0}(E)$ preserves weak equivalences and the canonical map
\[ E\wedge S^k\wedge \dotsm \wedge S^k \to \Omega^n\bigl( (E\wedge S^{k+1}
            \wedge \dotsm \wedge S^{k+1})^\fib\bigr) \]
is a weak equivalence in the stable model structure, where $(\free)^\fib$ denotes fibrant replacement in $\Sp^{\Sigma_n}$. It follows that $E \to P_{(1,\dotsc,1)}\bigl(\LEv_{S^0}(E)\bigr)(\ul{S}^0)$ is a weak equivalence.
It remains to prove that $\Ev_{{S^0}}$ detects weak equivalences 
of multilinear functors. As in the
proof of \cite[Prop.~5.8]{Goo:calc3}, the symmetry is irrelevant,
and the case $n=1$ is sufficient. If $f\co \dgrm{X}\to \dgrm{Y}$
is a map of linear functors with $f(S^0)$ a weak equivalence,
then $f(S^k)$ is a weak equivalence for every $k$, as one
deduces from the
natural weak equivalence~(\ref{eq:T_n}). It then follows
that $f(K)$ is a weak equivalence for every $K\in \Sfin$ by
induction on the cells in $K$, using that \dgrm{X} and
\dgrm{Y} are linear.
\end{proof}

For $\dgrm{X}\co\Sigma_n\wr(\Sfin)^{\wedge n}
\to \Sp(\mc{D})$, the 
$\Sigma_n$-spectrum  
$\Ev_{{S^0}}(\dgrm{X})=\dgrm{X}(S^0,\dotsc,S^0)$ in $\mc{D}$
is called the {\em coefficient spectrum of\/} \dgrm{X}. 
It has the correct homotopy type if $\dgrm{X}$ is
multilinear. Given a
functor $\dgrm{Y}\co\Sfin\to\mc{S}$, Goodwillie calls the coefficient 
spectrum of the multilinear functor 
$\hocr_nP_n\dgrm{Y}\simeq\hocr_nD_n\dgrm{Y}$ 
the $n$-{\em th derivative\/} of \dgrm{Y}.

\subsection{Goodwillie's theorem on multilinearized homotopy cross effects}

\begin{proposition}[Prop. 3.3 \cite{Goo:calc3}]\label{prop:Goo:3.3}
  Let $0\le m\le n$. For any $n$-excisive functor 
  \dgrm{X}, the 
  functor $\hocr_{m+1}\dgrm{X}$ is $(n-m)$-excisive in each variable. 
  In particular, the $n$-th homotopy cross effect is multilinear if 
  \dgrm{X} is $n$-excisive, and it is contractible if \dgrm{X} is 
  $(n-1)$-excisive.
\end{proposition}

\begin{proof}
  The proof by Goodwillie is
  again a variation on opaqueness and applies to the setup here.
\end{proof}

\begin{definition}\label{def:reduced-homogeneous}
  A functor \dgrm{X} is {\it $n$-reduced} if $P_{n-1}\dgrm{X}\simeq\ast$,
  and {\it $n$-homogeneous} if it is $n$-excisive and 
  $n$-reduced.
\end{definition}

\begin{definition}
  In order to distinguish a functor 
  $\dgrm{X}\co\mc{C}_1\times\dotsm\times\mc{C}_n\to\mc{D}$ in $n$ 
  variables $K_1,\dotsc,K_n$ notationally from the same functor \dgrm{X} 
  when viewed as a functor in one variable $\ul{K}=(K_1,\dotsc,K_n)$,
  the latter is denoted $\lambda\dgrm{X}$.
\end{definition}

The $n$-excisive approximation functor $P_n$ applies
to the functor $\lambda\dgrm{X}$. 
There is a commutative diagram:
\diagram{ \lambda\dgrm{X} \ar[d]\ar[r] & P_n\lambda \dgrm{X} \ar[d]^-{\beta} \\
          \lambda P_{1,\dotsc,1}\dgrm{X} \ar[r]_-{\alpha} & P_n\lambda(P_{1,\dotsc,1}\dgrm{X}) }{abcd}

\begin{lemma}\label{d1...dn}
  If a functor $\dgrm{X}\co\mc{C}_1\times\dotsm\times\mc{C}_n\to\mc{D}$ 
  is $(d_1,\dotsc,d_n)$-excisive, then 
  $\lambda \dgrm{X}$ is $(d_1+\dotsm+d_n)$-excisive. 
\end{lemma}

\begin{proof}
  This is Goodwillie's Lemma 6.6 in \cite{Goo:calc3} whose proof refers to \cite[Lemma 3.4]{Goo:calc2}. The proof applies here.
\end{proof}

\begin{lemma}
  If a functor $\dgrm{X}\co\mc{C}_1\times\dotsm\times\mc{C}_n\to\mc{D}$ is multireduced, $\lambda \dgrm{X}$ is $n$-reduced. 
\end{lemma}

\begin{proof}
  The proofs of Lemma 3.2 and Lemma 6.7 in \cite{Goo:calc3} apply.
\end{proof}

\begin{lemma}\label{n-exc+multired=multilin}
  If $\dgrm{X}\co\mc{C}_1\times\dotsm\times\mc{C}_n\to\mc{D}$
  is multireduced and $\lambda \dgrm{X}$ is $n$-excisive,
  then $\dgrm{X}$ is multilinear. 
\end{lemma}

\begin{proof}
  The proof of \cite[Lemma 6.9]{Goo:calc3} applies here as well.
\end{proof}

\begin{corollary}\label{alpha-beta}
  The maps $\alpha$ and $\beta$ in diagram \emph{(\ref{abcd})} are 
  objectwise weak equivalences for every 
  functor in $\Fun(\mc{C}_1\wedge\dotsm\wedge\mc{C}_n,\mc{D})$.
\end{corollary}

\begin{proof}
  In order to apply the Lemmata above, the functor in question
  is pulled back via 
  $p\co\mc{C}_1\times\dotsm\times\mc{C}_n\to\mc{C}_1\wedge\dotsm\wedge\mc{C}_n$.
  Then the map $\alpha$ is an objectwise weak equivalence by 
  Lemma \ref{d1...dn}. The analogous map 
  \[ \gamma \co P_n \lambda\dgrm{X} \to P_{1,\dotsc,1}(P_n\lambda \dgrm{X})\] 
  is an objectwise weak equivalence
  by Lemma \ref{n-exc+multired=multilin}. The maps $\beta$
  and $\gamma$ are related via a natural weak equivalence
  commuting the constructions $P_n$ and $P_{1,\dotsc,1}$, since
  homotopy limits, as well as joins, commute among themselves.
  Hence, $\beta$ is an objectwise weak equivalence as well. The proof 
  finishes by noting that the functor $p^*$ detects objectwise weak 
  equivalences, see Lemma~\ref{lem:p-quillen}.
\end{proof}

\begin{theorem}[Theorem 6.1 \cite{Goo:calc3}]\label{Goo:thm 6.1}
  For all \mc{S}-functors $\dgrm{X}\co\mc{C}\to\mc{D}$ there is an natural objectwise weak equivalence under $\hocr_n\dgrm{X}$:
  \[ \hocr_n P_n\dgrm{X}\simeq P_{(1,\dotsc,1)}\hocr_n\dgrm{X}.  \]
\end{theorem}

\begin{proof}
  Let $\dgrm{X}$ be a functor in $\Fun(\mc{C},\mc{D})$.
  Substituting $\hocr_n\dgrm{X}$ in diagram (\ref{abcd}) supplies 
  a natural zig-zag of objectwise weak equivalence
  \[ P_{1,\dotsc,1}\hocr_n\dgrm{X}\simeq P_n\lambda(\hocr_n\dgrm{X}) \]
  under $\hocr_n\dgrm{X}$ by Corollary \ref{alpha-beta}.
  In order to prove that the functors $P_n\lambda(\hocr_n\dgrm{X})$ and 
  $\hocr_nP_n\dgrm{X}$ are naturally weakly equivalent under $\hocr_n\dgrm{X}$, 
  denote by $J_U(\dgrm{X})$ the functor
  \[ K\mapsto J_U(\dgrm{X})(K)=\dgrm{X}(K\star U). \]
  for each finite set $U$. The join with $U$ commutes with coproducts in \mc{C}.
  Thus, for each $\ul{K}$ in $\mc{C}^{\wedge n}$ there is a weak equivalence
  \begin{align*} 
    J_U\lambda(\hocr_n\dgrm{X})(\ul{K})    &\simeq\hofib\left[ \dgrm{X}\bigl(\bigvee_{i=1}^n(K_i\star U)\bigr)\to\holim_{S\in P_0{\ul{n}}}\dgrm{X}\bigl(\bigvee_{i\notin S}(K_i\star U)\bigr) \right] \\
    &\simeq\hofib\left[ \dgrm{X}\bigl((\bigvee_{i=1}^nK_i)\star U\bigr)\to\holim_{S\in P_0{\ul{n}}}\dgrm{X}\bigl((\bigvee_{i\notin S}K_i)\star U\bigr) \right] \\
    &\simeq\hocr_nJ_U\dgrm{X}(\ul{K})
  \end{align*}
  and therefore an objectwise weak equivalence
  $ T_n\lambda(\hocr_n\dgrm{X})\simeq\hocr_nT_n\dgrm{X}$. It
  induces the desired objectwise weak equivalence
  $ P_n\lambda(\hocr_n\dgrm{X})\simeq\hocr_nP_n\dgrm{X}$.
\end{proof}

\begin{corollary}\label{prop:cross-right-exc}
  The $n$-th cross effect
  \[\cross_n\colon \Fun(\mc{C},\mc{D})_{n\text{-}\mathrm{exc}\text{-}\mathrm{cr}}\to 
  \Fun(\Sigma_n\wr \mc{C}^{\wedge n},\mc{D})_{\mathrm{ml}}\]
  is a right Quillen functor.
\end{corollary}

\begin{proof}
  The left adjoint preserves cofibrations by 
  Lemma~\ref{lem:cross-right-hf}. It suffices to show that 
  $\cross_n$ preserves fibrations. 
  Let $f\co\dgrm{X}\to\dgrm{Y}$ be an $n$-excisive cr fibration, hence
  an hf cr fibration such that the diagram
  \diagr{ \dgrm{X} \ar[r]\ar[d]_-f & P_n\dgrm{X} \ar[d]^-{P_nf} \\
    \dgrm{Y} \ar[r] & P_n\dgrm{Y} }
  is a homotopy pullback square. 
  Lemma \ref{lem:cross-right-hf} implies that
  the map $\cross_n f$ is an hf fibration. 
  It remains to check that the diagram
  \diagr{ \cross_n\dgrm{X} \ar[r]\ar[d]_-{\cross_nf} & P_{1,\dotsc,1}\cross_n\dgrm{X} \ar[d]^-{P_{1,\dotsc,1}\cross_nf} \\
    \cross_n\dgrm{Y} \ar[r] & P_{1,\dotsc,1}\cross_n\dgrm{Y} }
  is a homotopy pullback. This square is the front of a commutative cube,
  whose sides are induced by the natural map
  $\cross_n \dgrm{X} \to \hocr_n \dgrm{X}$.
  The back of the cube is
  the following diagram:
  \diagr{ \hocr_n \dgrm{X} \ar[r]\ar[d] & \hocr_n P_n\dgrm{X} \ar[d]\ar[r]^-{\simeq} & P_{(1,\dotsc,1)}\hocr_n\dgrm{X} \ar[d]
\\ 
        \hocr_n \dgrm{Y} \ar[r] & \hocr_n P_n\dgrm{Y} \ar[r]^-{\simeq} & P_{(1,\dotsc,1)}\hocr_n\dgrm{Y}  }
  The horizontal maps on the right are objectwise weak equivalences
  by Goodwillie's Theorem~\ref{Goo:thm 6.1}. The square on the
  left hand side is the image of a homotopy
  pullback square under $\hocr_n$. 
  Thus, $\cross_n f$ is a multilinear fibration, once the sides of the commutative cube are proven to be homotopy pullback squares as well. In fact,
  it suffices to check that the square
  \diagram{ \cross_n \dgrm{X} \ar[r] \ar[d] & \hocr_n \dgrm{X} \ar[d] \\
    \cross_n \dgrm{Y} \ar[r] & \hocr_n \dgrm{Y} }{stabil1}
  is a homotopy pullback square, because the opposite side of the cube is obtained by applying $P_{1,\hdots,1}$ and inherits the homotopy pullback property.
  Let $\dgrm{F}$ be the fiber of $f$. It is cr fibrant, thus by Lemma~\ref{cr-fibrant} the canonical map
  \[ \cross_n\dgrm{F}\to\hocr_n\dgrm{F} \]
  of vertical (homotopy) fibers in diagram~(\ref{stabil1}) 
  is an objectwise weak equivalence. 
  As the multilinear model structure is stable by Corollary~\ref{ml-stable},
  diagram (\ref{stabil1}) is a homotopy pullback square, which
  completes the proof.
\end{proof}

\section{Homogeneous functors}\label{sec:homogeneous-functors}

As recalled in Definition~\ref{def:reduced-homogeneous},
a functor $\dgrm{X}\co \mc{C}\to \mc{D}$
is $n$-homogeneous if it is $n$-excisive and
$P_{n-1}\dgrm{X}$ is contractible. In this section,
$\Fun(\mc{C},\mc{D})$
will be equipped via right Bousfield localization
with a model structure in
which every functor is weakly equivalent
to an $n$-homogeneous functor. As shown in 
Theorem~\ref{thm:cross-equiv-on-spectra},
this model category is 
Quillen equivalent to the multilinear model structure
on $\Fun(\Sigma_n\wr\mc{C}^{\wedge n},\mc{D})$.
Hence the $n$-homogeneous model category is
also Quillen equivalent, by Theorem~\ref{coefficient-spectra-equivalence},
to the model category of $\Sigma_n$-spectra in $\mc{D}$,
provided $\mc{C}$ is the category of finite pointed
simplicial sets. This yields a construction of
derivatives for functors in $\Fun(\Sfin,\mc{D})$ on the level
of model categories.

\begin{convention}\label{conv:n-homog}
  Suppose in addition to Convention~\ref{conv:n-exc} that \mc{D} admits a set of generating cofibrations with cofibrant domains.
\end{convention}

\subsection{The homogeneous model structure}
\label{sec:homog-model-struct}

\begin{definition}\label{def:homog-model-str}
  Consider the following set of objects of $\Fun(\mc{C},\mc{S})$:
  \[  \Lambda_n=\Bigl\lbrace\bigwedge_{i=1}^nR^{K_i}\,|\,K_1,\dotsc,K_n\in\ob(\mc{C})\Bigr\rbrace.\]
  More generally, let $I_{\mc{D}}$ be a set of generating cofibrations 
  with cofibrant domains in $\mc{D}$, which exists by 
  Convention~\ref{conv:n-homog}. Let $\cd(I_{\mc{D}})$ denote the set of 
  domains and codomains of all morphisms $i\in I$, and set 
  \[  \Lambda_{n,I_{\mc{D}}}= 
  \bigl\lbrace\bigl(R^{K_1}\wedge \dotsm \wedge R^{K_n}\bigr) \wedge  D\,|\,K_1,\dotsc,K_n\in\ob(\mc{C}), D\in \cd(I_{\mc{D}})\bigr\rbrace\]
  A map $f$ in $\Fun(\mc{C},\mc{D})_{n\text{-}\mathrm{exc}\text{-}\mathrm{cr}}$ is
  \begin{enumerate}
  \item
    an $n$-{\em homogeneous equivalence\/} 
    if it is a $\Lambda_{n,I_{\mc{D}}}$-colocal equivalence.
  \item
    an $n$-{\em homogeneous cofibration\/} if it has 
    the left lifting property with respect to all $n$-excisive 
    cr fibrations that are also $n$-homogeneous equivalences.
  \end{enumerate}
  The $n$-homogeneous fibrations are the $n$-excisive cr fibrations. 
\end{definition}

The notion of colocal equivalence is taken from Hirschhorn~\cite[3.1.4(b)]{Hir:loc}. 
Choosing a different set of generating cofibrations with cofibrant domains 
in \mc{D} yields the same classes, due to the following well-known lemma,
whence the choice of generating cofibrations of \mc{D} will be omitted 
from the notation.

\begin{lemma}\label{lem:gen-weq}
  Let $\mc{D}$ be an $\mc{S}$-model category and $f$ a morphism of fibrant objects. Suppose $\mc{D}$ admits a set $I_{\mc{D}}$ of generating cofibrations with cofibrant domains. The morphism $f$ is a weak equivalence if and only if for every domain and every codomain $D$ appearing in $\cd(I_{\mc{D}})$, the map ${\mc{D}}(D,f)$ is a weak equivalence of simplicial sets.
\end{lemma}

\begin{theorem}\label{thm:n-hom-model-exists}
Assume Convention \emph{\ref{conv:n-homog}}.
The classes described in Definition~\emph{\ref{def:homog-model-str}} form a right proper $\mc{S}$-model structure on $\Fun(\mc{C},\mc{D})$. 
\end{theorem}

\begin{proof} 
  This follows from \cite[Theorem~2.6]{Christensen-Isaksen}, which applies 
  to any cofibrantly generated right proper model category. In our case 
  this is the $n$-excisive model structure on $\Fun(\mc{C},\mc{D})$.
\end{proof}

The right Bousfield localization of the $n$-excisive cross effect
model structure on the category $\Fun(\mc{C},\mc{D})$ with respect 
to the set $\Lambda_{n,I_{\mc{D}}}$
is the {\it $n$-homogeneous model structure} and is denoted by 
$\Fun(\mc{C},\mc{D})_{n\text{-}\mathrm{hom}}$.
Theorem~\ref{thm:schwede-stable} implies that the $n$-homogeneous model 
structure on $\Fun(\mc{C},\Sp(\mc{D}))$ exists. 

\begin{definition}
The {\it $n$-homogeneous part} of a functor $\dgrm{X}$ is defined as
  \[ D_n\dgrm{X}:=\hofib[q_n\co P_n\dgrm{X}\to P_{n-1}\dgrm{X}]. \]
By construction, $D_n\dgrm{X}$ is indeed $n$-homogeneous.
\end{definition}

\begin{lemma}\label{hocr-D-to-hocr-P}
  For every functor \dgrm{X} the map $D_n\dgrm{X}\to P_n\dgrm{X}$ induces an objectwise equivalence
  $\hocr_nD_n\dgrm{X}\to \hocr_nP_n\dgrm{X}$.
\end{lemma}

\begin{proof}
  Proposition \ref{prop:Goo:3.3} implies that $\hocr_nP_{n-1}\dgrm{X}\simeq\ast$.
  The chain
  \begin{align*}
    \hocr_nD_n\dgrm{X} &\simeq\hocr_n\hofib[P_n\dgrm{X}\to P_{n-1}\dgrm{X}]\simeq \hofib[\hocr_nP_n\dgrm{X}\to \hocr_nP_{n-1}\dgrm{X}] \\
    & \simeq\hocr_nP_n\dgrm{X}
  \end{align*}
  of natural objectwise equivalences completes the proof.
\end{proof}

\begin{lemma}\label{lem:n-hom-we1}
  For $f$ in $\Fun(\mc{C},\mc{D})$ the following statements are equivalent:
  \begin{enumerate}
  \item[{\rm (1)}]
    The map $f$ is an $n$-homogeneous equivalence.
  \item[{\rm (2)}]
    The induced map $\hocr_nP_n(f)$ is an objectwise equivalence.
  \item[{\rm (3)}]
    The induced map $\hocr_nD_n(f)$ is an objectwise equivalence.
  \item[{\rm (4)}]
    The induced map $\hocr_n(f)$ is a multilinear equivalence.
  \end{enumerate}
\end{lemma}

\begin{proof}
  A map $f\co\dgrm{X}\to\dgrm{Y}$ is an $n$-homogeneous equivalence if and only if an appropriate map between cr fibrant approximations is one. Assume that \dgrm{X} and \dgrm{Y} are  cr fibrant.
  Lemma~\ref{lem:gen-weq} implies that $f$ is an $n$-homogeneous equivalence if and only if for every $D$ in $\cd(I_D)$ the induced map 
  \[ \mc{S}_{\Fun(\mc{C},\mc{D})}\left(\Bigl(\bigwedge_{i=1}^nR^{K_i}\Bigr)\wedge D,P_n(f)\right) \cong \mc{S}_{\mc{D}}\bigl(D,\cross_nP_n(f)(K_1,\dotsc,K_n)\bigr) \]
  is a weak equivalence. 
  Since \dgrm{X} is cr fibrant it follows by Lemma~\ref{cr-fibrant} and Theorem~\ref{Goo:thm 6.1} that
  \[\cross_nP_n\dgrm{X}\simeq\hocr_nP_n\dgrm{X}\simeq P_{1,\hdots,1}\hocr_n\dgrm{X} \]
  and similarly for \dgrm{Y}. This shows that assertion (1), (2) and (4) are equivalent to each other if one observes, that $\hocr_nP_n(f)$ is a multilinear equivalence between multilinear functors by Proposition~\ref{prop:Goo:3.3}. Statements (2) and (3) are equivalent by Lemma~\ref{hocr-D-to-hocr-P}.
\end{proof}

\begin{corollary}\label{n-hom-equ-to-n-hom}
  Every functor \dgrm{X} is $n$-homogeneously equivalent to $D_n\dgrm{X}$. 
\end{corollary}

\begin{proof}
  The map $\dgrm{X}\to P_n\dgrm{X}$ is an $n$-excisive equivalence, hence
  an $n$-homogeneous equivalence. 
  By Lemma~\ref{lem:n-hom-we1} and Lemma~\ref{hocr-D-to-hocr-P}, 
  the map $D_n\dgrm{X}\to P_n\dgrm{X}$ is an $n$-homogeneous equivalence.
\end{proof}

The next statement also holds for functors to an unstable category \mc{D},
but it will be shown later in \ref{n-hom-we3}, after some auxiliary
statements.
\begin{corollary}\label{n-hom-we2}
  A map $f$ in $\Fun(\mc{C},\Sp(\mc{D}))$ is an $n$-homogeneous equivalence 
  if and only if the induced map $D_n(f)$ is an objectwise equivalence.
\end{corollary}

\begin{proof}
  By Corollary~\ref{n-hom-equ-to-n-hom}, it remains
  to show that $D_n(f)$ is an 
  objectwise equivalence if and only if $\hocr_nD_n(f)$ is an objectwise 
  equivalence. This is the content of \cite[Prop.~3.4]{Goo:calc3} 
  where Goodwillie actually shows that a functor to $\Sp(\mc{D})$ is 
  $(n-1)$-excisive if it is $n$-excisive with contractible $n$-th homotopy 
  cross effect. In the proof, the following two properties are used: 
 \begin{enumerate}
  \item Every strongly homotopy cocartesian cube of cofibrant
    objects admits a weak equivalence 
    from a pushout cube \cite[Prop. 2.2]{Goo:calc2}. This holds in any 
    model category.
  \item If a map $\mc{X}\to\mc{Y}$ of $n$-cubes is homotopy Cartesian 
    as an $(n+1)$-cube and $\mc{Y}$ is homotopy Cartesian, then 
    \mc{X} is homotopy Cartesian. This holds in every model category.
\end{enumerate} 
  This implies the statement for stable model categories, because 
  in such a map is a weak equivalence if and only if its homotopy fiber 
  is contractible.
\end{proof}

\begin{theorem}\label{thm:cross-equiv-on-spectra}
  Assume Convention \emph{\ref{conv:n-homog}}. The functors
  \[ \Lcross_n\co\Fun(\Sigma_n\wr\mc{C}^{\wedge n},\Sp(\mc{D}))_{\mathrm{ml}}\rightleftarrows\Fun(\mc{C},\Sp(\mc{D}))_{n\text{-}\mathrm{hom}}\!:\cross_n \]
  form a Quillen equivalence.
\end{theorem}

\begin{proof}
  The $n$-th cross effect is a right Quillen functor from
  the $n$-excisive cr model structure by Corollary~\ref{prop:cross-right-exc}.
  In particular, $\cross_n$ preserves fibrations.
  Lemma~\ref{cr-fibrant} and Lemma~\ref{lem:n-hom-we1} show that $\cross_n$ preserves and detects 
  $n$-homogeneous equivalences on cr fibrant objects. Hence if 
  $p$ is an acyclic fibration in the $n$-homogeneous cr model
  structure with fiber $\dgrm{F}$, the map $\cross_n(p)$ is a fibration
  with contractible fiber $\cross_n(\dgrm{F})$. As the multilinear
  model structure is stable by Corollary~\ref{ml-stable},
  $\cross_n(p)$ is an acyclic fibration. Thus $\cross_n$ is a 
  right Quillen functor on the $n$-homogeneous model structure.
  The argument from \cite[pp. 678]{Goo:calc3} extends to show that
  the derived unit map $\dgrm{X}\to\hocr_n\Lcross_n\dgrm{X}$ is an equivalence.
  The already mentioned fact that $\cross_n$ detects 
  $n$-homogeneous equivalences on cr fibrant objects implies
  it is a Quillen equivalence.
\end{proof}

\subsection{Goodwillie's delooping theorem}

\begin{theorem}\label{thm:delooping}
  Suppose that $\mc{C}$ and $\mc{D}$ satisfy 
  Convention~\emph{\ref{conv:n-homog}}.
  The pair of adjoint functors obtained by composing with the 
  functors 
  \[\mathrm{Fr}_0\co\mc{D}\to\Sp(\mc{D})\ \text{ and }\ \Ev_0\co\Sp(\mc{D})\to\mc{D}\]
is a Quillen equivalence:
  \[ F\co \Fun(\mc{C},\mc{D})_{n\text{-}\mathrm{hom}}\rightleftarrows\Fun(\mc{C},\Sp(\mc{D}))_{n\text{-}\mathrm{hom}}\!:G\] 
\end{theorem}

\begin{proof}
  The proof is divided into several steps. 
  The pair is a Quillen adjunction by Lemma~\ref{deloop:step1}. 
  The total right derived functor of $G$ is faithful by 
  Lemma~\ref{deloop:step2}, and essentially surjective and
  full by Lemma~\ref{deloop:step4}.
\end{proof}

\begin{lemma}\label{deloop:step1}
The functors $(F,G)$ form a Quillen pair for the $n$-homogeneous model structures on both sides.
\end{lemma}

\begin{proof}
  The functor $F$ maps the generating sets $I^{\mathrm{cr}}$ and $J^{\mathrm{cr}}$ on the left hand side into the corresponding ones on the right hand side. 
  This implies that $F$ is a left Quillen functor on cr model structures.
  The functor $\Ev_0$ commutes with all colimits, limits and homotopy limits. 
  In particular, it commutes with the functors $(-)^\hf$ and $P_n$ up to natural weak equivalence. Thus, the characterization of fibrations in the hf model structure \ref{hf fibrations} and in the $n$-excisive model structure \ref{def:n-exc-model} yields that $G$ is a right Quillen functor on hf model structures and the $n$-excisive model structure. The acyclic fibrations agree in all these model structures. 

The fibrations in the $n$-homogeneous and the $n$-excisive model structure agree. Suppose that $f$ is an $n$-homogeneous acyclic fibration. 
Since $\Ev_0$ commutes with $P_n$ and with homotopy fibers, we have $\hocr_nP_nG(f) \simeq G(\hocr_nP_n f)$. It follows from Lemma~\ref{lem:n-hom-we1} that $G(f)$ is an $n$-homogeneous acyclic fibration, whence $G$ is a right Quillen functor also for the $n$-homogeneous model structure. 
Alternatively, the left adjoint $F$ preserves the set of objects that define the right Bousfield localization. 
\end{proof}

\begin{lemma}\label{deloop:step2}
  The functor $G$ preserves and detects $n$-homogeneous equivalences of bifibrant functors, and its total right derived functor is faithful on morphisms.
\end{lemma}

\begin{proof}
  General localization theory of model categories implies 
  that a map between $n$-homogeneously bifibrant functors is an 
  $n$-homogeneous equivalence if and only if it is an objectwise equivalence.
  The functor $\Ev_0$ preserves and detects weak equivalences of stably fibrant spectra. 
  Thus, the functor $G$ preserves and detects $n$-homogeneous equivalences of bifibrant functors.

  Moreover, the structure maps $\dgrm{X} \rightarrow \Omega \dgrm{X}$ of an 
  $n$-homogeneously bifibrant functor $\dgrm{X}\co \mc{C}\rightarrow \Sp(\mc{D})$ are objectwise weak equivalences.
  Hence, if $f$ and $g$ are maps of bifibrant objects in the $n$-homogeneous model structure on $\Sp(\mc{D})$-valued functors such that $f_0$ and $g_0$ are homotopic, then $f$ and $g$ are homotopic. In particular, the  total right derived functor of $G$ is faithful on morphisms.
\end{proof}

For the next lemma recall that all $\mc{S}$-functors are reduced.

\begin{lemma}\label{lem:delooping}
  Let $n>0$ and $\dgrm{X}$ a functor. Then there is a natural commutative diagram
  \[ \xymatrix{ & \dgrm{X} \ar[d]\ar[rd] & \\ K_n\dgrm{X} \ar[d] & \wh{P}_n\dgrm{X} \ar[r]^-\sim \ar[l] \ar[d] & P_n\dgrm{X} \ar[d] \\
    R_n\dgrm{X} & P_{n-1}^\prime \dgrm{X} \ar[r]^-\sim \ar[l] & P_{n-1}\dgrm{X}}\]
  in which the left hand square is a homotopy pullback square, the functor $K_n\dgrm{X}$ is contractible and the functor $R_n\dgrm{X}$ is $n$-homogeneous. 
\end{lemma}

\begin{proof}
  The diagram above is diagram (2.3) in \cite{Goo:calc3} and Goodwillie's 
  proof of its existence in Section 2 of that article applies. 
\end{proof}

\begin{lemma}\label{deloop:step3}
  Let $\dgrm{X}\co\mc{C}\to\mc{D}$ be an $n$-reduced functor.
  Then there exists a spectrum-valued functor 
  ${\mathbf R}_n\dgrm{X}\co \mc{C}\to \Sp(\mc{D})$ and a natural
  objectwise equivalence
  $G({\mathbf R}_n\dgrm{X})\simeq\dgrm{X}$. 
\end{lemma}

\begin{proof}
  If a functor $\dgrm{X}\co\mc{C}\to\mc{D}$ is $n$-reduced, not only
  the upper left 
  corner $K_n\dgrm{X}$, but also 
  the lower right corner $P_{n-1}\dgrm{X}$ of the 
  homotopy pullback square from Lemma~\ref{lem:delooping}
  is contractible. Let $U_n\dgrm{X}$ denote 
  the homotopy pullback of that square. The 
  natural objectwise weak equivalences
  \begin{equation}\label{eq:zig-zag-rn}
    \dgrm{X}  \xrightarrow{\simeq} \wh{P}_n \dgrm{X} \xrightarrow{\simeq} U_n\dgrm{X}
    \xleftarrow{\simeq} \Omega R_n \dgrm{X},
  \end{equation}
  do not form a direct map which can be iterated to obtain a 
  spectrum-valued functor ${\mathbf R}_n\dgrm{X}$. However, a trick
  by Goodwillie \cite{Goo:calc1} given at the end of the introduction
  works. For $j\ge 0$, let $(\mathbf{R}_n\dgrm{X})_j$ denote the 
  homotopy limit (not colimit) of the diagram
  \[ R_n^j\dgrm{X}\xrightarrow{\simeq} U(R_n^j\dgrm{X})\xleftarrow{\simeq} \Omega R_n^{j+1}\dgrm{X}\xrightarrow{\simeq} \Omega U(R_n^{j+1}\dgrm{X})\xleftarrow{\simeq} \Omega^2 R_n^{j+2}\dgrm{X}\xrightarrow{\simeq}\dotsm,\]
  starting with $R_n^0\dgrm{X}=\dgrm{X}$. Then 
  $(\mathbf{R}_n\dgrm{X})_j\xrightarrow{\simeq}R_n^j\dgrm{X}$ for all $j\geq 0$,
  and there are structure maps 
  $(\mathbf{R}_n\dgrm{X})_j\xrightarrow{\simeq}\Omega(\mathbf{R}_n\dgrm{X})_{j+1}$ defining a spectrum-valued functor $\mathbf{R}_n\dgrm{X}$ such that $G(\mathbf{R}_n\dgrm{X})\simeq\dgrm{X}$.
\end{proof}

\begin{lemma}\label{deloop:step4}
  The total right derived functor of $G$ is essentially surjective and full.
\end{lemma}

\begin{proof}
  By Lemma \ref{n-hom-equ-to-n-hom}, any functor in $\Fun(\mc{C},\mc{D})$ is $n$-homogeneously equivalent to an $n$-homogeneous functor \dgrm{X}. Lemma~\ref{deloop:step3} supplies a functor $\mathbf{R}_n\dgrm{X}$ with $G(\mathbf{R}_n\dgrm{X})\simeq\dgrm{X}$. The statement follows.
\end{proof}

\subsection{The Quillen equivalences}\label{sec:quillen-equivalences}

There is a commutative diagram of Quillen pairs: 
\diagram{ \Fun(\Sigma_n\wr\mc{C}^{\wedge n},\mc{D})_{\mathrm{ml}} \ar@<4pt>[rr]^-{\Lcross_n}\ar@<-4pt>[d]_-{F} & & \Fun(\mc{C},\mc{D})_{n\text{-}\mathrm{hom}} \ar@<4pt>[ll]^-{\cross_n} \ar@<-4pt>[d]_-{F} \\
   \Fun(\Sigma_n\wr(\mc{C})^{\wedge n},\Sp(\mc{D}))_{\mathrm{ml}} \ar@<4pt>[rr]^-{\Lcross_n}\ar@<-4pt>[u]_-{G} & & \Fun(\mc{C},\Sp(\mc{D}))_{n\text{-}\mathrm{hom}} \ar@<4pt>[ll]^-{\cross_n}\ar@<-4pt>[u]_-{G} }{equivalences}
The left vertical Quillen pair was shown to be a Quillen equivalence in Theorem \ref{thm:symm-ev-eq}, the lower horizontal one in Theorem \ref{thm:cross-equiv-on-spectra}, and the right vertical pair in Theorem \ref{thm:delooping}. The 2-out-of-3 property of Quillen equivalences yields the following statement.

\begin{corollary}\label{n-hom we}
  Suppose Convention \ref{conv:n-homog} holds. Then the pair
   $$ \Lcross_n\co\Fun(\Sigma_n\wr\mc{C}^{\wedge n},\mc{D})_{\mathrm{ml}} \rightleftarrows \Fun(\mc{C},\mc{D})_{n\text{-}\mathrm{hom}}\!:\cross_n $$
  is a Quillen equivalence.
\end{corollary}

The associated diagram of total derived functors on homotopy categories yields 
Goodwillie's diagram of equivalences of homotopy categories (as displayed in the introduction).
If $\mc{C}=\Sfin$, evaluation at $S^0$ prolongs this Quillen equivalence
to the category of $\Sigma_n$-spectra in \mc{D}, by Theorem~\ref{coefficient-spectra-equivalence}.

\begin{corollary}\label{n-hom stable}
The $n$-homogeneous model structure on $\Fun(\mc{C},\mc{D})$ is stable, if Convention \emph{\ref{conv:n-homog}} holds.
\end{corollary}

\subsection{More on the homogeneous model structures}

The following assertion was proved for $\Sp(\mc{D})$ as target category already in Corollary~\ref{n-hom-we2}. The validity of 
Convention~\ref{conv:n-homog} is assumed for the remainder
of this article.
 
\begin{lemma}\label{n-hom-we3}
  A map $f$ in 
  $\Fun(\mc{C},\mc{D})$ is an $n$-homogeneous equivalence if and only 
  if the induced map $D_n(f)$ is an objectwise weak equivalence.
\end{lemma}

\begin{proof}
By Lemma~\ref{lem:n-hom-we1} and Theorem~\ref{Goo:thm 6.1}, the map $f$ is an $n$-homogeneous equivalence if and only if $\hocr_nP_n(f)\simeq\hocr_nD_n(f)$ is an objectwise equivalence. This shows immediately that if $D_n(f)$ is an objectwise equivalence, then $f$ is an $n$-homogeneous equivalence. 

Conversely, let $f$ be an $n$-homogeneous equivalence. By Corollary~\ref{n-hom-equ-to-n-hom} it can be replaced by the induced $D_n(f)\co D_n\dgrm{X}\to D_n\dgrm{Y}$.
Lemma~\ref{deloop:step3} states that there exists $n$-homogeneous functors $\mathbf{R}_n\dgrm{X}$ and $\mathbf{R}_n\dgrm{Y}$ to $\Sp(\mc{D})$ and a map $g=\mathbf{R}_n(f)$ between them such that $G(g)=D_n(f)$. The functor $G$ detects $n$-homogeneous equivalences by Lemma~\ref{deloop:step2}. Thus, $g$ is an $n$-homogeneous equivalence. Corollary~\ref{n-hom-we2} shows that $g=D_n(g)$ is an objectwise equivalence of symmetric multilinear functors to $\Sp(\mc{D})$. Hence, $G(g)=D_n(f)$ is an objectwise weak equivalence.
\end{proof}

\begin{corollary}\label{cor:delooping}
  For any functor $\dgrm{X}\co\mc{C}\to\mc{D}$ and $n>0$ there exists a 
  functor $R_n\dgrm{X}\co\mc{C}\to\mc{D}$ and a zig-zag of natural weak equivalences
  $D_n\dgrm{X}\simeq \Omega R_n\dgrm{X}$.
\end{corollary}

\begin{proof}
  This follows from Lemma \ref{lem:delooping}.
\end{proof}

\begin{corollary}%\label{}
  Let $n>0$. A map $f\co\dgrm{X}\to\dgrm{Y}$ in $\Fun(\mc{C},\mc{D})$ induces 
  an objectwise weak equivalence $D_n(f)$ if and only if the diagram 
  \diagram{ P_n\dgrm{X} \ar[r]\ar[d] & P_{n-1}\dgrm{X} \ar[d] \\ P_{n}\dgrm{Y} \ar[r] & P_{n-1}\dgrm{Y} }{neues Pullback1}
  is an objectwise homotopy pullback square.
\end{corollary}

\begin{proof}
  If the diagram is a homotopy pullback square, then its homotopy fibers 
  are weakly equivalent via $D_n(f)$. For the converse note that $R_n(f)$ 
  is a weak equivalence if $D_n(f)$ is by Corollary~\ref{cor:delooping}. The 
  horizontal maps are part of homotopy fiber sequences
  \[ P_n\dgrm{X} \to P_{n-1}\dgrm{X} \to R_n\dgrm{X} \quad \mathrm{and} \quad
  P_n\dgrm{Y} \to P_{n-1}\dgrm{Y} \to R_n\dgrm{Y} \]
  by Lemma~\ref{lem:delooping}. So the square is a homotopy pullback.
\end{proof}

\begin{lemma}\label{n-hom fib}
  A map is an $n$-homogeneous acyclic fibration if and only if it is an $(n-1)$-excisive cr fibration and an $n$-homogeneous equivalence.
\end{lemma}

\begin{proof}
  By Definition~\ref{def:n-exc-model}, a map is an $(n-1)$-excisive 
  cr fibration if and only if it is a fibration $f\co\dgrm{X}\to\dgrm{Y}$ in 
  the hf cr model structure that induces the following objectwise homotopy 
  pullback diagram:
  \diagr{ \dgrm{X} \ar[r]\ar[d]_f & P_{n-1}\dgrm{X} \ar[d]^{P_{n-1}(f)} \\ \dgrm{Y} \ar[r] & P_{n-1}\dgrm{Y} }
  Then it follows that in the following diagram
  \diagram{ \dgrm{X} \ar[r]\ar[d]_f & P_n\dgrm{X} \ar[r]\ar[d]^{P_{n}(f)} & P_{n-1}\dgrm{X} \ar[d]^{P_{n-1}(f)} \\ \dgrm{Y} \ar[r] & P_n\dgrm{Y} \ar[r] & P_{n-1}\dgrm{Y} }{neues Pullback2}
  the outer square and the right hand square are homotopy pullbacks. 
  Therefore, the left hand square is a homotopy pullback, and $f$ is an 
  $n$-excisive cr fibration which is the same as an $n$-homogeneous fibration. 
  Because it is an $n$-homogeneous equivalence by assumption, it is an 
  $n$-homogeneous acyclic fibration.

  Suppose now that $f$ is an $n$-homogeneous acyclic fibration.
  An $n$-homogeneous fibration is the same as an 
  $n$-excisive cr fibration, hence the left hand square of diagram 
  (\ref{neues Pullback2}) is a homotopy pullback. Because 
  (\ref{neues Pullback1}) is also a homotopy pullback, the 
  combined outer square is so. Hence, $p$ is an $(n-1)$-excisive cr fibration. 
  It is an $n$-homogeneous equivalence by assumption.
\end{proof}

\begin{lemma}\label{n-hom fib 2}
  Let \dgrm{X} and \dgrm{Y} be $n$-excisively cr fibrant.
  Then a map $p\co\dgrm{X}\to\dgrm{Y}$ is an $n$-homogeneous acyclic fibration if and only if it is an $(n-1)$-excisive cr fibration.
\end{lemma}

\begin{proof}
  According to Lemma~\ref{n-hom fib}, it suffices to show that an $(n-1)$-excisive cr fibration between $n$-excisively cr fibrant objects is already an $n$-homogeneous equivalence. In this case the outer square of diagram \ref{neues Pullback2} is a homotopy pullback and the horizontal maps in the left hand square are weak equivalences. Thus, the square on the right hand side is a homotopy pullback square, and the induced map of fibers $D_n\dgrm{X}\to D_n\dgrm{Y}$ is a weak equivalence.
\end{proof}

\begin{lemma}\label{n-hom cof}
  A cr cofibration is an $n$-homogeneous cofibration if and only if it
  is an $(n-1)$-excisive equivalence.
\end{lemma}

\begin{proof}
  Right properness implies that a map is an acyclic cofibration in the $n$-homogeneous model structure if and only if it has the left lifting property with respect to all $n$-homogeneous fibrations between fibrant objects \cite[Prop.~13.2.1]{Hir:loc}. The fibrant objects are exactly the $n$-excisively cr fibrant functors. Thus, the stated equivalence follows from Lemma~\ref{n-hom fib 2}.
\end{proof}

\begin{corollary}\label{lem:cof-n-hom}
  A functor is cofibrant in the $n$-homogeneous model structure if and only if it is cr cofibrant and $n$-reduced. 
\end{corollary}

\begin{proof} 
  This follows from Lemma~\ref{n-hom cof}.
\end{proof}

In particular, one can view the $n$-homogeneous model structure
as the ``fiber'' model structure of the left Quillen functor
\[ \mathrm{Id} \co \Fun(\mc{C},\mc{D})_{n\text{-exc-cr}}\to\Fun(\mc{C},\mc{D})_{(n-1)\text{-exc-cr}}\]
of excisive model structures. 
%%% ARC 4.last
This supplies a concrete instance
where the homotopy fiber model category of \cite[Theorem~3.1]{Bergner:hofib}
exists.
%%% ARC 4.last

\end{document}